\documentclass[10pt]{article}
\usepackage[latin1]{inputenc}

\usepackage{enumitem}
\usepackage{amsthm}
\usepackage{amsmath,amssymb,wasysym,amsbsy}
\usepackage{bm}
\usepackage[dvips]{graphicx}
\usepackage{multicol,array,lscape}
\usepackage{nonfloat}
\usepackage[small,bf]{caption2}
\usepackage{breqn}
\usepackage[dvipsnames,svgnames,x11names]{xcolor}
\usepackage{pst-all}
\usepackage{pst-3dplot}
\usepackage{authblk}
\usepackage{tikz}
\usepackage{pgfplots}

\usetikzlibrary{arrows,decorations.pathmorphing,backgrounds,positioning,fit,petri}

\usetikzlibrary{arrows,shapes}
\usepackage{dot2texi}

\usetikzlibrary{positioning}

\usetikzlibrary{trees}

\usepackage{cancel} 
\usepackage{soul} 
\usepackage[%
  colorlinks=true,%
	linkcolor	=blue,%
	citecolor =blue,%
  urlcolor=blue%
]{hyperref}
\usepackage[left=3cm,right=3cm, top=3.3cm, bottom=3cm]{geometry}
\newtheorem{theo}{Theorem}[section]
\newtheorem{lemm}[theo]{Lemma}

\newtheorem{coro}[theo]{Corollary}
\newtheorem{prop}[theo]{Proposition}

\newtheorem{defi}[theo]{Definition}
\newtheorem{remark}[theo]{Remark}

\def\x        {\textit {\textbf{x}}}

\numberwithin{equation}{section}

\title{Theoretical and numerical analysis for an interspecies competition model with chemoattraction-consumption in fluids}
\author[a]{Carlos M. Beltr\'an-Larrotta}
\author[a]{Diego A. Rueda-G\'omez \thanks{Corresponding author:
\href{mailto:diaruego@uis.edu.co}{diaruego@uis.edu.co}  (D. A. Rueda-G\'omez).}}
\author[a]{\'Elder J. Villamizar-Roa}
\affil[a]{Universidad Industrial de Santander, Escuela de Matem\'{a}ticas, A.A. 678, Bucaramanga, Colombia.}

\DeclareRobustCommand{\uvec}[1]{{%
  \ifcsname uvec#1\endcsname
     \csname uvec#1\endcsname
   \else
    \bm{\hat{\mathbf{#1}}}%
   \fi
}}
\date{}
\begin{document}
\maketitle
\begin{abstract}
This work is devoted to the theoretical and numerical analysis of a two-species chemotaxis-Navier-Stokes system with Lotka-Volterra competitive kinetics in a bounded domain of $\mathbb{R}^d$, $d=2,3$. First,  we study the existence of global weak solutions  and establish a regularity criterion which provides sufficient conditions to ensure the strong regularity of the weak solutions. After,
we propose a finite element numerical
scheme in which we use a splitting technique obtained by introducing an auxiliary
variable given by the gradient of the chemical concentration and applying an inductive strategy, in order to deal with the chemoattraction terms in the two-species equations and
prove optimal error estimates. For this scheme, we study the well-posedness and derive some uniform estimates for the discrete variables required in the convergence analysis. Finally, we present some numerical simulations oriented to verify the good behavior of our scheme, as well as to check numerically the optimal error estimates proved in our theoretical analysis.
\vspace{0.3cm}

 \vspace{0.3cm}

\end{abstract}


\section{Introduction}
This paper focuses on the numerical analysis of a parabolic system of partial differential equations describing the evolution of two competing species which react on a  chemoattractant in a liquid surrounding environment. This system is given by the following initial-boundary value problem:
\begin{equation}\label{KNS}
\left\{
\begin{array}{lc}
\partial_t n + \mathbf{u}\cdot\nabla n=D_n \Delta n - \chi_1\nabla\cdot(n\nabla c)+\mu_1 n(1-n-a_1 w)  \ \  \mbox{ in } \Omega\times (0,T),\\ 
\partial_t w + \mathbf{u}\cdot\nabla w=D_w \Delta w - \chi_2\nabla\cdot(w\nabla c)+\mu_2 w(1-a_2 n-w)  \ \  \mbox{ in } \Omega\times (0,T),\\ 
\partial_t c +\mathbf{u}\cdot\nabla c= D_c \Delta c-(\alpha n+\beta w)c  \ \  \mbox{ in } \Omega\times (0,T),\\
\partial_t \mathbf{u}+k(\mathbf{u}\cdot\nabla)\mathbf{u}=D_{\mathbf{u}}\Delta \mathbf{u}+\nabla\pi+(\gamma n+\lambda w)\nabla\phi  \ \  \mbox{ in } \Omega\times (0,T),\\
\nabla\cdot \mathbf{u}=0  \ \  \mbox{ in } \Omega\times (0,T),\\
\displaystyle\frac{\partial n}{\partial {\boldsymbol \nu}} =\frac{\partial w}{\partial {\boldsymbol\nu}}=\frac{\partial c}{\partial {\boldsymbol\nu}}=0,\phantom{h} \mathbf{u}=0, \ \ \mbox{ on } \partial\Omega\times (0,T),\\[.2cm]
\left[n(\x,0), w(\x,0), c(\x,0),\mathbf{u}(\x,0)\right]=\left[n_{0}(\x), w_{0}(\x), c_0(\x
),\mathbf{u}_{0}(\x)\right]  \ \  \mbox{ in } \Omega,
\end{array}
\right.
\end{equation}
where $\Omega \subseteq \mathbb{R}^d$, $d=2,3$, is a bounded domain, $T>0$ and $\boldsymbol{\nu}$ denotes the unit outward normal vector to the boundary $\partial \Omega$. The unknowns are $n,w$ denoting the densities of the two species, $c$ standing the chemical concentration and ${\bf u},\pi$ representing the velocity and the pressure of the fluid, respectively. System (\ref{KNS}) is a model of Lotka-Volterra type in which the movement of the two species is directed by the gradient of an attractant chemical substance, with corresponding chemoattraction coefficients $\chi_1,\chi_2>0.$  The dynamics includes  interspecies competition with population growth rates $\mu_1,\mu_2>0,$ and strengths of competition $a_1,a_2>0.$ The chemical substance is consumed by the species with consumption rates $\alpha,\beta>0.$ Finally, the species and the chemical substance are transported by an incompressible fluid whose dynamics is modeled by the Navier-Stokes equations under the influence of the forcing term $(\gamma n+\lambda w) \nabla \phi.$\\

System (\ref{KNS}) was proposed in \cite{HK} as a generalization of the chemotaxis-Navier-Stokes model in which the complex interaction between  chemotaxis, the Lotka-Volterra competitive kinetics of two species and a fluid are taken into account. The seminal chemotaxis-Navier-Stokes model was proposed in \cite{tuval2005bacterial} to describe the interactions between a cell population and a chemical signal with a liquid environment. Indeed, it was observed that when bacteria of the species Bacillus subtilis are suspended in water, some spatial patterns may spontaneously emerge from initially almost homogeneous distributions of bacteria \cite{dombrowski2004self, tTyson,winkler2012global}. \\

From a theoretical point of view, model (\ref{KNS}) has been studied in two and three dimensional settings (e.g. \cite{CaoK,HK,HYJ,Zheng}). Global existence, boundedness and stabilization of solutions in the two-dimensional case were studied in \cite{HK}; while, the three-dimensional case with $k=0$ (that is, considering the Stokes system instead Navier-Stokes model), was studied in \cite{CaoK}. In \cite{HYJ}, the authors analyzed the  convergence and derive explicit rates of convergence for any supposedly given global bounded classical solution in $d$-dimensional domains ($d=2,3$) for different choices of parameters $a_i$ ($i=1,2$). The global existence and eventual smoothness of weak solutions in the three-dimensional case (with $k\neq 0$) were studied in \cite{Zheng}.  More exactly, for initial data satisfying 
\begin{eqnarray}\label{id}
0<n_0,w_0\in C(\bar{\Omega}), \: \ 0<c_0\in W^{1,q}(\Omega),\: \ {\bf u}_0\in D(A^\theta),\: \ \phi\in C^{1+\eta}(\bar{\Omega}),
\end{eqnarray}
for some $q>3,$ $\theta\in (\frac{3}{4},1)$ and $\eta\in (0,1)$, there exists a global weak solution $[n,w,c,{\bf u}]$ in the class
\begin{eqnarray*}
n,w \in L^2_{loc}([0,\infty); L^2(\Omega))\cap  L^{4/3}_{loc}([0,\infty); W^{1,4/3}(\Omega)), \ \ c \in L^2_{loc}([0,\infty); H^1(\Omega)), \ \ {\bf u}\in L^2_{loc}([0,\infty); V),
\end{eqnarray*}
satisfying (\ref{KNS}) in a totally variational sense. Moreover, there exist $T>0$ and $\gamma\in(0,1)$ such that the weak solution satisfies
$
[n,w,c,{\bf u}]\in C^{2+\frac{\gamma}{2}}([T,\infty); C^{2+\gamma}(\bar{\Omega}))^6.
$
The proof is obtained following the ideas of \cite{lankeit2016long,winkler2016global}. Indeed, the weak solution is obtained 
as the limit of smooth solutions to suitably regularized
problems, where appropriate compactness properties are derived on the basis of some a priori estimates. The existence of classical solutions for the
regularized problems are obtained by applying semigroups theory to ensure local existence and uniqueness
of local smooth solutions, and then, extending such solutions by using a priori estimates.\\

In this paper we deal with the existence of global weak solutions for (\ref{KNS}), as well as regularity properties. Our notion of weak solution is weaker than the one in \cite{Zheng} in the sense that our definition  establishes that the equation for the chemoattractant is satisfied a.e. in $\Omega\times (0,T)$, and the class of initial data (and $\phi$) is larger than the one in  \cite{Zheng} (see Definition \ref{weak} below).
In order to prove the existence of weak solutions, we consider a family of regular solutions for a suitable regularized problem, which differs from that considered in \cite{Zheng} (see system (\ref{KNS_aprox})-(\ref{elip}) below). Indeed, we introduce a decoupling through an auxiliary elliptic problem which allows to gain regularity for the chemical equation as well as to obtain an energy-type inequality after testing and combining conveniently the densities and the concentration equations (see details in Section \ref{SS2}). In connection to the existence of weak solutions, it is worthwhile to remark that the uniqueness of weak solutions in 3D is an open problem; consequently, motivated by the results in \cite{Guillen2,JC-EJ}, the second aim of this paper is to derive a regularity criterion to get global-in-time strong solutions. Our result can generalize those in \cite{JC-EJ}. \\

On the other hand, from a numerical point of view, there are different works focused on the chemotaxis phenomenon (see, for instance, \cite{CHH,GG1,GG2,GG3,GG4,GUO,JV,IBRA,SS} and references therein); there, some numerical schemes have been proposed  by using different methods, analyzing properties of the discrete solutions such as existence, uniqueness, positivity, energy-stability, convergence, mass-conservation, asymptotic behavior, error estimates,  among others. Now, when the interaction with a fluid is considered, as far as we know, the literature related to the numerical analysis of chemotaxis-Navier-Stokes system is scarce, see \cite{Saad,EDA}. For the 2D-Keller-Segel-Stokes model,  in \cite{Saad} it was studied the convergence of a numerical scheme, obtained by the combination of the finite volume method and the nonconforming finite element method; while, for a 3D-chemotaxis-Navier-Stokes system, the convergence of a finite element scheme has been analyzed in \cite{EDA}. However, up to our knowledge, for the model (\ref{KNS}) in which, the diffusion by chemotaxis, the Lotka-Volterra kinetics and the fluid interaction  are combined, there are no works dealing with  numerical analysis.\\

Taking into account the above, the third aim of this paper is to develop  the numerical analysis  of model (\ref{KNS}). It is worthwhile to remark that cross-diffusion mechanisms governing the chemotactic phenomena and the non-linear dynamics caused by their cuopling with the interaction between competing species and the fluid, made the model (\ref{KNS}) numerically challenging. Here, we propose a finite element numerical scheme in which, in order to treat with the chemoattraction terms in the two-species equations and the other nonlinear terms coming from the interspecies competition and the signal consumption, we use a splitting technique obtained by introducing an auxiliary variable given by the ${\bf s}=\nabla c$ and applying an inductive strategy. This idea allows us to prove optimal error estimates.  Likewise, we will use some skew-symmetric trilinear forms (see (\ref{a6}) and (\ref{a6aa})) to control the transport terms, preserving the alternance property as in the continuous case.   \\

The layout of this paper is as follows: In Section \ref{SS2}, we introduce some basic notations, preliminary results, and establish the main results corresponding to the continuous problem. Specifically, we prove the existence of weak solutions for system (\ref{KNS}) and a regularity criterion under which weak solutions of (\ref{KNS}) are also strong solutions.  In Section \ref{SS3}, we define an equivalent  weak formulation of (\ref{KNS}), from which we construct the finite element numerical approximation. We prove its well-posedness and some uniform estimates for any discrete solution, which are required in the convergence analysis.  In Section \ref{ESWN}, we carriet out the corresponding convergence analysis proving optimal error estimates in time and space. Finally, in Section \ref{SS5NS}, we provide some numerical simulations in agreement with the theoretical results reported in this paper and the reference \cite{HYJ}.

\section{The continuous problem}\label{SS2}
We will start by recalling some basic notations and preliminary results to be used later. Hereafter, $\Omega$ is a bounded domain of $\mathbb{R}^3$ with boundary of class $C^{2,1}$. The reason to assume this particular regularity comes from the parabolic and elliptic regularity results which we will use frequently throughout this paper (cf. Theorem \ref{Feiresl} and system \eqref{elip} below). We consider the standard Sobolev and Lebesgue spaces $W^{k,p}(\Omega)$ and $L^p(\Omega),$ with respective norms $\Vert \cdot\Vert_{W^{k,p}}$ and $\Vert \cdot\Vert_{L^p}.$ In particular, we denote $W^{k,2}(\Omega)=H^k(\Omega).$ Also, $W_0^{1,p}(\Omega)$ denotes the elements of $W^{1,p}(\Omega)$ with trace zero on $\partial \Omega.$  The $L^2(\Omega)$-inner product will be represented by $(\cdot,\cdot).$ We also consider the  space  $L^2_\sigma(\Omega)^3$ defined by the closure of $\{{\bf u}\in C^\infty_0(\Omega)^3,\ \nabla \cdot {\bf u}=0\}$ in $L^2(\Omega)$ and the  following function spaces
\begin{equation*}
V:=\{\mathbf{v} \in {H}^1_0(\Omega)^3:\nabla\cdot \mathbf{v}=0 \: \text{ in} \: \Omega\},\ \  
{H}^{1}_{\mathbf{s}}(\Omega):=\{\mathbf{w}\in {H}^{1}(\Omega)^3: \mathbf{w}\cdot \boldsymbol{\nu}=0 \mbox{ on } \partial\Omega\},
\end{equation*}
\begin{equation*}
{L}_0^2(\Omega):= \left\{ p \in L^2(\Omega): \int_{\Omega}p =0 \right\}.
\end{equation*}
 We will use the following equivalent norm in  $H_{\bf s}^1(\Omega)$ (see \cite[Corollary 3.5]{Nour}):
\begin{equation}\label{EQs}
\Vert {\bf s} \Vert_{H^1}^2={\Vert {\bf s}\Vert_{L^2}^2 + \Vert \mbox{rot }{\bf s}\Vert_{L^2}^2 + \Vert \nabla \cdot {\bf s}\Vert_{L^2}^2}, \ \ \forall {\bf s}\in H^{1}_{{\bf s}}(\Omega)
\end{equation}
and the well known Poincar\'e and embedding inequalities
\begin{eqnarray}\label{PIa}
\|\mathbf{v}\|_{{H}^1} \leq C_{P}\|\nabla \mathbf{v}\|_{{L}^2}, \ \ \forall \mathbf{v} \in H^1_0(\Omega)^3,\\
\|u\|_{{L}^\infty} \leq C\| u\|_{{H}^2}, \ \ \forall u \in {H}^2(\Omega),\label{ime1}
\end{eqnarray}
for some constants $C_{P},C>0,$ which depend on $\Omega,$ but are independent of $\mathbf{v}$ and $u,$ respectively. Also, we will use the classical 3D interpolation inequalities
\begin{eqnarray}
\Vert u\Vert_{L^3}&\leq &\Vert u\Vert_{L^2}^{1/2}\Vert u\Vert_{L^6}^{1/2} \  \ \forall u\in H^1(\Omega),\label{in3D}\\
\Vert u\Vert_{L^5}&\leq &\Vert u\Vert_{L^2}^{1/10}\Vert u\Vert_{H^1}^{9/10} \  \ \forall u\in H^1(\Omega).\label{in3Dl4}
\end{eqnarray}
We also consider the Stokes operator $A:= - P\Delta,$ with domain
$D(A)=V\cap H^2(\Omega)^3,$ where  $P:L^2(\Omega)^3 \rightarrow L^2_\sigma(\Omega)^3$ stands the Leray projector. Along this paper, we consider a fixed but arbitrary time interval $(0,T),$ $0<T<\infty.$ For $Y$ being a Banach space, we denote by $L^p(Y):=L^p(0,T;Y)$, $1\leq p\leq \infty$, the space of Bochner integrable functions defined on the interval $[0,T]$ with values in $Y$, endowed with the usual norm $\Vert \cdot\Vert_{L^p(Y)}$. We also consider the space $C(Y):=C([0,T];Y)$ of continuous functions from $[0,T]$ into $Y,$ with norm $\Vert \cdot\Vert_{C(Y)}.$ For a Banach space $Y,$ we denote the duality product between the dual $Y^{\prime}$ and $Y$ by $\langle\cdot,\cdot\rangle_{Y^{\prime}}$ or simply $\langle\cdot,\cdot\rangle$ if there is no ambiguity. From now on, for simplicity in the notation, we use the same notation for both scalar and vector valued functions spaces. Also, the letter $C$ will denote different positive constants (independent of discrete parameters) which may change from line to line (or even within the same line).\\

In order to analyze properties of existence and regularity of solutions of (\ref{KNS}), frequently we will use the following parabolic regularity result.

\begin{theo}[\cite{Feireisl}, Theorem 10.22] \label{Feiresl}
				Let $\Omega$ be a bounded domain with boundary $\partial\Omega$ of class $C^2$, $1<p,q<\infty$. Suppose that $f\in L^p(L^q)$, $u_0\in Z_{p,q}=\{L^q(\Omega);\mathcal{D}(\Delta_{\mathcal{N}})\}_{1-1/p,p}$, $\mathcal{D}(\Delta_{\mathcal{N}})=\{v\in W^{2,q}(\Omega):\frac{\partial v}{\partial \nu}=0 \ \text{on} \ \partial\Omega\}$,  where $\{\cdot;\cdot\}_{\cdot,\cdot}$ denotes the real interpolation space. Then the problem
				\begin{equation*}
				\begin{cases}
				\partial_tu-\Delta u=f, &\text{in} \ (0,T)\times\Omega, \\
				u(0,\cdot)=u_0, &\text{in} \ \Omega, \\
				\frac{\partial u}{\partial \nu}=0, &\text{on} \ (0,T)\times\partial\Omega,
				\end{cases}
				\end{equation*}
				admits a unique solution $u$ such that
				$
				u\in C(Z_{p,q})\cap L^p(W^{2,q}), \partial_tu\in L^p(L^q).
				$
Moreover, there exists a positive constant $C:=C(p,q,\Omega,T)$ such that
				\begin{equation*}
				\|u(t)\|_{C(Z_{p,q})}+\|\partial_tu\|_{L^p(L^q)}+\|\Delta u\|_{L^p(L^q)}\le C(\|f\|_{L^p(L^q)}+\|u_0\|_{Z_{p,q}}).
				\end{equation*}
			\end{theo}

If $p=q,$ it holds that
\begin{equation*}
				Z_{p,p}=\widehat{W}^{2-2/p,p}=\begin{cases}
				W^{2-2/p,p}(\Omega),\ \mbox{if}\ p<3, \\
				\left\{ u\in W^{2-2/p,p}(\Omega):\ \frac{\partial u}{\partial \nu}=0\ \mbox{on}\ \partial\Omega\right\}, \ \mbox{if}\ p>3.
				\end{cases}
				\end{equation*}	
We will denote by $\mathcal{X}_p$ the space		
$$\mathcal{X}_p=\left\{ u\in C(Z_{p,p})\cap L^p(W^{2,p}):\partial_tu\in L^p(L^p) \right\}.$$

\subsection{Variational formulation, and existence and regularity results}
We start this subsection by introducing the notion of weak and strong solutions of (\ref{KNS}).

\begin{defi}\label{weak}(\textit{Weak solution of (\ref{KNS})})
Let $0<T<\infty$,  $\nabla\phi\in L^\infty(\Omega)$, $n_0,w_0\in L^2(\Omega)$, {$c_0\in W^{1,q}(\Omega)$, $q>3$}, ${\bf u}_0\in L^2(\Omega)$ and $n_0,w_0,c_0\geq 0$ in $\Omega$. A 
quadruple $[n,w,c,\bf u]$ is said a 
weak solution of (\ref{KNS}) if
\begin{eqnarray*}
n,w\in {L^{{5}/{3}}(L^{{5}/{3}}})\cap L^{5/4}(W^{1,{5}/{4}}),\ \partial_t n, \partial_t w \in L^{10/9}((W^{1,10})'),\\
c\in L^\infty(H^1)\cap L^2(H^2),\ \partial_tc\in L^{10/7}(L^{10/7}),\\
{\bf u}\in L^\infty(L^2)\cap L^2(V),\ \partial_t {\bf u}\in L^{5/3}((W^{1,5/2})'),
\end{eqnarray*}
satisfying the equation (\ref{KNS})$_3$ a.e $(x,t)$ in $\Omega\times (0,T),$ and verifying 
\begin{eqnarray*}
\int_0^T\langle \partial_tn,\varphi_1\rangle+\int_0^T\int_\Omega \left [D_n \nabla n-n{\bf u}-\chi_1 n\nabla c\right]\cdot \nabla\varphi_1=\mu_1 \int_0^T\int_\Omega n(1-n-a_1w)\varphi_1,\\
\int_0^T\langle \partial_tw,\varphi_2\rangle+\int_0^T\int_\Omega \left[D_w \nabla w- w{\bf u}-\chi_2 w\nabla c\right]\cdot \nabla\varphi_2=\mu_2 \int_0^T\int_\Omega w(1-a_2n-w)\varphi_2,\\
\int_0^T\langle \partial_t{{\bf u}},{\boldsymbol \psi}\rangle+\int_0^T\int_\Omega \left[D_{\mathbf{u}}\nabla {{\bf u}}-k\, {{\bf u}}\otimes{{\bf u}}\right]\cdot\nabla{\boldsymbol \psi}=\int_0^T\int_\Omega (\gamma n+\lambda w)\nabla\phi\cdot{\boldsymbol \psi},
\end{eqnarray*}
{for all $\varphi_1,\varphi_2\in L^{10}(W^{1,10})$ and ${\boldsymbol \psi}\in L^{5/2}( W^{1,5/2}).$ }
\end{defi}

\begin{defi}\label{strong}(\textit{Strong solution of  (\ref{KNS})})
Let $\nabla\phi\in L^\infty(\Omega)$, $[n_0,w_0,c_0,{\bf u}_0]\in \widehat{W}^{3/2,4}\times \widehat{W}^{3/2,4}\times \widehat{W}^{3/2,4}\times V$ and $n_0,w_0,c_0\geq 0$ in $\Omega$. A strong solution of (\ref{KNS}) is a quadruple $[n,w,c,{\bf u}]$ with
\begin{eqnarray*}
[n,w,c,{\bf u}]\in \mathcal{X}_4\times\mathcal{X}_4\times \mathcal{X}_4\times \mathcal{X}_2,
\end{eqnarray*}
satisfying \eqref{KNS}$_{6,7}$ and such that \eqref{KNS}$_{1-5}$ holds a.e $(x,t)$ in $\Omega\times (0,T).$ 
\end{defi}

The first aim of this paper is to prove the following theorem of existence of weak solutions.
\begin{theo}\label{teor2} (\textit{Existence of weak solution}) There exists a global weak solution of system (\ref{KNS}), in the sense of Definition \ref{weak}.
\end{theo}

Under the extra assumption $[n,w,{\bf u}]\in L^{20/7}(L^{20/7})\times L^{20/7}(L^{20/7})\times L^8(L^4)$, weak solutions are strong in the sense of the following definition. This is the content of the next theorem.
\begin{theo}{\label{teor3} (\textit{Regularity criterion})}
{Let $[n,w,c,{\bf u}]$ be a weak solution of \eqref{KNS}. If $[n_0,w_0,c_0,{\bf u}_0]\in \widehat{W}^{3/2,4}\times\widehat{W}^{3/2,4}\times \widehat{W}^{3/2,4}\times V$ and
\begin{equation}
\label{reg1}
[n,w,{\bf u}]\in L^{20/7}(L^{20/7})\times L^{20/7}(L^{20/7})\times L^8(L^4),
\end{equation}
then $[n,w,c,{\bf u}]$ is a strong solution of \eqref{KNS}  in the sense of Definition \ref{strong}. }
\end{theo}

\subsection{Existence of weak solutions. Proof of Theorem \ref{teor2}}
We look for a weak solution of (\ref{KNS}) as the limit of the strong solutions for
a family of suitable regularized problems.  In this section, for simplicity in the calculations, we consider the parameter $k$ and the self-diffusion coefficients $D_{\bf u}$, $D_n$, $D_w$ and $D_c$ equal to 1.

\subsubsection{Approximated solutions}
 We consider a regularized problem, depending of a parameter $\epsilon$, and then, we obtain the existence of weak solutions as the limit, as the parameter $\epsilon$ goes to zero, of the corresponding  strong solutions. Precisely, given $\epsilon>0,$ we consider the regularized system
{\begin{equation}\label{KNS_aprox}
\left\{
\begin{array}{lc}
\partial_t n_{\epsilon}+{\bf u}_\epsilon \cdot\nabla n_\epsilon=\Delta n_\epsilon-\chi_1 \nabla\cdot\left(\frac{n_\epsilon^+}{1+\epsilon n_\epsilon^+}\nabla c_\epsilon\right)+\mu_1n_\epsilon^+(1-n_\epsilon-a_1w_\epsilon)\ \ \mbox{in}\ \Omega\times(0,T),&\\
\partial_t w_{\epsilon}+{\bf u}_\epsilon \cdot\nabla w_\epsilon=\Delta w_\epsilon-\chi_2 \nabla\cdot\left(\frac{w_\epsilon^+}{1+\epsilon w_\epsilon^+}\nabla c_\epsilon\right)+\mu_2w_\epsilon^+(1-a_2n_\epsilon-w_\epsilon)\ \ \mbox{in}\ \Omega\times(0,T),&\\
 \partial_t s_{\epsilon}+{\bf u}_\epsilon \cdot\nabla s_\epsilon=\Delta s_\epsilon- \alpha s^+_\epsilon\frac{1}{\epsilon}\ln (1+\epsilon n_\epsilon^+)- \beta s^+_\epsilon\frac{1}{\epsilon}\ln (1+\epsilon w_\epsilon^+) \ \ \ \ \ \ \mbox{in}\ \Omega\times(0,T),  &\\
\partial_t {\bf u}_{\epsilon}+(Y_\epsilon{\bf u}_\epsilon\cdot\nabla){\bf u}_\epsilon= \Delta {\bf u}_\epsilon+\nabla\pi_\epsilon+(\gamma n_\epsilon+\lambda w_\epsilon) \nabla \phi\ \ \ \ \mbox{in}\ \Omega\times(0,T), &\\
\nabla\cdot {\bf u}_\epsilon=0\ \ \ \ \mbox{in}\ \Omega\times(0,T), &\\
\left[n_\epsilon(0),w_\epsilon(0),s_\epsilon(0),{\bf u}_\epsilon(0)\right]=\left[n_{0,\epsilon},w_{0,\epsilon},s_{0,\epsilon},{\bf u}_{0,\epsilon}\right]\ \ \ \ \mbox{in}\ \Omega,\\[.1cm]
\frac{\partial n_\epsilon(x,t)}{\partial \nu}=\frac{\partial w_\epsilon(x,t)}{\partial \nu}=\frac{\partial s_\epsilon(x,t)}{\partial \nu}=0, \quad {\bf u}_\epsilon(x,t)=0\ \ \ \ \ \ \mbox{on}\ \partial\Omega\times(0,T),
\end{array}
\right.
\end{equation}}
where $c_\epsilon$ is the unique solution of the elliptic problem
\begin{equation}\label{elip}
\left\{
\begin{array}{lc}
c_\epsilon-\epsilon\Delta c_\epsilon=s_\epsilon\ \mbox{in}\ \Omega,\\
\frac{\partial c_\epsilon(x,t)}{\partial \nu}=0\ \mbox{on}\ \partial\Omega,
\end{array}
\right.
\end{equation}
and $Y_\epsilon=(1+\epsilon A)^{-1}$ is the Yosida approximation \cite{Sohr}. In (\ref{KNS_aprox}), $a^+_\epsilon=\mbox{max}\{a_\epsilon,0\}\geq 0$ for a general variable $a$, and $n_{0,\epsilon},c_{0,\epsilon}, s_{0,\epsilon},{\bf u}_{0,\epsilon}$ are initial data satisfying
\begin{equation}\label{elip1}
\left\{
\begin{array}{lc}
n_{0,\epsilon}\in W^{4/5,5/3}(\Omega),\ n_{0,\epsilon}>0,\ \ n_{0,\epsilon}\rightarrow n_0\ \mbox{in}\ L^2(\Omega),\ \mbox{as}\ \epsilon\rightarrow 0,\\
w_{0,\epsilon}\in W^{4/5,5/3}(\Omega),\ w_{0,\epsilon}>0,\ \ w_{0,\epsilon}\rightarrow w_0\ \mbox{in}\ L^2(\Omega),\ \mbox{as}\ \epsilon\rightarrow 0,\\
s_{0,\epsilon}\geq 0, \ s_{0,\epsilon}=c_{0,\epsilon}-\epsilon\Delta c_{0,\epsilon}\in {W}^{1,q}(\Omega),\ q>3,\\ 
c_{0,\epsilon}\geq0, \ c_{0,\epsilon}\rightarrow c_0\ \mbox{in}\ H^1(\Omega),\ \mbox{as}\ \epsilon\rightarrow 0,\\
{s_{0,\epsilon}\rightarrow c_0\ \mbox{in}\ W^{1,q}(\Omega),\ \mbox{as}\ \epsilon\rightarrow 0,}\\
{\bf u}_{0,\epsilon}\in V,\ {\bf u}_{0,\epsilon}\rightarrow {\bf u}_0\ \mbox{in}\ L^2(\Omega),\ \mbox{as}\ \epsilon\rightarrow 0.
\end{array}
\right.
\end{equation}

Next proposition provides the existence of strong solutions for (\ref{KNS_aprox})-(\ref{elip}).
\begin{prop}\label{p1}
For each $\epsilon>0,$ there exist unique non-negative functions $n_{\epsilon}\in \mathcal{X}_{5/3}$, $w_\epsilon\in \mathcal{X}_{5/3}$, $s_\epsilon\in \mathcal{X}_{5/3}$, ${\bf u}_\epsilon\in \mathcal{X}_2$, which, together with some $\pi_\epsilon\in L^2(H^1)$ solve the system (\ref{KNS_aprox}) pointwisely a.e. $(x,t)\in \Omega\times (0,T).$
\end{prop}
\begin{proof}
We will use the Leray-Schauder fixed point theorem. Let us consider the Banach space $X=L^\infty(L^2)\cap L^2(H^1)$ and $\Gamma:X\times X\times X\rightarrow X\times X\times X$ defined by $\Gamma([\bar{n}_\epsilon,\bar{w}_\epsilon,\bar{s}_\epsilon])=[{n}_\epsilon,{w}_\epsilon,s_\epsilon],$ where ${n}_\epsilon,{w}_\epsilon,{s}_\epsilon$ are the three first components of the solution $[{n}_\epsilon,{w}_\epsilon, s_\epsilon, {\bf u}_\epsilon,\pi_\epsilon]$ of the following system:
\begin{equation}\label{KNS_lin}
\left\{
\begin{array}{lc}
\partial_t n_{\epsilon}-\Delta n_\epsilon=-{\bf u}_\epsilon \cdot\nabla \bar{n}_\epsilon-\chi_1 \nabla\cdot\left(\frac{ \bar{n}^+_\epsilon}{1+\epsilon \bar{n}_\epsilon^+}\nabla \bar{c}_\epsilon\right)+\mu_1\bar{n}_\epsilon^+(1-\bar{n}_\epsilon-a_1\bar{w}_\epsilon)\ \mbox{in}\ \Omega\times(0,T),&\\
\partial_t w_{\epsilon}-\Delta w_\epsilon=-{\bf u}_\epsilon \cdot\nabla \bar{w}_\epsilon-\chi_2 \nabla\cdot\left(\frac{ \bar{w}^+_\epsilon}{1+\epsilon \bar{w}_\epsilon^+}\nabla \bar{c}_\epsilon\right)+\mu_2\bar{w}_\epsilon^+(1-a_2\bar{n}_\epsilon-\bar{w}_\epsilon)\ \mbox{in}\ \Omega\times(0,T),&\\
\partial_t s_{\epsilon}-\Delta s_\epsilon=-{{\bf u}_\epsilon} \cdot\nabla \bar{s}_\epsilon- \alpha \bar{s}^+_\epsilon \frac{1}{\epsilon}\ln(1+\epsilon\bar{n}_\epsilon^+)- \beta \bar{s}^+_\epsilon \frac{1}{\epsilon}\ln(1+\epsilon\bar{w}_\epsilon^+)\ \mbox{in}\ \Omega\times(0,T),  &\\
\partial_t {\bf u}_{\epsilon}-\Delta {\bf u}_\epsilon+(Y_\epsilon{\bf u}_\epsilon\cdot\nabla){\bf u}_\epsilon-\nabla\pi_\epsilon=(\gamma\bar{n}_\epsilon+\lambda\bar{w}_\epsilon) \nabla \phi\ \mbox{in}\ \Omega\times(0,T), &\\
\nabla\cdot {\bf u}_\epsilon=0\ \mbox{in}\ \Omega\times(0,T), &\\
\left[n_\epsilon(0),w_\epsilon(0),s_\epsilon(0),{\bf u}_\epsilon(0)\right]=\left[n_{0,\epsilon},w_{0,\epsilon},s_{0,\epsilon},{\bf u}_{0,\epsilon}\right]\ \ \ \ \mbox{in}\ \Omega,\\[.3cm]
\frac{\partial n_\epsilon(x,t)}{\partial \nu}=\frac{\partial w_\epsilon(x,t)}{\partial \nu}=\frac{\partial s_\epsilon(x,t)}{\partial \nu}=0, \quad {\bf u}_\epsilon(x,t)=0\ \ \ \ \ \ \mbox{on}\ \partial\Omega\times(0,T),\end{array}
\right.
\end{equation}
with $\bar{c}_\epsilon$ being the unique solution of (\ref{elip}) with right hand side $\bar{s}_\epsilon.$ First of all, $\Gamma$ is well defined. {In fact, if $[\bar{n}_\epsilon,\bar{w}_\epsilon, \bar{s}_\epsilon]\in X\times X\times X,$ following \cite{Sohr}, Chapter V, Section 2.3, there exists a unique  ${\bf u}_\epsilon\in \mathcal{X}_2\hookrightarrow L^{10}(L^{10})$ solution of $\eqref{KNS_lin}_{4}$-$\eqref{KNS_lin}_{5}$ {with initial data ${\bf u}_{0,\epsilon}$ and non-slip boundary condition.} Also, since $\bar{s}_\epsilon\in X$ and recalling that $\partial\Omega\in C^{1,1},$ from the elliptic regularity applied to the problem \eqref{elip} (cf. \cite{Grisvard}, Theorem
2.4.2.7 and Theorem 2.5.1.1), there exists a unique solution $\bar{c}_\epsilon\in L^\infty(H^2)\cap L^2(H^3)$. Then $\nabla\bar{c}_\epsilon\in L^\infty(H^1)\cap L^2(H^2)\hookrightarrow L^{10}(L^{10})$.
Since $\nabla\bar{n}_\epsilon^+,\nabla\bar{w}_\epsilon^+\in L^2(L^2)$ we get
\begin{equation*}
\chi_1 \nabla\cdot\left(\frac{ \bar{n}^+_\epsilon}{1+\epsilon \bar{n}_\epsilon^+}\nabla \bar{c}_\epsilon\right),\ \chi_2 \nabla\cdot\left(\frac{ \bar{w}^+_\epsilon}{1+\epsilon \bar{w}_\epsilon^+}\nabla \bar{c}_\epsilon\right) \in L^{5/3}(L^{5/3}).
\end{equation*}
Moreover, since ${\bf u}_\epsilon\in L^\infty(H^1)\cap L^2(H^2)\hookrightarrow L^{10}(L^{10})$ we have $\displaystyle {\bf u}_\epsilon \cdot\nabla \bar{n}_\epsilon,{\bf u}_\epsilon \cdot\nabla \bar{w}_\epsilon, {\bf u}_\epsilon \cdot\nabla \bar{s}_\epsilon, \bar{n}_\epsilon^+,\bar{n}_\epsilon^+\bar{n}_\epsilon,\\ \bar{n}_\epsilon^+\bar{w}_\epsilon, \bar{w}_\epsilon^+,\bar{w}_\epsilon^+\bar{w}_\epsilon,\bar{w}_\epsilon^+\bar{n}_\epsilon\in L^{5/3}(L^{5/3})$. Therefore the right hand sides of $\eqref{KNS_lin}_1$, $\eqref{KNS_lin}_2$ and $\eqref{KNS_lin}_3$ belong to $L^{5/3}(L^{5/3})$ and thus, from Theorem \ref{Feiresl}, there exists a unique $[n_\epsilon,w_\epsilon,s_\epsilon]\in\mathcal{X}_{5/3}\times \mathcal{X}_{5/3}\times \mathcal{X}_{5/3}$ solution of \eqref{KNS_lin}$_1$-\eqref{KNS_lin}$_3$, such that
\begin{align}
\label{KNS_lin03}
\|n_\epsilon\|_{\mathcal{X}_{5/3}}&\le C(\|n_{0,\epsilon}\|_{W^{4/5,5/3}},\|{\bf u}_\epsilon\|_{\mathcal{X}_2},\|\bar{n}_\epsilon\|_X,\|\bar{w}_\epsilon\|_X,\|\bar{s}_\epsilon\|_X), \\
\label{KNS_lin04}
\|w_\epsilon\|_{\mathcal{X}_{5/3}}&\le C(\|w_{0,\epsilon}\|_{W^{4/5,5/3}},\|{\bf u}_\epsilon\|_{\mathcal{X}_2},\|\bar{n}_\epsilon\|_X,\|\bar{w}_\epsilon\|_X,\|\bar{s}_\epsilon\|_X), \\
\label{KNS_lin05}
\|s_\epsilon\|_{\mathcal{X}_{5/3}}&\le C(\|c_{0,\epsilon}-\epsilon\Delta c_{0,\epsilon}\|_{W^{4/5,5/3}},\|{\bf u}_\epsilon\|_{\mathcal{X}_2},\|\bar{n}_\epsilon\|_X,\|\bar{w}_\epsilon\|_X,\|\bar{s}_\epsilon\|_X).
\end{align}
Since $\mathcal{X}_{5/3}\hookrightarrow L^2(H^{3/2})\cap L^\infty(L^2)\hookrightarrow X$, we conclude $\Gamma$ is well defined}.\\
\\
{{Now we prove that $\Gamma:X\times X\times X\rightarrow X\times X\times X$ is compact.}} {{Notice that $\mathcal{X}_{5/3}\hookrightarrow L^\infty(H^{1/2})\cap L^{5/3}(H^{17/10})$. Thus, by Lemma 6 in \cite{Guillen2}}}, {$\mathcal{X}_{5/3}\hookrightarrow L^2(H^{3/2})\cap L^\infty(H^{1/2})$ and $H^{3/2}\overset{c}{\hookrightarrow}H^1\hookrightarrow L^{5/3},$ $H^{1/2}\overset{c}{\hookrightarrow}L^{2}\hookrightarrow L^{5/3}$, where $\overset{c}{\hookrightarrow}$ denotes the compact embedding. Then, by using the Aubin-Lions and Simon compactness theorems \cite{lions2,Simon}, the embedding of $ \{ u\in L^2(H^{3/2})\cap L^\infty(H^{1/2}):\partial_tu\in L^{5/3}(L^{5/3}) \}$ into $X$ is compact}.\\
\\
{{Next, we will prove that the fixed points of $\bar{\alpha}\Gamma,$ $\bar{\alpha}\in[0,1]$ are bounded in $X\times X\times X$.}}
{If $\bar{\alpha}=0$ the result is clear; then we assume $\bar{\alpha}\in(0,1].$ If $[n_\epsilon,w_\epsilon,s_\epsilon]$ is a fixed point of $\bar{\alpha}\Gamma$, then $[n_\epsilon,w_\epsilon, s_\epsilon]=\bar{\alpha}\Gamma[n_\epsilon,w_\epsilon,s_\epsilon]$, which says that $[n_\epsilon,w_\epsilon,s_\epsilon]$ verifies}
{\begin{equation}\label{KNS_lin2}
\left\{\begin{array}{lc}
\partial_t n_{\epsilon}-\Delta n_\epsilon=-\bar{\alpha}{\bf u}_\epsilon \cdot\nabla n_\epsilon-\bar{\alpha}\chi_1 \nabla\cdot\left(\frac{ n^+_\epsilon}{1+\epsilon n_\epsilon^+}\nabla c_\epsilon\right)+\bar{\alpha}\mu_1 n_\epsilon^+(1-n_\epsilon-a_1 w_\epsilon)\ \mbox{in}\ \Omega\times(0,T),&\\
\partial_t w_{\epsilon}-\Delta w_\epsilon=-\bar{\alpha}{\bf u}_\epsilon \cdot\nabla w_\epsilon-\bar{\alpha}\chi_2 \nabla\cdot\left(\frac{ w^+_\epsilon}{1+\epsilon w_\epsilon^+}\nabla c_\epsilon\right)+\bar{\alpha}\mu_2w_\epsilon^+(1-a_2n_\epsilon-w_\epsilon)\ \mbox{in}\ \Omega\times(0,T),&\\
\partial_t s_{\epsilon}-\Delta s_\epsilon=-\bar{\alpha}{{\bf u}_\epsilon} \cdot\nabla s_\epsilon- \bar{\alpha}\alpha s^+_\epsilon \frac{1}{\epsilon}\ln(1+\epsilon n_\epsilon^+)-\bar{\alpha} \beta s^+_\epsilon \frac{1}{\epsilon}\ln(1+\epsilon w_\epsilon^+)\ \mbox{in}\ \Omega\times(0,T),  &\\
\partial_t {\bf u}_{\epsilon}-\Delta {\bf u}_\epsilon+(Y_\epsilon{\bf u}_\epsilon\cdot\nabla){\bf u}_\epsilon-\nabla\pi_\epsilon=(\gamma n_\epsilon+\lambda w_\epsilon) \nabla \phi\ \mbox{in}\ \Omega\times(0,T), &\\
\nabla\cdot {\bf u}_\epsilon=0\ \mbox{in}\ \Omega\times(0,T).\end{array}
\right.
\end{equation}}
\hspace{-2mm}{We first prove that $n_\varepsilon\geq 0$. Let $[n_\varepsilon, w_\varepsilon, s_\varepsilon, {\bf u}_\varepsilon, \pi_\varepsilon ]$ be a solution of $\eqref{KNS_lin2}$. As described before, the
right hand side of $\eqref{KNS_lin2}_1$ belongs to $L^{5/3}(L^{5/3})$. Moreover $n_\epsilon\in X\hookrightarrow L^{10/3}(L^{10/3})\hookrightarrow L^{5/2}(L^{5/2})$. Thus, testing $\eqref{KNS_lin2}_1$ by $n_\epsilon^-=\min\{n_\epsilon,0\}\leq 0$, and taking into account that the condition $\nabla\cdot {\bf u}_\epsilon=0$ implies $\int_\Omega({\bf u}_\epsilon\cdot\nabla n_\epsilon^-)n_\epsilon^-=0$, we obtain}
\begin{equation*}
\frac{1}{2}\frac{d}{dt}\|n_\epsilon^-\|_{L^2}^2+\|\nabla n_\epsilon^-\|_{L^2}^2=\bar{\alpha}\chi_1 \left(\frac{n^+_\epsilon}{1+\epsilon n_\epsilon^+}\nabla c_\epsilon,\nabla n^-_\epsilon\right)+\bar{\alpha}\mu_1 (n_\epsilon^+(1-n_\epsilon-a_1 w_\epsilon),n^-_\epsilon)=0,
\end{equation*}
from which $n_\epsilon^-=0$ and then $n_\epsilon\ge 0$ a.e. in $\Omega\times(0,T)$. Analogously, testing $\eqref{KNS_lin2}_2$ and $\eqref{KNS_lin2}_3$ by $w_\epsilon^-$ and $s_\epsilon^-$ respectively, we get $w_\epsilon\ge 0$ and $s_\epsilon\ge 0$ a.e. in $\Omega\times(0,T)$.\\
Now we prove that $s_\epsilon$ is bounded in $X$. For that, testing $\eqref{KNS_lin2}_3$ by $s_\epsilon,$ integrating in $\Omega$, taking into account that $\nabla\cdot {\bf u}_\epsilon=0$, we get 
\begin{align}
	\label{KNS_lin3}
\frac{1}{2}\frac{d}{dt}\Vert s_\epsilon\Vert_{L^2}^2+\Vert \nabla s_\epsilon\Vert_{L^2}^2&=-\bar{\alpha}\left(\alpha s_\epsilon\frac{1}{\epsilon}\ln(1+\epsilon n_\epsilon)+\beta s_\epsilon\frac{1}{\epsilon}\ln(1+\epsilon w_\epsilon),s_\epsilon\right)\leq0.
\end{align}
Thus, integrating in time in (\ref{KNS_lin3}) we get $\|s_\epsilon\|_{L^\infty(L^2)}+\|s_\epsilon\|_{L^2(H^1)}\leq C.$
Now we will bound $n_\epsilon$ in $X$. Testing $\eqref{KNS_lin2}_1$ by $n_\epsilon$ and integrating in $\Omega$ we get
\begin{align}
\label{KNS_lin5}
\frac{1}{2}\frac{d}{dt}\|n_\epsilon\|_{L^2}^2+\|\nabla n_\epsilon\|_{L^2}^2&= \bar{\alpha}\chi_1 \left(\frac{n_\epsilon}{1+\epsilon n_\epsilon}\nabla c_\epsilon,\nabla n_\epsilon\right) +\bar{\alpha}(\mu_1 n_\epsilon(1-n_\epsilon-a_1 w_\epsilon),n_\varepsilon)\notag\\
&\le \bar{\alpha}\chi_1\|n_\epsilon\|_{L^4}\|\nabla c_\epsilon\|_{L^4}\|\nabla n_\epsilon\|_{L^2}+\bar{\alpha}\mu_1\|n_\epsilon\|_{L^2}^2 \notag\\
&\le C\bar{\alpha}\chi_1(\|n_\epsilon\|_{L^2}^{1/4}\|\nabla n_\epsilon\|_{L^2}^{3/4}+\|n_\epsilon\|_{L^2})\|\nabla c_\epsilon\|_{L^4}\|\nabla n_\epsilon\|_{L^2} +\bar{\alpha}\mu_1\|n_\epsilon\|_{L^2}^2\notag\\
&\le \frac{1}{2}\|\nabla n_\epsilon\|_{L^2}^2+C\|n_\epsilon\|_{L^2}^2\|\nabla c_\epsilon\|_{L^4}^8+C\|n_\epsilon\|_{L^2}^2\|\nabla c_\epsilon\|_{L^4}^2+\bar{\alpha}\mu_1\|n_\epsilon\|_{L^2}^2.
\end{align}
Since $c_\epsilon$ is the solution of \eqref{elip} and $s_\epsilon\in X$, then $c_\epsilon\in L^\infty(H^2)\cap L^2(H^3)$ (cf. \cite{Grisvard}, Theorem
2.4.2.7 and Theorem 2.5.1.1), and 
$
\|c_\epsilon\|_{L^\infty(H^2)\cap L^2(H^3)}\le C(\epsilon)\|s_\epsilon\|_X.
$
In particular,
\begin{equation*}
\|\nabla c_\epsilon\|_{L^\infty(L^4)}\le C(\epsilon)\|s_\epsilon\|_X\le C(\epsilon).
\end{equation*}
Then, by applying the Gronwall inequality in \eqref{KNS_lin5} we get 
$
\|n_\epsilon\|_X\le C(\epsilon).$ In a similar way we get
$
\|w_\epsilon\|_X\le C(\epsilon).
$\\
\\
{Finally, we will prove that $\Gamma:X\times X\times X\rightarrow X\times X\times X$ is continuous.} 
{Let $[\bar{n}^m_\epsilon,\bar{w}^m_\epsilon,\bar{s}^m_\epsilon]_{m\in \mathbb{N}}\subset X\times X\times X$ be such that
$
[\bar{n}^m_\epsilon,\bar{w}^m_\epsilon,\bar{s}^m_\epsilon]\to [\bar{n}_\epsilon,\bar{w}_\epsilon,\bar{s}_\epsilon] \ \text{in} \ X\times X\times X\ \mbox{as}\ m\rightarrow\infty.
$
In particular $[\bar{n}^m_\epsilon,\bar{w}^m_\epsilon,\bar{s}^m_\epsilon]_{m\in \mathbb{N}}$ is bounded in $X\times X\times X,$ and from \eqref{KNS_lin03}-\eqref{KNS_lin05}, $[n^m_\epsilon,w^m_\epsilon,s^m_\epsilon]=\Gamma[\bar{n}^m_\epsilon,\bar{w}^m_\epsilon,\bar{s}^m_\epsilon]$ is bounded in $\mathcal{X}_{5/3}\times\mathcal{X}_{5/3}\times\mathcal{X}_{5/3}$. Then there exist a subsequence, still denoted by $[n^m_\epsilon,w^m_\epsilon,s^m_\epsilon]_{m\in \mathbb{N}}$, and $[\hat{n}_\epsilon,\hat{w}_\epsilon,\hat{s}_\epsilon]\in \mathcal{X}_{5/3}\times\mathcal{X}_{5/3}\times\mathcal{X}_{5/3}$ such that
\begin{equation}
\label{KNS_lin7}
[n^m_\epsilon,w^m_\epsilon,s^m_\epsilon]\to[\hat{n}_\epsilon,\hat{w}_\epsilon,\hat{s}_\epsilon], \ \text{weakly in } \mathcal{X}_{5/3}\times\mathcal{X}_{5/3}\times\mathcal{X}_{5/3} \ \text{and strongly in } X\times X\times X.
\end{equation}
The equality $[n^m_\epsilon,w^m_\epsilon,s^m_\epsilon]=\Gamma[\bar{n}^m_\epsilon,\bar{w}^m_\epsilon,\bar{s}^m_\epsilon]$ says that
\begin{equation}\label{KNS_lin8}
\left\{\begin{array}{lc}
\partial_t n_{\epsilon}^m-\Delta n_\epsilon^m=-{\bf u}_\epsilon^m \cdot\nabla \bar{n}^m_\epsilon-\chi_1 \nabla\cdot\left(\frac{ \bar{n}^{m+}_\epsilon}{1+\epsilon \bar{n}^{m+}_\epsilon}\nabla \bar{c}^m_\epsilon\right)+\mu_1 \bar{n}^{m+}_\epsilon(1-\bar{n}^m_\epsilon-a_1 \bar{w}^m_\epsilon)\ \mbox{in}\ \Omega\times(0,T),&\\
\partial_t w^m_{\epsilon}-\Delta \bar{w}^m_\epsilon=-{\bf u}_\epsilon^m \cdot\nabla \bar{w}^m_\epsilon-\chi_2 \nabla\cdot\left(\frac{ \bar{w}^{m+}_\epsilon}{1+\epsilon \bar{w}^{m+}_\epsilon}\nabla \bar{c}^m_\epsilon\right)+\mu_2\bar{w}^{m+}_\epsilon(1-a_2\bar{n}^m_\epsilon-\bar{w}^m_\epsilon)\ \mbox{in}\ \Omega\times(0,T),&\\
\partial_t s^m_{\epsilon}-\Delta \bar{s}^m_\epsilon=-{{\bf u}_\epsilon^m} \cdot\nabla\bar{s}^m_\epsilon-\alpha (\bar{s}^{m+}_\epsilon) \frac{1}{\epsilon}\ln(1+\epsilon \bar{n}^{m+}_\epsilon)-\beta(\bar{s}^{m+}_\epsilon)\frac{1}{\epsilon}\ln(1+\epsilon \bar{w}^{m+}_\epsilon)\ \mbox{in}\ \Omega\times(0,T),  &\\
\partial_t {\bf u}^m_{\epsilon}-\Delta {\bf u}^m_\epsilon+(Y_\epsilon{\bf u}^m_\epsilon\cdot\nabla){\bf u}^m_\epsilon-\nabla\pi^m_\epsilon=(\gamma \bar{n}^m_\epsilon+\lambda \bar{w}^m_\epsilon) \nabla \phi\ \mbox{in}\ \Omega\times(0,T), &\\
\nabla\cdot {\bf u}^m_\epsilon=0\ \mbox{in}\ \Omega\times(0,T),\end{array}
\right.
\end{equation}
and 
\begin{equation}
\label{KNS_lin9}
\left\{
\begin{array}{lc}
\bar{c}^m_\epsilon-\epsilon\Delta\bar{c}^m_\epsilon=\bar{s}^m_\epsilon\ \mbox{in}\ \Omega,\\
\frac{\partial\bar{c}^m_\epsilon(x,t)}{\partial \nu}=0\ \mbox{on}\ \partial\Omega.
\end{array}
\right.
\end{equation}
Testing $\eqref{KNS_lin8}_4$ by ${\bf u}^m_\epsilon\in\mathcal{X}_2$ and using properties of the Yosida approximation (cf. \cite{Sohr}, Chapter 5), we can obtain 
\begin{equation}
\label{KNS_lin13}
\|{\bf u}^m_\epsilon\|_{\mathcal{X}_2}\le C(\epsilon).
\end{equation}
Again, noting that $c^m_\epsilon\in L^\infty(H^2)\cap L^2(H^3)$ and taking into account the strong estimates for $n^m_\epsilon$,${w}^m_\epsilon,$ ${s}^m_\epsilon$ in $\mathcal{X}_{5/3}$  and ${\bf u}^m_\epsilon$ in $\mathcal{X}_{2}$, we can pass to the limit in \eqref{KNS_lin8}, obtaining}
\begin{equation*}
\left\{\begin{array}{lc}
\partial_t \hat{n}_{\epsilon}-\Delta\hat{n}_\epsilon=-{\bf u}_\epsilon \cdot\nabla \bar{n}_\epsilon-\chi_1 \nabla\cdot\left(\frac{ \bar{n}^+_\epsilon}{1+\epsilon \bar{n}^+_\epsilon}\nabla \bar{c}_\epsilon\right)+\mu_1 \bar{n}^+_\epsilon(1-\bar{n}_\epsilon-a_1 \bar{w}_\epsilon)\ \mbox{in}\ \Omega\times(0,T),&\\
\partial_t \hat{w}_{\epsilon}-\Delta \hat{w}_\epsilon=-{\bf u}_\epsilon \cdot\nabla \bar{w}_\epsilon-\chi_2 \nabla\cdot\left(\frac{ \bar{w}^+_\epsilon}{1+\epsilon \bar{w}^+_\epsilon}\nabla \bar{c}_\epsilon\right)+\mu_2\bar{w}^+_\epsilon(1-a_2\bar{n}_\epsilon-\bar{w}_\epsilon)\ \mbox{in}\ \Omega\times(0,T),&\\
\partial_t \hat{s}_{\epsilon}-\Delta \hat{s}_\epsilon=-{{\bf u}_\epsilon} \cdot\nabla\bar{s}_\epsilon-\alpha \bar{s}^+_\epsilon \frac{1}{\epsilon}\ln(1+\epsilon \bar{n}^+_\epsilon)-\beta\bar{s}^+_\epsilon \frac{1}{\epsilon}\ln(1+\epsilon \bar{w}^+_\epsilon)\ \mbox{in}\ \Omega\times(0,T),  &\\
\partial_t {\bf u}_{\epsilon}-\Delta {\bf u}_\epsilon+(Y_\epsilon{\bf u}_\epsilon\cdot\nabla){\bf u}_\epsilon-\nabla\pi_\epsilon=(\gamma \bar{n}_\epsilon+\lambda \bar{w}_\epsilon) \nabla \phi\ \mbox{in}\ \Omega\times(0,T), &\\
\nabla\cdot {\bf u}_\epsilon=0\ \mbox{in}\ \Omega\times(0,T),\end{array}
\right.
\end{equation*}
where $\bar{c}_\epsilon$ is the unique solution of  (\ref{elip}) with right hand side $\bar{s}_\epsilon.$ Namely, $[\hat{n}_\epsilon,\hat{w}_\epsilon,\hat{s}_\epsilon]=\Gamma[\bar{n}_\epsilon,\bar{w}_\epsilon,\bar{s}_\epsilon]$. This implies the continuity of $\Gamma$. Gathering the previous information we conclude that $\Gamma$ satisfies the hypotheses of the Leray-Schauder fixed point theorem. Consequently $\Gamma$ has a fixed point $[n_\epsilon,w_\epsilon,s_\epsilon]$, that is, $\Gamma[n_\epsilon,w_\epsilon,s_\epsilon]=[n_\epsilon,w_\epsilon,s_\epsilon]$, which gives a solution of system \eqref{KNS_aprox}. The uniqueness follows a standard procedure, therefore, we omit it.
\\
\end{proof}

\subsubsection{Uniform estimates.}
In this section we find suitable estimates, independent of $\epsilon$, which allow pass to the limit in \eqref{KNS_aprox}-\eqref{elip}, as $\epsilon$ goes to zero.
\begin{lemm}\label{lemma1}
	Let $[{n}_\epsilon,w_\epsilon,s_\epsilon,{\bf u}_\epsilon,\pi_\epsilon],$ ${n}_\epsilon\geq 0, w_\epsilon\geq0, s_\epsilon\geq0,$ be a strong solution of (\ref{KNS_aprox}). Then, $s_\epsilon\in \mathcal{X}_{5/3}$ satisfying
	\begin{equation}
	\label{max1}
	\begin{cases}
\partial_t	s_{\epsilon}-\Delta s_\epsilon=-{\bf u}_\epsilon\cdot\nabla s_\epsilon-\alpha s_\epsilon\frac{1}{\epsilon}\ln(1+\epsilon n_\epsilon)-\beta s_\epsilon\frac{1}{\epsilon}\ln(1+\epsilon w_\epsilon),\ \mbox{in}\ \Omega\times(0,T),\\
	\frac{\partial s_\epsilon}{\partial\nu}=0,\ \mbox{on}\ \partial\Omega\times(0,T),\\
	s_{\epsilon}(0)=s_{0,\epsilon}\in {W}^{1,q}(\Omega),\ q>3,
	\end{cases}
	\end{equation}
verifies that
\begin{eqnarray}\label{mp}
0\leq s_\epsilon(\x,t)\le C.
\end{eqnarray}
\end{lemm}
\begin{proof}
From the proof of  Proposition \ref{p1} we know that $s_\epsilon\geq0.$ The assertion $s_\epsilon(\x,t)\le C$ can be obtained from the parabolic comparison principle (e.g. \cite{lankeit2016long}, Lemma 2.3.) \end{proof}

Next, observe that $s\mapsto ks-\mu s^2$, $s \in [0,\infty)$ and 
$s\mapsto (ks-\frac{\mu}{2}s^2)\ln s$, $s \in (0,\infty)$ are bounded from above by some constant $C$; thus using these estimates and testing $\eqref{KNS_aprox}_1$ and $\eqref{KNS_aprox}_2$ by $\ln n_\epsilon$ and $\ln w_\epsilon$ respectively, integrating on $\Omega$, (recalling that $n_\epsilon^+=n_\epsilon$ and $w_\epsilon^+=w_\epsilon$) we get
{\begin{align}\label{ch2}
\frac{d}{dt}\int_{\Omega}n_\epsilon\ln n_\epsilon&=-\int_\Omega\frac{|\nabla n_\epsilon|^2}{n_\epsilon}+\chi_1\int_\Omega\frac{\nabla n_\epsilon\cdot\nabla c_\epsilon}{1+\epsilon n_\epsilon}+\mu_1\int_\Omega n_\epsilon\ln n_\epsilon -\mu_1\int_\Omega n_\epsilon^2 \ln n_\epsilon\notag\\
&\phantom{= }-a_1\mu_1\int_\Omega  n_\epsilon w_\epsilon\ln n_\epsilon +\mu_1\int_\Omega n_\epsilon-\mu_1\int_\Omega n_\epsilon^2 -a_1\mu_1\int_\Omega  n_\epsilon w_\epsilon\notag\\
&\le-\int_\Omega\frac{|\nabla n_\epsilon|^2}{n_\epsilon}+\chi_1\int_\Omega\frac{\nabla n_\epsilon\cdot\nabla c_\epsilon}{1+\epsilon n_\epsilon}-\frac{\mu_1}{2}\int_\Omega n_\epsilon^2 \ln n_\epsilon \leq-a_1\mu_1\int_\Omega  n_\epsilon w_\epsilon\ln n_\epsilon + C,
\end{align}
and analogously
{\begin{align}\label{ch3}
\frac{d}{dt}\int_{\Omega}w_\epsilon\ln w_\epsilon&\le-\int_\Omega\frac{|\nabla w_\epsilon|^2}{w_\epsilon}+\chi_2\int_\Omega\frac{\nabla w_\epsilon\cdot\nabla c_\epsilon}{1+\epsilon w_\epsilon}-\frac{\mu_2}{2}\int_\Omega w_\epsilon^2 \ln w_\epsilon{\leq  }-a_2\mu_2\int_\Omega  n_\epsilon w_\epsilon\ln w_\epsilon + C.
\end{align}}
Now we estimate $\frac{d}{dt}\int_\Omega\frac{|\nabla s_\epsilon|^2}{s_\epsilon}$, for any $\epsilon>0$ on  $(0,T)$. By using $\eqref{KNS_aprox}_3$, and following Lemma 2.8 in \cite{lankeit2016long} we can obtain
\begin{multline}\label{ctr6}
\frac{d}{dt}\int_\Omega\frac{|\nabla s_\epsilon|^2}{s_\epsilon}+k_1\int_\Omega s_\epsilon|D^2\ln s_\epsilon|^2+\frac{k_1}{2}\int_\Omega\frac{|\nabla s_\epsilon|^4}{s_\epsilon^3} \le -2\alpha\int_\Omega\frac{\nabla c_\epsilon\cdot\nabla n_\epsilon}{1+\epsilon n_\epsilon} +k_2\int_\Omega s_\epsilon \\
+2\alpha\epsilon\int_\Omega\frac{\nabla(\Delta c_\epsilon)\cdot\nabla n_\epsilon}{1+\epsilon n_\epsilon}+Ck_3\int_\Omega |\nabla {\bf u}_\epsilon|^2
-2\beta\int_\Omega\frac{\nabla c_\epsilon\cdot\nabla w_\epsilon}{1+\epsilon w_\epsilon}+2\beta\epsilon\int_\Omega\frac{\nabla(\Delta c_\epsilon)\cdot\nabla w_\epsilon}{1+\epsilon w_\epsilon},
\end{multline}
where we have used uniform estimate to $s_\epsilon$ provided by Lemma \ref{lemma1}. {Note that
\begin{align}\label{ctr7}
2\alpha\epsilon\!\int_\Omega\frac{\nabla(\Delta c_\epsilon)\cdot\nabla n_\epsilon}{1+\epsilon n_\epsilon}=2\alpha\epsilon^{1/2}\!\int_\Omega\frac{\nabla(\Delta c_\epsilon)\cdot\nabla n_\epsilon}{n_\epsilon^{1/2}}\frac{(\epsilon n_\epsilon)^{1/2}}{1+\epsilon n_\epsilon} 
\le \delta_1\int_\Omega\frac{|\nabla n_\epsilon|^2}{n_\epsilon}+C_{\delta_1}\epsilon\alpha^2\int_\Omega|\nabla(\Delta c_\epsilon)|^2
\end{align}}
{and
\begin{align}\label{ctr8}
2\beta\epsilon\!\int_\Omega\!\frac{\nabla(\Delta c_\epsilon)\cdot\nabla w_\epsilon}{1+\epsilon w_\epsilon}=2\beta\epsilon^{1/2}\!\int_\Omega\!\frac{\nabla(\Delta c_\epsilon)\cdot\nabla w_\epsilon}{w_\epsilon^{1/2}}\frac{(\epsilon w_\epsilon)^{1/2}}{1+\epsilon w_\epsilon}  
\le \delta_2\!\int_\Omega\! \frac{|\nabla w_\epsilon|^2}{w_\epsilon}+C_{\delta_2}\epsilon\beta^2\!\int_\Omega|\nabla(\Delta c_\epsilon)|^2.
\end{align}
{From $\eqref{elip}$ it holds $\epsilon\nabla(\Delta c_\epsilon)=\nabla c_\epsilon-\nabla s_\epsilon$}. {Then, replacing \eqref{ctr7} and \eqref{ctr8} in \eqref{ctr6}, we get
\begin{multline}\label{ctr9}
\frac{d}{dt}\int_\Omega\frac{|\nabla s_\epsilon|^2}{s_\epsilon}+k_1\int_\Omega s_\epsilon|D^2\ln s_\epsilon|^2+\frac{k_1}{2}\int_\Omega\frac{|\nabla s_\epsilon|^4}{s_\epsilon^3} 
\le -2\alpha\int_\Omega\frac{\nabla c_\epsilon\cdot\nabla n_\epsilon}{1+\epsilon n_\epsilon}-2\beta\int_\Omega\frac{\nabla c_\epsilon\cdot\nabla w_\epsilon}{1+\epsilon w_\epsilon} \\
+\delta_1\int_\Omega\frac{|\nabla n_\epsilon|^2}{n_\epsilon}+\delta_2\int_\Omega\frac{|\nabla w_\epsilon|^2}{w_\epsilon}
+k_2\int_\Omega s_\epsilon+Ck_3\int_\Omega|\nabla u_\epsilon|^2+2\int_\Omega\vert \nabla c_\epsilon\vert^2+2\int_\Omega\vert \nabla s_\epsilon\vert^2.
\end{multline}}}
On the other hand, multiplying $\eqref{KNS_aprox}_4$ by ${\bf u}_\epsilon$, and integrating in $\Omega$, we have
\begin{align*}
\frac{1}{2}\frac{d}{dt}\|{\bf u}_\epsilon\|_{L^2}^2+\|\nabla {\bf u}_\epsilon\|_{L^2}^2=((\gamma n_\epsilon+\lambda w_\epsilon)\nabla\phi,{\bf u}_\epsilon).
\end{align*}
Furthermore, by the Gagliardo--Niremberg inequality we get
\begin{align*}
&\hspace{-0.4cm}(\gamma n_\epsilon\nabla\phi,{\bf u}_\epsilon)+(\lambda w_\epsilon\nabla\phi,{\bf u}_\epsilon)\le \gamma C\|n_\epsilon\|_{L^{6/5}}\|{\bf u}_\epsilon\|_{L^6}+\lambda C\|w_\epsilon\|_{L^{6/5}}\|{\bf u}_\epsilon\|_{L^6}  \\
&\le \gamma C\|n_\epsilon^{1/2}\|_{L^{12/5}}^2\|\nabla {\bf u}_\epsilon\|_{L^2} +\lambda C\|w_\epsilon^{1/2}\|_{L^{12/5}}^2\|\nabla {\bf u}_\epsilon\|_{L^2}\\
&\le \delta_3\|\nabla {\bf u}_\epsilon\|_{L^2}^2+\gamma^2 C_{\delta_3}\|n_\epsilon^{1/2}\|_{L^{6/5}}^4+\lambda^2 C_{\delta_3}\|w_\epsilon^{1/2}\|_{L^{6/5}}^4 \\
&\le \delta_3\|\nabla {\bf u}_\epsilon\|_{L^2}^2+\gamma^2C_{\delta_3}(\|n_\epsilon^{1/2}\|_{L^2}^{3/2}\|\nabla n_\epsilon^{1/2}\|_{L^2}^{1/2}+\|n_\epsilon^{1/2}\|_{L^2}^2)^2\\
&\ \ +\lambda^2C_{\delta_3}(\|w_\epsilon^{1/2}\|_{L^2}^{3/2}\|\nabla w_\epsilon^{1/2}\|_{L^2}^{1/2}+\|w_\epsilon^{1/2}\|_{L^2}^2)^2 \\
&\le \delta_3\|\nabla {\bf u}_\epsilon\|_{L^2}^2+\gamma^2C_{\delta_3}\|n_\epsilon\|_{L^1}(\|n_\epsilon\|_{L^1}^2\|\nabla n_\epsilon^{1/2}\|_{L^2}^{1/2}+1)^2 \\
&\ \ +\lambda^2C_{\delta_3}\|w_\epsilon\|_{L^1}(\|w_\epsilon\|_{L^1}^2\|\nabla w_\epsilon^{1/2}\|_{L^2}^{1/2}+1)^2\\
&\le \delta_3\|\nabla {\bf u}_\epsilon\|_{L^2}^2+\gamma^2C_{\delta_3}(\|\nabla n_\epsilon^{1/2}\|_{L^2}+1)+\lambda^2C_{\delta_3}(\|\nabla w_\epsilon^{1/2}\|_{L^2}+1).
\end{align*}
Thus, we conclude that there are $\delta_3,C>0$ such that for every $\epsilon>0$
\begin{equation}\label{ctr10}
\frac{1}{2}\frac{d}{dt}\int_\Omega|{\bf u}_\epsilon|^2+(1-\delta_3)\int_\Omega|\nabla {\bf u}_\epsilon|^2\le C_{\delta_3}\left(\gamma^2\int_\Omega\frac{|\nabla n_\epsilon|^2}{n_\epsilon}+\lambda^2\int_\Omega\frac{|\nabla w_\epsilon|^2}{w_\epsilon}\right)+C, \quad \text{on} \ (0,T).
\end{equation}
Then, multiplying \eqref{ch2} and \eqref{ch3} by a suitable constants to cancel terms with \eqref{ctr9}, multiplying \eqref{ctr10} by a constant if necessary, and using Lemma \ref{lemma1} we obtain} 
{\begin{multline}\label{ctr12}
\frac{d}{dt}\left[\int_{\Omega}n_\epsilon\ln n_\epsilon+\int_{\Omega}w_\epsilon\ln w_\epsilon+C\int_\Omega\frac{|\nabla s_\epsilon|^2}{s_\epsilon}+C\int_\Omega|{\bf u}_\epsilon|^2\right]+C\int_\Omega\frac{|\nabla n_\epsilon|^2}{n_\epsilon}  \\
+C\int_\Omega\frac{|\nabla w_\epsilon|^2}{w_\epsilon}+C\int_\Omega s_\epsilon|D^2\ln s_\epsilon|^2+C\int_\Omega\frac{|\nabla s_\epsilon|^4}{s_\epsilon^3}+C\int_\Omega|\nabla {\bf u}_\epsilon|^2 \\
+C\int_\Omega n_\epsilon^2 \ln n_\epsilon
+C\int_\Omega  n_\epsilon w_\epsilon\ln n_\epsilon+C\int_\Omega w_\epsilon^2 \ln w_\epsilon
+C\int_\Omega  n_\epsilon w_\epsilon\ln w_\epsilon\\
\le C\int_\Omega s_\epsilon+\int_\Omega\vert \nabla c_\epsilon\vert^2+\int_\Omega\vert \nabla s_\epsilon\vert^2+C
\le \int_\Omega\vert \nabla c_\epsilon\vert^2+\int_\Omega\vert \nabla s_\epsilon\vert^2+C.
\end{multline}
Multiplying \eqref{elip} by $c_\epsilon$ and $-\Delta c_\epsilon$ we get
\begin{equation}
\label{KNS_lin15}
\frac{1}{2}\|c_\epsilon\|_{L^2}^2+\epsilon\|\nabla c_\epsilon\|_{L^2}^2\le\frac{1}{2}\|s_\epsilon\|_{L^2}^2,
\end{equation}
and
\begin{equation}
\label{KNS_lin16}
\frac{1}{2}\|\nabla c_\epsilon\|_{L^2}^2+\epsilon\|\Delta c_\epsilon\|_{L^2}^2\le\frac{1}{2}\|\nabla s_\epsilon\|_{L^2}^2.
\end{equation}
Recalling that  $\|\nabla s_\epsilon\|_{L^2(L^2)}\le C$ uniformly in $\epsilon$, from (\ref{KNS_lin16}) we have $\|\nabla c_\epsilon\|_{L^2(L^2)}\le C$ uniformly in $\epsilon.$ Moreover, $\sqrt{\epsilon}\Delta c_\epsilon$ is bounded in $L^2(L^2)$ uniformly in $\epsilon$. Thus, integrating in time (\ref{ctr12}), in particular, we obtain that
\begin{equation*}
\left\|{s_\epsilon^{-1/2}\nabla s_\epsilon}{}\right\|_{L^\infty(L^2)}\le C,
\end{equation*}
and since $\int_\Omega|\nabla s_\epsilon|^2=\int_\Omega\frac{|\nabla s_\epsilon|^2}{s_\epsilon}s_\epsilon$ and $s_\epsilon$ is uniformly bounded (Lemma \ref{lemma1}), we get
\begin{equation}\label{tr1}
\|\nabla s_\epsilon\|_{L^\infty(L^2)}\le C,
\end{equation}
uniformly in $\epsilon$.} {Also, from (\ref{ctr12}), Lemma \ref{lemma1} and the identity $\Delta s_\epsilon=s_\epsilon\Delta (\ln s_\epsilon)+\frac{|\nabla s_\epsilon|^2}{s_\epsilon}$, we have
\begin{equation}\label{tr1c}
\|\Delta s_\epsilon\|_{L^2(L^2)}\le C,
\end{equation}
uniformly in $\epsilon$. In addition, from \eqref{KNS_lin15}--\eqref{tr1} we have
\begin{equation*}
\{c_\epsilon\} \ \text{ is bounded in } L^\infty(H^1), \ \{\sqrt{\epsilon}\Delta c_\epsilon\} \ \text{ is bounded in } L^\infty(L^2).
\end{equation*}
Furthermore, computing $\Delta$ in \eqref{elip}, and testing the resulting equation by $\Delta c_\epsilon$, we get
\begin{equation*}
\frac{1}{2}\|\Delta c_\epsilon\|_{L^2}^2+\epsilon\|\nabla\Delta c_\epsilon\|_{L^2}^2\le\frac{1}{2}\|\Delta s_\epsilon\|_{L^2}^2.
\end{equation*}
Therefore, by \eqref{tr1c} it holds $c_\epsilon$ is bounded in $L^2(H^2),$ and $\{\sqrt{\epsilon}\Delta c_\epsilon\}$ is bounded in $L^2(H^1).$ Thus we have
\begin{equation}\label{ctr13}
\begin{cases}
\{c_\epsilon\} \ \text{ is bounded in } L^\infty(H^1)\cap L^2(H^2)\hookrightarrow L^{10}(L^{10}), \\
\{\sqrt{\epsilon}\Delta c_\epsilon\} \ \text{ is bounded in } L^\infty(L^2)\cap L^2(H^1).
\end{cases}
\end{equation}
Since $2\nabla\sqrt{n_\epsilon}=\frac{\nabla n_\epsilon}{\sqrt{n_\epsilon}}$, by \eqref{ctr12}, $\nabla\sqrt{n_\epsilon}\in L^2(L^2)$. Also $\sqrt{n_\epsilon}\in L^2(H^1)\hookrightarrow L^2(L^6)$. Moreover, from step three in the proof of  Proposition \ref{p1} we know that $n_\epsilon\in L^\infty(L^1),$ which implies $\sqrt{n_\epsilon}\in L^\infty(L^2)$. Thus,
\begin{equation}\label{ctr14}
\begin{cases}
\{\sqrt{n_\epsilon}\} \ \text{ is bounded in } L^\infty(L^2)\cap L^2(L^6)\hookrightarrow L^{10/3}(L^{10/3}), \\
\{\nabla\sqrt{n_\epsilon}\} \ \text{ is bounded in } L^2(L^2).
\end{cases}
\end{equation}
Since $\nabla n_\epsilon=2\sqrt{n_\epsilon}\nabla\sqrt{n_\epsilon},$ by \eqref{ctr14} we have
\begin{equation}\label{ctr16}
\{n_\epsilon\} \ \text{ is bounded in } L^{5/4}(W^{1,5/4}).
\end{equation}
Moreover, by $\eqref{ctr14}_1$ it holds
\begin{equation}\label{ctr17}
\{n_\epsilon\} \ \text{ is bounded in } L^{5/3}(L^{5/3}).
\end{equation}
Analogously we get 
\begin{equation}\label{ctr170}
\{w_\epsilon\} \ \text{ is bounded in } L^{5/4}(W^{1,5/4})\cap L^{5/3}(L^{5/3}).
\end{equation}
On the other hand, from \eqref{ctr12} we get
\begin{equation}\label{ctr15}
\{{\bf u}_\epsilon\} \ \text{ is bounded in } L^\infty(L^2)\cap L^2(V)\hookrightarrow L^{10/3}(L^{10/3}),
\end{equation}
and from (\ref{tr1}) and (\ref{tr1c}) we obtain
\begin{equation}\label{ctr15_1}
\{s_\epsilon\} \ \text{ is bounded in } L^\infty(H^1)\cap L^2(H^2)\hookrightarrow L^{10}(L^{10}).
\end{equation}

\subsubsection{The limit as $\epsilon\to 0$}

{Notice that from \eqref{elip} and $\eqref{ctr13}_2$ we obtain
\begin{equation}\label{ctr18}
s_\epsilon-c_\epsilon=-\epsilon\Delta c_\epsilon\to 0 \quad\text{as } \epsilon\to 0, \quad \text{in}\ L^\infty(L^2)\cap L^2(H^1).
\end{equation}}
{Therefore, from $\eqref{ctr13}_1$ and \eqref{ctr14}--\eqref{ctr18}, there exist limit functions $[n,w,c,{\bf u}]$ such that
 for some subsequence of $\{[n_\epsilon,w_\epsilon,c_\epsilon,{\bf u}_\epsilon,s_\epsilon]\}_{\epsilon>0}$, denoted in the same way, the following convergences hold, as $\epsilon\to 0$,
\begin{equation}\label{ctr19}
\begin{cases}
n_\epsilon\rightharpoonup n \quad \text{weakly in } L^{5/3}(L^{5/3})\cap L^{5/4}(W^{1,5/4}), \\
w_\epsilon\rightharpoonup w \quad \text{weakly in } L^{5/3}(L^{5/3})\cap L^{5/4}(W^{1,5/4}), \\
c_\epsilon\rightharpoonup c \quad \text{weakly in } L^2(H^2) \ \text{and weakly* in } L^\infty(H^1), \\
{\bf u}_\epsilon\rightharpoonup {\bf u} \quad \text{weakly in } L^2(V) \ \text{and weakly* in } L^\infty(L^2), \\
{s_\epsilon\rightharpoonup c \quad \text{weakly in } L^2(H^1) \ \text{and weakly* in } L^\infty(L^2).}
\end{cases}
\end{equation}}
Now, by $\eqref{ctr13}$, \eqref{ctr17}, \eqref{ctr15} we can get
	\begin{equation}
	\label{ctr24}
	\begin{cases}
	\partial_tn_\epsilon\rightharpoonup \partial_tn\quad \text{weakly in } L^{10/9}((W^{1,10})'), \\
	\partial_tw_\epsilon\rightharpoonup \partial_tw\quad \text{weakly in } L^{10/9}((W^{1,10})'), \\
	\partial_ts_\epsilon\rightharpoonup \partial_tc\quad \text{weakly in } L^{5/3}(L^{5/3}), \\
	\partial_t{\bf u}_\epsilon\rightharpoonup \partial_t{\bf u}\quad \text{weakly in } L^{5/3}((W^{1,5/2})').
	\end{cases}
	\end{equation}}

{From (\ref{ctr14}), \eqref{ctr16}, \eqref{ctr15}, \eqref{ctr15_1}, \eqref{ctr24}, taking into account that $W^{1,5/4}\overset{c}{\hookrightarrow}L^2\hookrightarrow (W^{1,10})'$, $H^2\overset{c}{\hookrightarrow}H^1\hookrightarrow L^{5/3}$ and $H^1\overset{c}{\hookrightarrow}L^2\hookrightarrow L^{5/3}$, and using the Aubin-Lions lemma and the Simon compactness theorem, we have
\begin{equation}
\label{ctr25}
\begin{cases}
n_\epsilon\to n\quad \text{strongly in } L^{5/4}(L^2)\cap C([0,T];(W^{1,10})'), \\
w_\epsilon\to w\quad \text{strongly in } L^{5/4}(L^2)\cap C([0,T];(W^{1,10})'), \\
s_\epsilon\to c\quad \text{strongly in } L^2(H^1)\cap C([0,T];L^2),\\
{\bf u}_\epsilon\to {\bf u}\quad \text{strongly in } L^2(L^2).
\end{cases}
\end{equation}}

{By \eqref{ctr13}$_1$, $\{\nabla c_\epsilon\}$ is bounded in $L^{10/3}(L^{10/3}),$ which, joint with (\ref{ctr17}) implies that
\begin{equation}
\label{ctr26}
\frac{n_\epsilon}{1+\epsilon n_\epsilon}\nabla c_\epsilon\rightharpoonup \zeta \quad \text{weakly in } L^{10/9}(L^{10/9}),
\end{equation}
for some $\zeta\in L^{10/9}(L^{10/9}).$}
From $\eqref{ctr19}_3$ we have $\nabla c_\epsilon\rightharpoonup\nabla c$ weakly in $L^{10/3}(L^{10/3})$. Also from \eqref{ctr25}$_1$, $\{n_\epsilon\}_{\epsilon>0}$ is relatively compact in $L^{5/4}(L^{5/4})$. Then, taking into account \eqref{ctr17}, and the interpolation inequality
\begin{equation*}
\|n_\epsilon\|_{L^p(L^p)}\le\|n_\epsilon\|_{L^{5/4}(L^{5/4})}^{1-\theta}\|n_\epsilon\|_{L^{5/3}(L^{5/3})}^\theta, \ \theta=4-(5/p)\in(0,1),
\end{equation*}
(which leads to $\frac{5}{4}<p<\frac{5}{3}$), we obtain that $\{n_\epsilon\}$ is relatively compact in $L^p(L^p)$. Since $L^{5/3}(L^{5/3})\hookrightarrow L^{20/13}(L^{20/13})$ and noting that
\begin{align*}
\int_{\Omega}(n_\epsilon\nabla c_\epsilon-n\nabla c)\varphi&=
\int_{\Omega}n_\epsilon(\nabla c_\epsilon-\nabla c)\varphi
+\int_{\Omega}(n_\epsilon-n)\nabla c_\epsilon\varphi \\
&\le \int_{\Omega}n_\epsilon(\nabla c_\epsilon-\nabla c)\varphi
+\|n_\epsilon-n\|_{L^{20/13}}\|\nabla c_\epsilon\|_{L^{10/3}}\|\varphi\|_{L^{20}} \underset{\epsilon\to 0}{\longrightarrow}0,
\end{align*}
{we can conclude, from \eqref{ctr26}, that $\zeta=n\nabla c$. That is
\begin{equation}
\label{ctr26_1}
\frac{n_\epsilon}{1+\epsilon n_\epsilon}\nabla c_\epsilon\rightharpoonup n\nabla c \quad \text{weakly in } L^{10/9}(L^{10/9}).
\end{equation}}

{In a similar way, by $\eqref{ctr19}_4$, noting that $\{n_\epsilon {\bf u}_\epsilon\}$ is bounded in $L^{10/9}(L^{10/9}),$ and  using that $\{n_\epsilon\}_{\epsilon>0}$ is relatively compact in $L^p(L^p)$ for all $p<5/3,$ we get
\begin{equation}
\label{ctr27}
n_\epsilon {\bf u}_\epsilon\rightharpoonup n{\bf u}\quad \text{weakly in } L^{10/9}(L^{10/9}).
\end{equation}}
On the other hand, by using the Dunford-Pettis theorem (see details in  \cite[ Proposition 2.1]{lankeit2016long}), we get
\begin{equation}
\label{ctr261}
n^2_\epsilon\rightarrow n^2 \quad \text{in } L^{1}(L^1)\ \mbox{and}\  n_\epsilon w_\epsilon\rightarrow nw \quad \text{in } L^{1}(L^1).
\end{equation}
Therefore, taking the limit in the regularized problem $\eqref{KNS_aprox}_1$, as $\epsilon\to 0$, by $\eqref{ctr19}_1$, \eqref{ctr26_1}-\eqref{ctr261}, the limit $n$ satisfies the weak formulation for the $n$-equation in Definition \ref{weak}; which also holds  for $w$.  On the other hand, from \eqref{ctr19}, using that $\{{\bf u}_\epsilon s_\epsilon\}$ is bounded in $L^{10/3}(L^{10/3})$ and $\eqref{ctr25}_3$ we get
\begin{equation}
\label{ctr27_1}
{\bf u}_\epsilon s_\epsilon\rightharpoonup c{\bf u}\quad \text{weakly in } L^{5/3}(L^{5/3}).
\end{equation}
Now, since $\frac{\ln(1+\epsilon n_\epsilon)}{\epsilon}\le n_\epsilon$ and $\frac{\ln(1+\epsilon w_\epsilon)}{\epsilon}\le w_\epsilon$, by  \eqref{ctr17} and \eqref{ctr15_1} we have
\begin{equation}
\label{ctr31d}
\left\{s_\epsilon\frac{\ln(1+\epsilon n_\epsilon)}{\epsilon}\right\},\ \left\{s_\epsilon\frac{\ln(1+\epsilon w_\epsilon)}{\epsilon}\right\} \ \text{ are bounded in } L^{10/7}(L^{10/7}).
\end{equation}
By \eqref{ctr17} and the dominated convergence theorem 
\begin{align*}
	&\left\|\frac{\ln(1+\epsilon n_\epsilon)}{\epsilon}-n\right\|_{L^{5/3}(L^{5/3})} \\
	&\le \left\|\frac{\ln(1+\epsilon n_\epsilon)}{\epsilon}-\frac{\ln(1+\epsilon n)}{\epsilon}\right\|_{L^{5/3}(L^{5/3})}+\left\|\frac{\ln(1+\epsilon n)}{\epsilon}-n\right\|_{L^{5/3}(L^{5/3})} \underset{\epsilon\to 0}{\longrightarrow}0.
\end{align*}
Then, by the weak convergence in $\eqref{ctr19}_3$, the strong convergence above and \eqref{ctr31d}, we get
\begin{equation}
\label{ctr32}
s_\epsilon\frac{\ln(1+\epsilon n_\epsilon)}{\epsilon}\rightharpoonup cn \quad \text{weakly in } L^{10/7}(L^{10/7}).
\end{equation}
Analogously, we get
\begin{equation}
\label{ctr320}
s_\epsilon\frac{\ln(1+\epsilon w_\epsilon)}{\epsilon}\rightharpoonup cw \quad \text{weakly in } L^{10/7}(L^{10/7}).
\end{equation}
Therefore, from (\ref{ctr24}), (\ref{ctr25}), (\ref{ctr27_1}), and (\ref{ctr32})-(\ref{ctr320}) we get
\begin{eqnarray}\label{weak_c}
\int_0^T\langle \partial_tc,\varphi\rangle+\int_0^T\int_\Omega \nabla c\cdot \nabla  \varphi-\int_0^T\int_\Omega c{\bf u}\cdot\nabla\varphi-\alpha\int_0^T\int_\Omega cn\varphi-\beta\int_0^T\int_\Omega cw\varphi=0,
\end{eqnarray}
for all $\varphi\in L^{10/3}(W^{1,10/3})$.  Observe that (\ref{weak_c}) is a weak formulation for the $c$-equation of \eqref{KNS}. Thus, integrating by parts in (\ref{weak_c}) and using that $n,w\in L^{5/3}(L^{5/3}),$ $c\in L^2(H^2)\cap L^\infty(H^1)$ and ${\bf u}\in L^2(H^1)\cap L^\infty(L^2)$, we get
\begin{equation}\label{ctr34}
\partial_t c+{\bf u}\cdot\nabla c=\Delta c-\alpha cn-\beta cw, \quad\text{in} \ L^{10/7}(L^{10/7}).
\end{equation}

Finally, to pass to the limit in the $u_\epsilon$-equation, we argue as usual, recalling that $\{Y_\epsilon u_\epsilon\}$ is uniformly bounded in $L^2(L^6),$ and
\begin{eqnarray*}
Y_\epsilon {\bf u}_\epsilon\otimes {\bf u}_\epsilon-{\bf u}\otimes {\bf u}=Y_\epsilon {\bf u}_\epsilon\otimes ({\bf u}_\epsilon-{\bf u})+Y_\epsilon ({\bf u}_\epsilon-{\bf u})\otimes {\bf u}_\epsilon+(Y_\epsilon-I) {\bf u}\otimes {\bf u}.
\end{eqnarray*}
Also, from \eqref{elip1}, $\left[n_{0,\epsilon}(\x),w_{0,\epsilon}(\x),s_{0,\epsilon}(\x),{\bf u}_{0,\epsilon}(\x)\right]\to [n_0,w_0,c_0,{\bf u}_0]$ as $\epsilon\to 0$ and thus, by \eqref{ctr25}, $n(0)=n_0$, $w(0)=w_0$, $c(0)=c_0$, ${\bf u}(0)={\bf u}_0.$ To conclude the proof of Theorem \ref{teor2} it remains to see the non-negativity of $n,c$ and $w,$ but it follows 
from the non-negativity of the regularized solutions. 

\subsection{Regularity of weak solutions. Proof of Theorem \ref{teor3}.}

In this section we prove the regularity criterion established in Theorem \ref{teor3}.
\begin{proof}
{The proof is based on a bootstrapping argument following the ideas of \cite{Guillen2, JC-EJ} combined with the theory of parabolic regularity, interpolation and embedding properties of Sobolev spaces. From Theorem \ref{teor2}, there exists a weak solution $[n,w,c,{\bf u}]$ of system \eqref{KNS}  in the sense of Definition \ref{weak}. The regularity is obtained into several steps in which, we use Theorem \ref{Feiresl}.  Below diagram we skecth the main steps in the proof.}\\
\\
{\it{Step 1}}: From Theorem 4.2.1 in \cite{casas}, to see that if ${\bf u}\in L^8(L^4)$, then ${\bf u}\in\mathcal{X}_2\hookrightarrow L^{10}(L^{10})$.  Now, starting with the regularity $n,w\in L^{20/7}(L^{20/7}), c\in \mathcal{X}_2,$ we can get  $c\in\mathcal{X}_{20/7}$ has showed in the next diagram.
\begin{center}
\begin{tikzpicture}[node distance=5mm,line width=1.0pt,
cuadro/.style={rectangle, draw=black!60, fill=blue!5, very thick, minimum size=5mm},
]
\node (B0)  {\ };
\node[cuadro]        (Ba)       [right=of B0] {$n,w\in L^{20/7}(L^{20/7}),$\  \ \newline $c\in \mathcal{X}_2$};
\node[cuadro]        (Bb)       [right=of Ba] {$c\in \mathcal{X}_{20/9}$};
\draw [->] (Ba) --  (Bb) ;

\node[cuadro]        (Bc)       [right=of Bb] {$c\in L^{20}(L^{20})$};
\draw [->] (Bb) --  (Bc) ;

\node[cuadro]        (Bd)       [right=of Bc] {$c\in \mathcal{X}_{5/2}$};
\draw [->] (Bc) --  (Bd) ;

\node[cuadro]        (Be)       [below=of Bd] {$c\in L^{\infty}(L^{\infty})\cap L^5(W^{1,5})$};
\draw [->] (Bd) --  (Be) ;

\node[cuadro]        (Bf)       [left=of Be] {$c\in \mathcal{X}_{20/7}$};
\draw [->] (Be) --  (Bf) ;

\end{tikzpicture}
\end{center}

{Observing that $cn$ and $cw$ cannot be improved beyond $L^{20/7}(L^{20/7}),$ we cannot go further in the regularity of $c$. Thus, we first improve the regularity of $n$,$w$, and then go back later again on the regularity of $c.$}\\

{\it{Step 2}}: We can prove that $n,w \in L^\infty(L^2)\cap L^2(H^1),$ as showed in the next diagram.
\begin{center}
\begin{tikzpicture}[node distance=5mm,line width=1.0pt,
cuadro/.style={rectangle, draw=black!60, fill=blue!5, very thick, minimum size=5mm},
]

\node[cuadro]        (Ca)       [below=of Ba] {$n,w\in L^{20/7}(L^{20/7})\cap L^{5/4}(W^{1,5/4})$};
\node[cuadro]        (Cb)       [right=of Ca] {$n,w\in \mathcal{X}_{20/19}$};
\draw [->] (Ca) --  (Cb) ;

\node[cuadro]        (Cc)       [right=of Cb] {$n,w\in \mathcal{X}_{10/9}$};
\draw [->] (Cb) --  (Cc) ;

\node[cuadro]        (Cd)       [right=of Cc] {$n,w\in \mathcal{X}_{20/17}$};
\draw [->] (Cc) --  (Cd) ;

\node[cuadro]        (Ce)       [below=of Cd] {$n,w\in \mathcal{X}_{5/4}$};
\draw [->] (Cd) --  (Ce) ;

\node[cuadro]        (Cf)       [left=of Ce] {$n,w\in \mathcal{X}_{4/3}$};
\draw [->] (Ce) --  (Cf) ;

\node[cuadro]        (Cg)       [left=of Cf] {$n,w\in \mathcal{X}_{10/7}$};
\draw [->] (Cf) --  (Cg) ;

\node[cuadro]        (Dg)       [below=of Cg] {$n,w\in \mathcal{X}_{20/13}$};
\draw [->] (Cg) --  (Dg) ;

\node[cuadro]        (Eg)       [right=of Dg] {$n,w\in L^{\infty}(L^2)\cap L^2(H^1)$};
\draw [->] (Dg) --  (Eg) ;

\end{tikzpicture}
\end{center}

{\it{Step 3}}:\ Finally, we can prove that $c,n,w\in\mathcal{X}_4$ as indicated in the diagram below.

\begin{center}
\begin{tikzpicture}[node distance=5mm,line width=1.0pt,
cuadro/.style={rectangle, draw=black!60, fill=blue!5, very thick, minimum size=5mm},
]

\node[cuadro]        (Da)       [below=of Ca] {$n,w \in L^{\infty}(L^{2})\cap L^2(H^1)$};
\node[cuadro]        (Db)       [right=of Da] {$c\in \mathcal{X}_{10/3}$};
\draw [->] (Da) --  (Db) ;

\node[cuadro]        (Dc)       [right=of Db] {$n,w\in \mathcal{X}_{5/3}$};
\draw [->] (Db) --  (Dc) ;

\node[cuadro]        (Dd)       [right=of Dc] {$c\in \mathcal{X}_{4}$};
\draw [->] (Dc) --  (Dd) ;

\node[cuadro]        (De)       [right=of Dd] {$n,w\in \mathcal{X}_{20/11}$};
\draw [->] (Dd) --  (De) ;

\node[cuadro]        (Df)       [below=of De] {$n,w\in \mathcal{X}_{2}$};
\draw [->] (De) --  (Df) ;

\node[cuadro]        (Ef)       [left=of Df] {$n,w\in \mathcal{X}_{5/2}$};
\draw [->] (Df) --  (Ef) ;

\node[cuadro]        (Ff)       [left=of Ef] {$n,w\in \mathcal{X}_{20/7}$};
\draw [->] (Ef) --  (Ff) ;

\node[cuadro]        (Gf)       [left=of Ff] {$n,w\in \mathcal{X}_{4}$};
\draw [->] (Ff) --  (Gf) ;
\end{tikzpicture}
\end{center}
Thus we conclude the proof.
\end{proof}

\section{Fully discrete numerical approximation}\label{SS3}
In this section,  we construct a  numerical approximation for the weak solutions of the system (\ref{KNS}). With this aim, first we introduce an equivalent system obtained by considering an auxiliary variable; and then, we consider a finite element semi-implicit approximation for this equivalent formulation. We prove  the well-posedness and some uniform estimates for the discrete solutions.

\subsection{An equivalent system}
In order to control numerically the second order nonlinear terms in the equations that describe the dynamics of the species and prove optimal error estimates in the convergence  analysis, we will consider the auxiliary variable $\mathbf{s} = \nabla c$. Thus, we obtain the following mixed variational form for the variables $n, w, c, \mathbf{s},\mathbf{u}$ and $\pi$:  
\begin{equation}\label{Chemoweak3}
\left\{
\begin{array}{lc} \vspace{0.1 cm}
(\partial_t n,\bar{n}) + D_n(\nabla n,\nabla\bar{n}) + (\mathbf{u}\cdot\nabla n,\bar{n})=\chi_1(n\mathbf{s},\nabla\bar{n})+\mu_1(n(1-n-a_1 w),\bar{n}),&\\
(\partial_t w,\bar{w}) + D_w(\nabla w,\nabla\bar{w}) + (\mathbf{u}\cdot\nabla w,\bar{w})=\chi_2(w\mathbf{s},\nabla\bar{w})+\mu_2(w(1-a_2 n-w),\bar{w}),&\\
(\partial_t\mathbf{s},\bar{\mathbf{s}})+D_c (\mbox{rot}(\mathbf{s}),\mbox{rot}(\bar{\mathbf{s}}))+D_c (\nabla\cdot\mathbf{s},\nabla\cdot\bar{\mathbf{s}})=(\mathbf{u}\cdot\mathbf{s}+(\alpha n+\beta w)c,\nabla\cdot\bar{\mathbf{s}}),& \\
(\partial_t c,\bar{c}) + D_c (\nabla c,\nabla\bar{c}) +(\mathbf{u}\cdot\nabla c,\bar{c})=-((\alpha n+\beta w)c,\bar{c}),&\\
(\partial_t\mathbf{u},\bar{\mathbf{u}})+D_{\mathbf{u}}(\nabla \mathbf{u},\nabla\bar{\mathbf{u}}) +k((\mathbf{u}\cdot\nabla)\mathbf{u},\bar{\mathbf{u}})=( \pi,\nabla\cdot\overline{\mathbf{u}})+((\gamma n+\lambda w)\nabla\phi,\bar{\mathbf{u}}),& \\
(\nabla \cdot {\mathbf{u}},\bar{\pi}) =0,&
\end{array}
\right.
\end{equation}
for all $[\bar{n},\bar{w},\bar{\bf s},\bar{c},\bar{\bf u},\bar{\pi}]\in H^1(\Omega)\times H^1(\Omega)\times H_{\mathbf{s}}^1(\Omega)\times H^1(\Omega) \times H_0^1(\Omega)\times  {L}_0^2(\Omega)$, where the initial condition considered for the variable  ${\bf s}$ is ${\bf s}(\x,0)={\bf s}_0(\x) =\nabla c_0(\x)$. 
Notice that (\ref{Chemoweak3})$_3$ has been obtained taking the gradient operator to equation (\ref{KNS})$_3$, multiplying the resulting equation by $\bar{\bf s}\in  H^1_{\bf s}(\Omega)$, integrating by parts in some terms and adding 
 $(\mbox{rot}(\mathbf{s}),\mbox{rot}(\bar{\mathbf{s}}))$ (using that $\mbox{rot}({\bf s})=0$).

We will prove that (\ref{Chemoweak3}) defines a variational formulation of  (\ref{KNS}). Observe that if $[{n},{w},{c},{\mathbf{u}},{\pi}]$ is a classical solution of (\ref{KNS}), considering $\mathbf{s} = \nabla c$ and making previous procedure, we can conclude that $[{n},{w},{c},\mathbf{s},{\mathbf{u}},{\pi}]$ satisfies (\ref{Chemoweak3}). Now, if $[{n},{w},{c},\mathbf{s},{\mathbf{u}},{\pi}]$ is a smooth solution of (\ref{Chemoweak3}), integrating by parts equation (\ref{Chemoweak3})$_4$ we get
\begin{eqnarray}\label{ceq}
\partial_t c-D_c \Delta c+ \mathbf{u}\cdot \nabla c= -(\alpha n+\beta w)c \quad \ a.e.\ \mbox{in}\ \Omega.
\end{eqnarray}
Again, taking the gradient operator on both sides of (\ref{ceq}), testing by $\bar{\bf{s}}\in  H_{\bf{s}}^1(\Omega),$ subtracting the resulting equation  from (\ref{Chemoweak3})$_3$ and denoting ${\boldsymbol{\rho}}=\nabla c-{\bf{s}}$, we arrive at
$$
(\partial_t{\boldsymbol{\rho}},\bar{\bf{s}})
+\, D_c(\nabla \cdot \boldsymbol{\rho},\nabla \cdot \bar{\bf{s}})
+\, D_c(\mbox{rot} ({\boldsymbol{\rho}}),\mbox{rot} (\bar{\bf{s}}))
-({\bf u} \cdot \boldsymbol{\rho},\nabla \cdot \bar{\bf{s}})=0, \quad \forall \bar{\bf{s}} \in H_{\bf{s}}^1(\Omega),
$$
from which, taking $\bar{\bf s}=\boldsymbol{\rho}\in  H_{\bf{s}}^1(\Omega)$ and using (\ref{EQs}) and (\ref{in3D}), we deduce that
$$
\begin{array}{rcl}
\displaystyle\frac{1}{2} \displaystyle\frac{d}{dt} \Vert \boldsymbol{\rho} \Vert_{L^2}^2
+ \, 
D_c\Vert \nabla \cdot \boldsymbol{\rho} \Vert_{L^2}^2
+D_c\Vert \mbox{rot} ({\boldsymbol{\rho}}) \Vert_{L^2}^2
&\!\!\!\! = \!\!\!\!& ({\bf u} \cdot \boldsymbol{\rho}, \nabla \cdot \boldsymbol{\rho} )
\le \Vert {\bf u} \Vert_{L^6} \Vert \boldsymbol{\rho} \Vert_{L^3} \Vert \nabla \cdot \boldsymbol{\rho} \Vert_{L^2}\\
&\!\!\!\! \le \!\!\!\! & \displaystyle\frac{D_c}{4} \Vert \nabla \cdot \boldsymbol{\rho} \Vert_{L^2}^2 +
\frac{C}{D_c}
\Vert {\bf u} \Vert_{L^6}^2 \Vert \boldsymbol{\rho} \Vert_{L^2}\Vert \boldsymbol{\rho} \Vert_{H^1}\\
& \!\!\!\! \le \!\!\!\! & \displaystyle\frac{D_c}{2} \left(
\Vert \nabla \cdot {\boldsymbol{\rho}} \Vert_{L^2}^2
+\Vert \mbox{rot} ({\boldsymbol{\rho}}) \Vert_{L^2}^2 
\right) + C(1+

\Vert {\bf u} \Vert_{L^6}^4) \Vert {\boldsymbol{\rho}} \Vert_{L^2}^2.
\end{array}$$
Taking into account that ${\boldsymbol \rho}(0)=0$, we deduce that ${\boldsymbol \rho}=0$ and therefore ${\bf s}=\nabla c$.
 Finally, replacing ${\bf s} = \nabla c$ in (\ref{Chemoweak3})$_{1,2}$ and integrating by parts once again, we get that $[{n},w,{c},{\mathbf{u}},{\pi}]$ is a smooth solution of  (\ref{KNS}). 
 
\subsection{Numerical scheme}\label{NScheme}
For the construction of the numerical approximations, we  will use finite element approximations in space and finite differences in time (considered for simplicity on a uniform partition of $[0,T]$ with time step
$\Delta t=T/N : (t_{m} = {m}\Delta t)_{{m}=0}^{{m}=N}$), combined with splitting ideas to decouple the computation of the fluid part from competition with chemotaxis one. Moreover, in order to deal with the velocity trilinear form and the nonlinear convective terms, we will use the skew-symmetric forms $A$ and $B$ given in (\ref{a6})-(\ref{a6aa}).  With respect to the spatial discretization, we consider conforming FE spaces:
$N_h\times   W_h \times C_h  \times {{\Sigma}}_h \times{ U}_h \times \Pi_h  \subset H^1(\Omega)\times {H}^1(\Omega)\times H^1(\Omega)\times H^1_{\bf s}(\Omega)\times H^1_0(\Omega)\times {L}^2_0(\Omega)$ corresponding to a family of shape-regular and quasi-uniform triangulations of $\overline{\Omega}$, $\{\mathcal{T}_h\}_{h>0}$,  made up of simplexes $K$ (triangles in 2D and 
tetrahedra en 3D), so that $\overline{\Omega}= \cup_{K\in \mathcal{T}_h} K$, where $h = \max_{K\in \mathcal{T}_h} h_K$, with $h_K$ being the diameter of $K$. 

We assume that ${ U}_h$ and $\Pi_h $ satisfy the following discrete {\it inf-sup} condition: There exists a constant $\beta>0$, independent of $h$, such that
\begin{equation}\label{LBB}
\sup_{\mathbf{v}\in { U}_h\setminus\{0\}} \frac{-(g,\nabla \cdot \mathbf{v}) }{\|\mathbf{v}\|_{{\bf U}_h}} \geq \beta \|g\|_{\Pi_h}, \quad \forall g
\in \Pi_h.
\end{equation}
\begin{remark}{\bf (Some possibilities for the choice of the discrete spaces)} 
For the spaces $[{ U}_h, \Pi_h]$, we can choose the Taylor-Hood approximation $\mathbb{P}_{r}\times \mathbb{P}_{r-1}$ (for $r\geq 2$) (\cite{GR,Sten}); or the approximation $\mathbb{P}_1-bubble \times \mathbb{P}_1$ (\cite{GR}) (for $r=1$). On the other hand, the spaces 
$[N_h, W_h , C_h ,{{\Sigma}}_h]$ are approximated by $\mathbb{P}_{r_1}\times\mathbb{P}_{r_2}\times\mathbb{P}_{r_3}\times\mathbb{P}_{r_4}$- continuous FE, with $r_i\geq 1$ ($i=1,...,4$).
\end{remark}

Then, we can define the well-known Stokes operator $(\mathbb{P}_{\mathbf{u}},\mathbb{P}_\pi):{H}^1_0(\Omega)\times L^2_0(\Omega)\rightarrow { U}_h\times\Pi_h$ such that  $[\mathbb{P}_{\mathbf{u}}\mathbf{u},\mathbb{P}_\pi \pi]\in {U}_h\times\Pi_h$ satisfies
\begin{equation}\label{StokesOp}
\left\{
\begin{array}{lc} \vspace{0.1 cm}
D_{\mathbf{u}}(\nabla
(\mathbb{P}_{\mathbf{u}}{\mathbf{u}}  -{\mathbf{u}}), \nabla\bar{\mathbf{u}}) - (\mathbb{P}_\pi \pi - \pi, \nabla \cdot \bar{\mathbf{u}})=0,\quad \forall \bar{\mathbf{u}}\in { U}_h,& \\
(\nabla\cdot
(\mathbb{P}_{\mathbf{u}}{\mathbf{u}}  -{\mathbf{u}}), \bar{\pi})=0, \quad \forall \bar{\pi}\in \Pi_h,&
\end{array}
\right.
\end{equation}
and the following approximation and stability properties hold (\cite{GV}):
\begin{equation}\label{StabStk1}
\| [\mathbf{u}-\mathbb{P}_ {\mathbf{u}}\mathbf{u},\pi-\mathbb{P}_{\pi} \pi]\|_{ H^1\times L^2} + \frac{1}{h}\| \mathbf{u}-\mathbb{P}_ {\mathbf{u}}\mathbf{u}\|_{L^2} \leq K h^r \|[\mathbf{u},\pi]\|_{H^{r+1}\times H^{r}},
\end{equation}
\begin{equation}\label{StabStk2}
\|[\mathbb{P}_ {\mathbf{u}}\mathbf{u},\mathbb{P}_{\pi} \pi]\|_{W^{1,6}\times L^6} \leq C \|[\mathbf{u},\pi]\|_{H^{2}\times H^1}.
\end{equation}
Moreover, we consider the following interpolation operators:  
$$
\mathbb{P}_n:H^1(\Omega)\rightarrow N_h, \quad
\mathbb{P}_w:H^1(\Omega)\rightarrow  W_h, \quad \mathbb{P}_c:H^1(\Omega)\rightarrow  C_h ,\quad  \mathbb{P}_{\boldsymbol \sigma}:{H}^1_{\boldsymbol{s}}(\Omega)\rightarrow{{\Sigma}}_h,
$$
such that for all $n\in H^1(\Omega)$, $w\in H^1(\Omega)$, $c\in H^1(\Omega)$ and ${\boldsymbol{s}}\in {H}^1_{\bf{s}}(\Omega)$, $\mathbb{P}_n n \in N_h$, $\mathbb{P}_w w \in W_h$, $\mathbb{P}_c c\in C_h$ and $\mathbb{P}_{\boldsymbol{s}} {\boldsymbol {s}}\in {{\Sigma}}_h$ satisfy
\begin{equation}\label{Interp1New}
\left\{
\begin{array}{lc} \vspace{0.1 cm}
(\nabla
(\mathbb{P}_n n - n), \nabla\bar{n})+(\mathbb{P}_n n - n,\bar{n})=0,\quad \forall \bar{n}\in N_h,& \\
\vspace{0.1 cm}
(\nabla
(\mathbb{P}_w w - w), \nabla\bar{w})+(\mathbb{P}_w w - w,\bar{w})=0,\quad \forall \bar{w}\in W_h,& \\
\vspace{0.1 cm}
(\nabla
(\mathbb{P}_c c - c), \nabla\bar{c})  + (\mathbb{P}_c c - c, \bar{c})=0,\quad \forall \bar{c}\in C_h,& \\
\vspace{0.1 cm}
(\nabla \cdot (\mathbb{P}_{\boldsymbol{s}} {\boldsymbol{s}} - {\boldsymbol{s}}),\nabla \cdot \bar{\boldsymbol{s}}) + (\mbox{rot}(\mathbb{P}_{\boldsymbol{s}} {\boldsymbol{s}}-{\boldsymbol{s}}),\mbox{rot}(\bar{\boldsymbol{s}})) + (\mathbb{P}_{\boldsymbol {s}} {\boldsymbol{s}} - {\boldsymbol{s}},\bar{\boldsymbol{s}})=0,\ \ \forall \bar{\boldsymbol{s}}\in {{\Sigma}}_h.&
\end{array}
\right.
\end{equation}
From the Lax-Milgram Theorem, we have that the interpolation operators $\mathbb{P}_n$, $\mathbb{P}_w$, $\mathbb{P}_c$ and $\mathbb{P}_{\boldsymbol{s}}$ are well defined. Moreover,  the following interpolation errors and stability properties hold
\begin{equation}\label{aprox01} 
\left\{
\begin{array}{lc}
\| n-\mathbb{P}_n n\|_{L^2}+ h \|  n-\mathbb{P}_n n  \|_{H^1} \leq K h^{r_1+1} \|n\|_{H^{r_1+1}}, & \forall n \in H^{r_1+1}(\Omega),\\[.2cm]
\| w-\mathbb{P}_w w\|_{L^2}+ h \|  w-\mathbb{P}_w w  \|_{H^1} \leq K h^{r_2+1} \|w\|_{H^{r_2+1}}, & \forall w \in H^{r_2+1}(\Omega),\\[.2cm]
\| c-\mathbb{P}_c c\|_{L^2}+ h \|  c-\mathbb{P}_c c  \|_{H^1} \leq K h^{r_3+1} \|c\|_{H^{r_3+1}}, & \forall c \in H^{r_3+1}(\Omega),\\[.2cm]
 \Vert {\bf s}- \mathbb{P}_{\bf s} {\bf s}  \Vert_{L^2} + h \Vert {\bf s} - \mathbb{P}_{\bf s} {\bf s} \Vert_{H^1} \leq Ch^{r_4 + 1} \Vert  {\bf s}\Vert_{H^{r_4 +1}}, &  \forall  {\bf s}\in {H}^{r_4 +1}(\Omega),
\end{array}
\right.
\end{equation} 
\begin{equation}\label{aprox01-aNN}
\| [\mathbb{P}_n n, \mathbb{P}_w w, \mathbb{P}_c c, \mathbb{P}_{\boldsymbol s} {\boldsymbol{s}}] \|_{H^{1}} \leq  \| [n,w,c, {\boldsymbol{s}}] \|_{H^1},
\end{equation}
\begin{equation}\label{aprox01-a}
\| [\mathbb{P}_n n, \mathbb{P}_w w, \mathbb{P}_c c, \mathbb{P}_{\boldsymbol s} {\boldsymbol{s}}] \|_{W^{1,6}} \leq C \| [n,w,c, {\boldsymbol{s}}] \|_{H^2}.
\end{equation}
We also consider the following skew-symmetric trilinear forms which will be used in the formulation of the numerical scheme:
\begin{eqnarray}
B(\mathbf{v}_1,\mathbf{v}_2,\mathbf{v}_3)&=&\frac{1}{2}\Big[\Big( (\mathbf{v}_1\cdot \nabla)\mathbf{v}_2,\mathbf{v}_3 \Big) - \Big( (\mathbf{v}_1\cdot \nabla)\mathbf{v}_3,\mathbf{v}_2 \Big)\Big], \ \forall \mathbf{v}_1,\mathbf{v}_2,\mathbf{v}_3 \in {H}^1(\Omega),\label{a6}\\
A(\mathbf{v},w_1,w_2)&=&\frac{1}{2}\Big[\Big( (\mathbf{v}\cdot \nabla)w_1,w_2 \Big) - \Big( (\mathbf{v}\cdot \nabla)w_2,w_1 \Big)\Big], \ \forall \mathbf{v} \in {H}^1(\Omega), \ w_1,w_2\in H^1(\Omega), \label{a6aa}
\end{eqnarray}
which satisfy the following properties
\begin{eqnarray}
B(\mathbf{v}_1,\mathbf{v}_2,\mathbf{v}_3)\!&=&\!\Big( (\mathbf{v}_1\cdot \nabla)\mathbf{v}_2,\mathbf{v}_3 \Big), \ \ \forall \mathbf{v}_1\in {V} , \ \ \mathbf{v}_2,\mathbf{v}_3  \in  {H}^1(\Omega), \label{aA1} \\
A(\mathbf{v}_1,w_1,w_2)
&=&\!\Big( (\mathbf{v}_1\cdot \nabla)w_1,w_2 \Big), \ \ \forall \mathbf{v}_1\in {V} , \ \ w_1,w_2  \in  {H}^1(\Omega),\label{aA2} \\
B(\mathbf{v}_1,\mathbf{v}_2,\mathbf{v}_2)&=&0,\ \ \forall \mathbf{v}_1,\mathbf{v}_2\in {H}^1(\Omega), \label{a7} \\
A(\mathbf{v},w,w)&=&0,\ \forall w\in H^1(\Omega),\ \ \mathbf{v}\in {H}^1(\Omega). \label{a8} 
\end{eqnarray}

Then, taking into account  (\ref{Chemoweak3}), we consider the following first order in time, linear and decoupled scheme:\\

{\textbf{Initialization:}  Let $[n^0_{h}, w^0_{h},c^0_h,\mathbf{s}^0_h,\mathbf{u}^0_h]=[\mathbb{P}_{n} n_{0},\mathbb{P}_{w} w_{0}, \mathbb{P}_c c_0, \mathbb{P}_{\mathbf{s}} {\mathbf{s}}_0, \mathbb{P}_{\mathbf{u}} \mathbf{u}_0] \in N_{h} \times W_{h} \times C_h \times {{\Sigma}}_h \times {U}_h$.}  \\

{\textbf{Time step $m$:} Given the vector $[n^{m-1}_{h}, w^{m-1}_{h},c^{m-1}_h,\mathbf{s}^{m-1}_h,\mathbf{u}^{m-1}_h]\in N_{h} \times W_{h} \times C_h \times {{\Sigma}}_h \times { U}_h$, compute $[n^m_{h}, w^m_{h},c^m_h,\mathbf{s}^{m}_h,\mathbf{u}^m_h,\pi_h^m]\in N_{h} \times W_{h} \times C_h \times {{\Sigma}}_h \times {U}_h\times \Pi_h$ such that for each $[\bar{n}, \bar{w},\bar{c},\bar{\mathbf{s}},\bar{\mathbf{u}},\bar{\pi}] \in N_{h} \times W_{h} \times C_h \times {{\Sigma}}_h \times { U}_h\times \Pi_h$ it holds:}
\begin{flalign}\label{scheme1} 
&a) \ \ 
 (\delta_t n^m_{h},\bar{n}) + D_n (\nabla n^m_{h},\nabla\bar{n}) + A(\mathbf{u}^{m-1}_h, n^m_{h},\bar{n})\nonumber\\
 &\hspace{2cm}=  \  \chi_1(n^{m}_{h}\mathbf{s}^{m-1}_{h},\nabla\bar{n})+\mu_1(n^{m}_{h},\bar{n})-\mu_1(n^{m}_{h}[n^{m-1}_{h}]_{+},\bar{n})-\mu_1a_1 (n^{m}_{h}[w^{m-1}_{h}]_{+},\bar{n})   
 ,\nonumber\\
&b) \ \ (\delta_t w^m_{h},\bar{w}) +D_w (\nabla w^m_{h},\nabla\bar{w}) + A(\mathbf{u}^{m-1}_h, w^m_{h},\bar{w})\nonumber\\
&\hspace{2cm}=\chi_2(w^{m}_{h}\mathbf{s}^{m-1}_{h},\nabla\bar{w})+\mu_2(w^{m}_{h},\bar{w})-\mu_2a_2(w^{m}_{h}[n^{m-1}_{h}]_{+},\bar{w})-\mu_2(w^{m}_{h}[w^{m-1}_{h}]_{+},\bar{w}),\nonumber\\
&c) \ \ (\delta_t \mathbf{s}^m_h,\bar{\mathbf{s}})+D_c(\nabla \cdot \mathbf{s}^m_h,\nabla \cdot \bar{\mathbf{s}})+D_c(\mbox{rot}(\mathbf{s}^m_h),\mbox{rot}(\bar{\mathbf{s}}))=(\mathbf{u}^{m}_h\cdot\mathbf{s}^{m-1}_h+(\alpha n^{m}_{h}+\beta w^{m}_{h})c^{m-1}_h,\nabla\cdot\bar{\mathbf{s}}), & \nonumber\\
&d) \ \ (\delta_t c^m_h,\bar{c}) + D_c (\nabla c^m_h,\nabla \bar{c})+A(\mathbf{u}^{m-1}_h,c^m_h,\bar{c})=-(\alpha c^{m}_h[n^{m}_{h}]_{+},\bar{c})-(\beta c^{m}_h[w^{m}_{h}]_{+},\bar{c}), \\
&e) \ \ (\delta_t \mathbf{u}^m_h,\bar{\mathbf{u}})+kB(\mathbf{u}^{m-1}_h,\mathbf{u}^m_h,\bar{\mathbf{u}})+D_{\mathbf{u}}(\nabla \mathbf{u}^m_h,\nabla \bar{\mathbf{u}})-(\pi^m_h,\nabla \cdot \bar{\mathbf{u}}) = ((\gamma n^{m}_{h}+\lambda w^{m}_{h})\nabla\phi,\bar{\mathbf{u}}), \nonumber \\
&f) \ \ (\nabla \cdot {\mathbf{u}_h^m},\bar{\pi}) =0, \nonumber
\end{flalign}
{\noindent where, in general, we denote $\delta_t a^m_h=\frac{a^m_h-a^{m-1}_h}{\Delta t}$ and $[a^{m}_h]_{+} = \mbox{max}\lbrace{0,a^{m}_h}\rbrace$.}\\

\begin{remark}
The skew-symmetric forms $A$ and $B$ verifying the  properties (\ref{aA1})-(\ref{a8}) will be important in order to prove the well-posedness of the scheme (\ref{scheme1}), to get uniform estimates for the discrete solutions and to develop the convergence analysis.
\end{remark}

\subsection{Well-posedness and uniform estimates}
In this section, we analyze the well-posedness of the scheme (\ref{scheme1}) and prove some uniform estimates for the discrete solutions. To this aim, we make the following inductive hypothesis:  there exists a positive constant $K>0$, independent of $m$, such that
\begin{equation}\label{IndHyp}
\Vert [{\mathbf{s}}^{m-1}_h,{c}^{m-1}_h]\Vert_{L^{10/3}}\leq K, \qquad \forall m\geq 1.
\end{equation} 
After the convergence analysis, we verify the validity of (\ref{IndHyp}).

{\begin{theo}{\bf (Well-posedness of (\ref{scheme1}))}\label{WPs}
There exists a constant $C > 0$ (independent of discrete parameters) that if $\Delta t < C$, then the numerical scheme (\ref{scheme1}) is well-posed, that is, there exists a unique $[n^m_{h}, w^m_{h},c^m_h,\mathbf{s}^{m}_h,\mathbf{u}^m_h,\pi_h^m]\in N_{h} \times W_{h} \times C_h \times {{\Sigma}}_h \times {U}_h\times \Pi_h$ solution of the scheme (\ref{scheme1}).
\end{theo}

\begin{proof}
First, in order to show the existence and uniqueness of $[n^m_h,w^m_h]\in N_h\times W_h$ solution of (\ref{scheme1})$_{a),b)}$, it suffices to prove uniqueness (since is an algebraic linear system). Suppose that  there exist $[n^m_{h,1}, w^m_{h,1}], [n^m_{h,2}, w^m_{h,2}]\in N_{h} \times W_{h}$ two possible solutions of  (\ref{scheme1})$_{a),b)}$; then defining $[n^m_{h},w^m_h]=[n^m_{h,1}-n^m_{h,2},w^m_{h,1}-w^m_{h,2}]$,
we have that $[n^m_{h}, w^m_{h}]\in N_{h} \times W_{h}$ satisfies
\begin{eqnarray}
\displaystyle
&(n^m_{h},\bar{n})&\!\!\!+\Delta t D_n (\nabla n^m_{h},\nabla \bar{n})+\Delta t A(\mathbf{u}^{m-1}_h, n^m_{h},\bar{n})=\chi_1 \Delta t (n^m_{h}{\bf s}^{m-1}_{h},\nabla \bar{n})\nonumber\\\vspace{0.1 cm}
&&\!\!\!+\mu_1\Delta t(n^{m}_{h},\bar{n})-\mu_1\Delta t(n^{m}_{h}[n^{m-1}_{h}]_{+},\bar{n})-\mu_1\Delta t a_1 (n^{m}_{h}[w^{m-1}_{h}]_{+},\bar{n}), \ \ \forall \bar{n}\in N_h, \label{modelf02clina}\\ \vspace{0.1 cm}
&(w^m_{h},\bar{w})&\!\!\!+\Delta t D_w (\nabla w^m_{h},\nabla \bar{n})+\Delta t A(\mathbf{u}^{m-1}_{h}, w^m_h,\bar{w}) = \chi_2\Delta t (w^m_{h}{\bf s}^{m-1}_{h},\nabla \bar{w})\nonumber\\ \vspace{0.1 cm}
&&\!\!\!+\mu_2\Delta t(w^{m}_{h},\bar{w})-\mu_2\Delta ta_2(w^{m}_{h}[n^{m-1}_{h}]_{+},\bar{w})-\mu_2\Delta t(w^{m}_{h}[w^{m-1}_{h}]_{+},\bar{w}),\ \ \forall \bar{w}\in W_h.  \label{modelf02clinb}
\end{eqnarray}
Taking $\bar{n}=n^m_{h}$ in (\ref{modelf02clina}), using (\ref{a8}) and taking into account that $(n^{m}_{h}[n^{m-1}_{h}]_{+},n^{m}_{h})=(| n^{m}_{h}|^{2},[n^{m-1}_{h}]_{+})\geq0$ and $(n^{m}_{h}[w^{m-1}_{h}]_{+},n^{m}_{h})=(| n^{m}_{h}|^{2},[w^{m-1}_{h}]_{+})\geq0$ (since $[n^{m-1}_{h}]_{+},[w^{m-1}_{h}]_{+}\geq 0$),  we deduce that
\begin{equation}\label{Nn}
(1-\mu_1\Delta t)\Vert n^m_{h}\Vert_{L^2}^2+\Delta t D_n  \Vert\nabla n^m_{h}\Vert_{L^2}^2 \leq \chi_1\Delta t (n^m_{h}{\bf s}^{m-1}_{h},\nabla n^m_{h}).
\end{equation}
Using the 3D interpolation inequality (\ref{in3Dl4}), the inductive hypothesis (\ref{IndHyp}), and the H\"older and Young inequalities, we bound the right hand side of (\ref{Nn}) in the following way 
\begin{eqnarray*}
(1-\mu_1\Delta t)\Vert n^m_{h}\Vert_{L^2}^2+\Delta t D_n \Vert \nabla n^m_{h}\Vert_{L^2}^2&\!\!\!\!\leq\!\!\!\!&\chi_1\Delta t \Vert n^m_{h}\Vert_{L^5}\Vert \mathbf{s}^{m-1}_{h}\Vert_{L^{10/3}}\Vert \nabla n^m_{h}\Vert_{L^2}\\
&\!\!\!\!\leq\!\!\!\!& \displaystyle\frac{\Delta tD_n }{4}\Vert \nabla n^m_{h}\Vert_{L^2}^{2}+\frac{C(\chi_1)^{2}\Delta t}{D_n}\Vert n^m_{h}\Vert_{L^2}^{1/5}\Vert n^m_{h}\Vert_{H^1}^{9/5}\\
&\!\!\!\!\leq\!\!\!\!&\displaystyle\frac{\Delta tD_n }{2}(\Vert n^m_{h}\Vert_{L^2}^{2}+\Vert \nabla n^m_{h}\Vert_{L^2}^{2})+C_1\Delta t\Vert n^m_{h}\Vert_{L^2}^{2},
\end{eqnarray*}
from which we arrive at
\begin{equation}\label{1condit}
\Big[\frac{1}{2}-\Big(\mu_1 + \frac{D_n}{2}\Big)\Delta t\Big]\Vert n^m_{h}\Vert_{L^2}^2+\frac{1}{2} \Vert n^m_{h}\Vert_{L^2}^2+\frac{\Delta tD_n}{2}\Vert \nabla n^m_{h}\Vert_{L^2}^2\leq C_1\Delta t\Vert n^m_{h}\Vert_{L^2}^{2}.  
\end{equation}
In the same way, we have the next estimate for $w^n_h$
\begin{equation}\label{2condit}
\Big[\frac{1}{2}-\Big(\mu_2 + \frac{D_w}{2}\Big)\Delta t\Big]\Vert w^m_{h}\Vert_{L^2}^2+\frac{1}{2}\Vert w^m_{h}\Vert_{L^2}^2+\frac{\Delta tD_w}{2}\Vert \nabla w^m_{h}\Vert_{L^2}^2\leq C_2\Delta t\Vert w^m_{h}\Vert_{L^2}^{2}.
\end{equation}
Thus, if $\Delta t \leq  1/E$ where $E=2\max\{\mu_1+\frac{D_n}{2}+C_1,\mu_2+\frac{D_w}{2}+C_2 \}$, we conclude that $n^m_h=w^m_h=0$, that is, $n^m_{h,1}=n^m_{h,2}$ and $w^m_{h,1}=w^m_{h,2}$. Now, knowing the existence and uniqueness of $[n^m_h,w^m_h]\in N_h\times W_h$, we have that there exists a unique $[c^m_h,\mathbf{s}^{m}_h,\mathbf{u}^m_h,\pi_h^m]\in  C_h \times {{\Sigma}}_h \times {U}_h\times \Pi_h$ solution of (\ref{scheme1})$_{c)}$-(\ref{scheme1})$_{f)}$. In fact, 
assuming that  there exist $[c^m_{h,1},\mathbf{s}^{m}_{h,1},\mathbf{u}^m_{h,1},\pi_{h,1}^m], [c^m_{h,2},\mathbf{s}^{m}_{h,2},\mathbf{u}^m_{h,2},\pi_{h,2}^m]\in C_h \times {{\Sigma}}_h \times {U}_h\times \Pi_h$ two possible solutions; then defining $c^m_{h}=c^m_{h,1}- c^m_{h,2}$, $\mathbf{s}^{m}_{h}=\mathbf{s}^{m}_{h,1}-\mathbf{s}^{m}_{h,2}$ $\mathbf{u}^m_{h}=\mathbf{u}^m_{h,1}-\mathbf{u}^m_{h,2}$ and $\pi_{h}^m=\pi_{h,1}^m - \pi_{h,2}^m$,
we get that $[c^m_h,\mathbf{s}^{m}_h,\mathbf{u}^m_h,\pi_h^m]\in C_h \times {{\Sigma}}_h \times {U}_h\times \Pi_h$ satisfies
\begin{equation}
\left\{
\begin{array}
[c]{lll}%
\vspace{0.1 cm}
\displaystyle
( \mathbf{s}^m_h,\bar{\mathbf{s}})+\Delta t D_c(\nabla \cdot \mathbf{s}^m_h,\nabla \cdot \bar{\mathbf{s}})+\Delta t D_c (\mbox{rot}(\mathbf{s}^m_h),\mbox{rot}(\bar{\mathbf{s}}))= \Delta t(\mathbf{u}^m_h\cdot \mathbf{s}^{m-1}_h,\nabla\cdot\bar{\mathbf{s}}),\\\vspace{0.1 cm}
\displaystyle
(c^m_h,\bar{c}) + \Delta tD_c (\nabla c^m_h,\nabla \bar{c})+\Delta t A(\mathbf{u}^{m-1}_h,c^m_h,\bar{c})=-\Delta t\alpha(c^{m}_h[n^{m}_{h}]_{+},\bar{c}))-\Delta t\beta(c^{m}_h[w^{m}_{h}]_{+},\bar{c}),\\\vspace{0.1 cm} 
(\mathbf{u}^m_h,\bar{\mathbf{u}})+k\Delta t B(\mathbf{u}^{m-1}_h,\mathbf{u}^m_h,\bar{\mathbf{u}})+\Delta t D_{\mathbf{u}} (\nabla \mathbf{u}^m_h,\nabla \bar{\mathbf{u}})-\Delta t(\pi^m_h,\nabla \cdot \bar{\mathbf{u}}) =0,\\
(\bar{\pi},\nabla \cdot {\mathbf{u}_h^m}) =0,
\end{array}
\right.  \label{modelf02clin}
\end{equation}
for all $[\bar{c},\bar{\mathbf{s}},\bar{\mathbf{u}},\bar{\pi}] \in  C_h \times {{\Sigma}}_h \times {U}_h\times \Pi_h$. Taking $\bar{\bf u}={\bf u}^m_h$ in (\ref{modelf02clin})$_3$ and $\bar{\pi}=\Delta t{\pi}^m_h$ in (\ref{modelf02clin})$_4$, adding the resulting equations and using (\ref{a7}), we deduce that ${\bf u}^m_h={\boldsymbol 0}$; from which, using the discrete {\it inf-sup} condition (\ref{LBB}), we deduce that $\pi^m_h=0$. Moreover, using that ${\bf u}^m_h={\boldsymbol 0}$ in (\ref{modelf02clin})$_1$, and testing by $\bar{\bf s}={\bf s}^m_h$, we get ${\bf s}^m_h={\boldsymbol 0}$. Finally, taking $\bar{c}=c^m_h$ in (\ref{modelf02clin})$_2$, using (\ref{a8}) and the fact that $(c^{m}_{h}[n^{m}_{h}]_{+},c^{m}_{h})=(|c^{m}_{h}|^{2},[n^{m}_{h}]_{+})\geq0$ and $(c^{m}_{h}[w^{m}_{h}]_{+},c^{m}_{h})=(|c^{m}_{h}|^{2},[w^{m}_{h}]_{+})\geq0$ (since $[n^{m}_{h}]_{+},[w^{m}_{h}]_{+}\geq 0$), we conclude that $c^m_h=0$. Thus, using again that (\ref{scheme1})$_{c)}$-(\ref{scheme1})$_{f)}$ is an linear algebraic system, we conclude the proof.
\end{proof}

Now, in order to obtain some uniform estimates for the discrete variables,  we will use the following discrete Gronwall lemma:
 \begin{lemm}\label{Diego2} (\cite[p. 369]{Hey})
 Assume that $\Delta t>0$ and $B,b^k, d^k, g^k, h^k\ge 0$ satisfy:
 	\begin{equation*}\label{e-Diego2-1}
 	d^{m+1}+ \Delta t\sum_{k=0}^{m} b^{k+1}  \le \Delta t \sum_{k=0}^{m} g^k \, d^k +\Delta t\sum_{k=0}^{m} h^k + B,  \quad \forall {m} \ge 0.
 	\end{equation*}
 	Then, it holds
 	\begin{equation*}\label{e-Diego2-2}
 	d^{m+1} + \Delta t \, \sum_{k=0}^{m} b^{k+1} \le \exp \left( \Delta t \, \sum_{k=0}^{m} g^k\right)
 	\, \left( \Delta t \, \sum_{k=0}^{m} h^k + B \right), \quad \forall {m} \ge 0.
 	\end{equation*}
 \end{lemm}
 
Some uniform weak estimates for the discrete  variables $n^m_h$ and $w^m_h$} are necessary in the convergence analysis. They are proved in the following result. 
\begin{lemm}{\bf (Uniform estimates for $n_h^m$ and $w^m_h$)}\label{uen} 
Assume the inductive hypothesis (\ref{IndHyp}). There exists a constant $C$ such that if $\Delta t C\leq \frac{1}{2}$, then $n^m_h$ and $w^m_h$ are bounded in $ l^{\infty}(L^2)\cap l^{2}(H^1)$.
\end{lemm}
\begin{proof}
Testing (\ref{scheme1})$_a$ by $\bar{n}=2\Delta t n^m_h$, using the equality $(a-b,2a)=\vert a\vert^2-\vert b\vert^2+\vert a-b\vert^2$ and  the property (\ref{a8}), we have  
\begin{eqnarray}\label{nh1-p}
&\Vert&\!\!\!\!\!\!n^m_h\Vert^2_{L^2}-\Vert n^{m-1}_h \Vert^2_{L^2}+ \Vert n^m_h - n^{m-1}_h\Vert^2_{L^2}+{2\Delta t D_n }\Vert\nabla n^m_h\Vert^2_{L^2}= 2 \Delta t \chi_1(n^{m}_h{\bf s}^{m-1}_h,\nabla n^m_h)\nonumber\\
&&+ 2\Delta t \mu_1 \Vert n^m_h\Vert_{L^2}^2 -2\Delta t\mu_1(n_h^m[n_h^{m-1}]_+,n_h^m)-2a_1\Delta t\mu_1(n_h^m[w_h^{m-1}]_+,n_h^m),
\end{eqnarray}
and taking into account that the terms $([n_h^{m-1}]_+,\vert n_h^m\vert^2)$ and $([w_h^{m-1}]_+,\vert n_h^m\vert^2)$ are nonnegative, we get
\begin{equation}\label{nh1}
\Vert n^m_h\Vert^2_{L^2}-\Vert n^{m-1}_h \Vert^2_{L^2}+ \Vert n^m_h - n^{m-1}_h\Vert^2_{L^2}+{2\Delta t D_n}\Vert\nabla n^m_h\Vert^2_{L^2}\leq 2 \Delta t \chi_1\vert (n^{m}_h{\bf s}^{m-1}_h,\nabla n^m_h)\vert+ 2\Delta t \mu_1 \Vert n^m_h\Vert_{L^2}^2.
\end{equation}
Using the H\"older and Young inequalities,  the 3D interpolation inequality (\ref{in3Dl4}) and the inductive hypothesis (\ref{IndHyp}), we obtain
\begin{equation}\label{nh2}
2 \Delta t \chi_1\vert (n^{m}_h{\bf s}^{m-1}_h,\nabla n^m_h)\vert \leq  {C\Delta t \chi_1  \Vert n^{m}_h  \Vert_{L^5} \Vert   {\bf s}^{m-1}_h\Vert_{L^{10/3}} \Vert \nabla n^m_h\Vert_{L^2} }  \leq \Delta t D_n \Vert \nabla n^m_h\Vert^2_{L^2}+ C\Delta t  \Vert n^{m}_h \Vert_{L^2}^2.
\end{equation}
Thus, using (\ref{nh2}) in (\ref{nh1}), and adding the resulting expression  from $m=1$ to $m=r$, we arrive at
 \begin{equation}\label{nh3-a}
 \Vert n^r_h\Vert^2_{L^2}
 +\Delta t\sum_{m=1}^r D_n\Vert\nabla n^m_h\Vert^2_{L^2} 
 \leq \|n^0_h\|^2_{L^2}+
 C\Delta t\sum_{m=1}^{r-1} \Vert n^{m}_h \Vert_{L^2}^2+
 C\Delta t \Vert n^{r}_h \Vert_{L^2}^2.
 \end{equation}
Therefore, if $\Delta t$ is small enough such that $\displaystyle \frac{1}{2}-C\Delta t\geq 0$,  applying Lemma \ref{Diego2} to (\ref{nh3-a}) we deduce
\begin{equation*}
\Vert n^r_h\Vert^2_{L^2} +
{\Delta t}\sum_{m=1}^{r}D_n \Vert n^m_h\Vert^2_{H^1}
\leq  C, \ \ \forall m\geq1, 
\end{equation*}
where the constant $C>0$ depends on the data $(\chi_1,\mu_1, n_0,T)$, but is independent of $(\Delta t, h)$ and $r$; and we conclude that $n^m_h$ is bounded in $ l^{\infty}(L^2)\cap l^{2}(H^1)$. The proof for $w^m_h$ follows analogously.
\end{proof}

\section{Error estimates}\label{ESWN}
The aim of this subsection is to obtain optimal error estimates for any solution $[n^m_h,w^m_h, c^m_h,{\bf s}^m_h,\mathbf{u}^m_h,\pi^m_h]$ of the scheme (\ref{scheme1}), with respect to a su\-ffi\-cien\-tly regular solution $[n,w, c,{\bf s},\mathbf{u},\pi]$ of  (\ref{Chemoweak3}). 
We start by introducing the following notations for the errors at $t=t_{m}$:  $e_n^m=n^m-n^m_h$, $e_w^m=w^m-w^m_h$, $e_c^m=c^m-c^m_h$, $e_{\bf s}^m={\bf s}^m-{\bf s}^m_h$, $e_{\mathbf{u}}^m=\mathbf{u}^m-\mathbf{u}^m_h$ and $e_\pi^m=\pi^m-\pi^m_h$, where $a^{m}$ {denotes}, in general, the value of $a$ at time $t_{m}$. Then, subtracting the scheme (\ref{scheme1}) to (\ref{Chemoweak3}) at $t=t_m$, we obtain that $[e_n^m,e_w^m,e^m_c,e_{{\bf s}}^m,e^m_{\mathbf{u}},e^m_\pi]$ satisfies
\begin{eqnarray}\label{errn}
&(\delta_t e_n^m,&\!\!\!\!\bar{n}) + D_n (\nabla e_n^m, \nabla \bar{n})+A(\mathbf{u}^m-\mathbf{u}^{m-1},n^m,\bar{n})  +A(\mathbf{u}^{m-1}_h,e^m_n,\bar{n}) + A(e^{m-1}_{\mathbf{u}},n^m,\bar{n})=(\rho_n^m,\bar{n})\nonumber\\
&&\!\!\!\!\!\!\!\!\!\!\!\!\!\!\!\!\!
+\mu_1(e^m_{n}(1-[n^{m-1}_{h}]_{+}-a_{1}[w^{m-1}_{h}]_{+})-n^{m}(n^m-n^{m-1}+ a_1(w^m-w^{m-1}) + n^{m-1} - [n^{m-1}_h]_+),\bar{n})\nonumber\\
&&\!\!\!\!\!\!\!\!\!\!\!\!\!\!\!\!\! - \mu_1 a_1(n^m (w^{m-1} - [w^{m-1}_h]_+),\bar{n})
+\chi_1(n^{m}(\mathbf{s}^m-\mathbf{s}^{m-1}+e_{\mathbf{s}}^{m-1})+e^{m}_{n}\mathbf{s}^{m-1}_{h},\nabla\bar{n}),
\end{eqnarray}
\begin{eqnarray}\label{errw}
&(\delta_t e_w^m,&\!\!\!\!\bar{w}) + D_w(\nabla e_w^m, \nabla \bar{w})+A(\mathbf{u}^m-\mathbf{u}^{m-1},w^m,\bar{w})  +A(\mathbf{u}^{m-1}_h,e^m_w,\bar{w}) + A(e^{m-1}_{\mathbf{u}},w^m,\bar{w})=(\rho_w^m,\bar{w})\nonumber\\
&&\!\!\!\!\!\!\!\!\!\!\!\!\!\!\!\!\!
+\mu_2(e^m_{w}(1-a_2[n^{m-1}_{h}]_{+}-[w^{m-1}_{h}]_{+})-w^{m}(w^m-w^{m-1} + a_2(n^m-n^{m-1})+ w^{m-1} - [w^{m-1}_h]_+),\bar{w})\nonumber\\
&&\!\!\!\!\!\!\!\!\!\!\!\!\!\!\!\!\! - \mu_2 a_2(w^m (n^{m-1} - [n^{m-1}_h]_+),\bar{w})
+\chi_2(w^{m}(\mathbf{s}^m-\mathbf{s}^{m-1}+e_{\mathbf{s}}^{m-1})+e^{m}_{w}\mathbf{s}^{m-1}_{h},\nabla\bar{w}),
\end{eqnarray}
\begin{eqnarray}\label{errs}
&\left(\delta_t e_{\bf s}^m,\bar{\bf s}\right)&\!\!\!\!+D_c (\nabla \cdot e_{\bf s}^m, \nabla \cdot\bar{\bf s}) + D_c(\mbox{rot}(e_{\bf s}^m), \mbox{rot}(\bar{\bf s}))  = (\rho_{\bf s}^m,\bar{\bf s})+((\mathbf{s}^m-\mathbf{s}^{m-1})\cdot {\bf u}^m, \nabla \cdot\bar{\bf s}) \nonumber\\
&&\!\!\!\!\!\! +  (e^{m-1}_{\mathbf{s}}\cdot{\bf u}^{m}+e^{m}_{\bf u}\cdot\mathbf{s}^{m-1}_h+(\alpha n^{m}+\beta w^{m})(c^m-c^{m-1})+(\alpha n^{m}+\beta w^{m})e_{c}^{m-1},\nabla \cdot\bar{\bf s})\nonumber\\
&&\!\!\!\!\!\! + ((\alpha e_{n}^{m}+\beta e_{w}^{m})c^{m-1}_{h},\nabla \cdot\bar{\bf s}),
\end{eqnarray}
\begin{eqnarray}\label{errc}
&\left(\delta_t e_c^m,\bar{c}\right) &\!\!\!\!+ D_c(\nabla e_c^m, \nabla \bar{c})+A(\mathbf{u}^m-\mathbf{u}^{m-1},c^m,\bar{c})  +A(\mathbf{u}^{m-1}_h,e^m_c,\bar{c}) + A(e^{m-1}_{\mathbf{u}},c^m,\bar{c})=(\rho_c^m,\bar{c})\nonumber\\
&&\hspace{-0.5 cm}  -\alpha(c^{m}(n^m - [n^{m}_h]_+)+e_{c}^{m}[n^{m}_h]_+,\bar{c})
 -\beta(c^{m}(w^m - [w^{m}_h]_+)+e_{c}^{m}[w^{m}_h]_+,\bar{c}),
\end{eqnarray}
\begin{eqnarray}\label{erru}
&\left(\delta_t e_{\mathbf{u}}^m,\bar{\mathbf{u}}\right) &\!\!\!\!+ D_{\mathbf{u}}(\nabla e_{\mathbf{u}}^m, \nabla \bar{\mathbf{u}})= (\rho_{\mathbf{u}}^m,\bar{\mathbf{u}}) - kB(\mathbf{u}^m-\mathbf{u}^{m-1},{\mathbf{u}}^m,\bar{\mathbf{u}})   -kB(e^{m-1}_{\mathbf{u}},{\mathbf{u}}^m,\bar{\mathbf{u}})\nonumber\\
&&\hspace{-0.8 cm} -kB(\mathbf{u}^{m-1}_h,e^m_{\mathbf{u}},\bar{\mathbf{u}}) +(e^m_\pi,\nabla \cdot \bar{\mathbf{u}})+((\gamma e^{m}_{n}+\lambda e^{m}_{w})\nabla\phi,\bar{\mathbf{u}}),
\end{eqnarray}
\begin{equation}\label{errpi}
(\nabla \cdot e^m_{\mathbf{u}},\bar{\pi})=0,
\end{equation}
for all $[\bar{n},\bar{w},\bar{c},\bar{\bf s},\bar{\mathbf{u}},\bar{\pi}]\in N_{h} \times W_{h} \times C_h \times {{\Sigma}}_h \times {U}_h\times \Pi_h$, where $\rho_n^m=\delta_t n^m - (\partial_t n)^m$ and so on. Now, taking into account the interpolation operators $\mathbb{P}_n$, $\mathbb{P}_w$, $\mathbb{P}_c$, $\mathbb{P}_{\bf s}$, $\mathbb{P}_{\mathbf{u}}$ and $\mathbb{P}_{\pi}$ defined in (\ref{StokesOp}) and (\ref{Interp1New}), we decompose the total errors $e_n^m$, $e_w^m$, $e^m_c$, $e_{\bf s}^m$, $e_{\mathbf{u}}^m$ and $e^m_{\pi}$ as follows
\begin{eqnarray}
e_n^m&=&(n^m -\mathbb{P}_n n^m) + ( \mathbb{P}_n n^m - n^m_h )
= \theta^m_n+\xi^m_n, \label{u1an}\\
e_w^m&=&(w^m -\mathbb{P}_w w^m) + ( \mathbb{P}_w w^m - w^m_h )
= \theta^m_w+\xi^m_w, \label{u1aw}\\
e_c^m&=&(c^m -\mathbb{P}_c c^m) + ( \mathbb{P}_c c^m - c^m_h )=\theta^m_c+\xi^m_c,\label{u1ac}\\
e_{\bf s}^m&=&({\bf s}^m -\mathbb{P}_{\bf s} {\bf s}^m) + ( \mathbb{P}_{\bf s} {\bf s}^m - {\bf s}^m_h )=\theta^m_{\bf s}+\xi^m_{\bf s}, \label{u1asig}\\
e_{\mathbf{u}}^m&=&(\mathbf{u}^m -\mathbb{P}_{\mathbf{u}} \mathbf{u}^m) + ( \mathbb{P}_{\mathbf{u}} \mathbf{u}^m - \mathbf{u}^m_h )=\theta^m_{\mathbf{u}}+\xi^m_{\mathbf{u}},\label{u1au}\\
e_\pi^m&=&(\pi^m -\mathbb{P}_\pi \pi^m) + ( \mathbb{P}_\pi \pi^m - \pi^m_h )=\theta^m_\pi+\xi^m_\pi,\label{u1api}
\end{eqnarray}
where, in general, $\theta^m_{a}$ and $\xi^m_{a}$ denote the interpolation and discrete errors with respect to the $a$ variable, respectively. Recalling that the interpolation errors satisfy (\ref{StabStk1}) and (\ref{aprox01}), the aim of this section is to prove the following result:
 \begin{theo}\label{theo1N}
	Assume (\ref{IndHyp}). There exists a constant $C>0$ (depending on the data of the problem (\ref{Chemoweak3})) such that if $\Delta t C \leq  \frac{1}{2}$,  the following estimates for the discrete errors hold
\begin{eqnarray}\label{EEtheo1}
\|[\xi^m_n,\xi^m_w,\xi^m_c,\xi^m_{\mathbf{u}},\xi^m_{\boldsymbol{s}}]\|_{l^{\infty}(L^2)\cap l^2(H^1)}\!\!\!& \leq&\!\!\! C(T) \Big(\Delta t +\max\{h^{r_1+1},h^{r_2+1},h^{r_3+1},h^{r_4+1},
h^{r+1} \}\Big),\\
\|\xi^m_{\mathbf{u}}\|_{l^{\infty}(H^1)\cap l^2(W^{1,6})} + \|\xi^m_{\pi}\|_{l^{2}(L^6)} \!\!
\!& \leq&\!\!\!  C(T) \Big(\Delta t +\max\{h^{r_1+1},h^{r_2+1},h^{r_3+1},h^{r_4+1},
	h^{r}\}\Big),\label{EEtheo2}
\end{eqnarray}
	where the constant $C(T)>0$ is independent of $m, \Delta t$ and $h$.
\end{theo}
From decompositions (\ref{u1an})-(\ref{u1api}), Theorem \ref{theo1N} and interpolation errors (\ref{StabStk1}) and (\ref{aprox01}), we can deduce the following corollary.
\begin{coro}
	Under the assumptions of Theorem \ref{theo1N}, the following estimates for the total errors hold
	\begin{eqnarray*}\label{EEtheo1b}
	\|[e^m_n,e^m_w,e^m_c,e^m_{\mathbf{u}},e^m_{\boldsymbol{s}}]\|_{l^{\infty}(L^2)}\!\!\!& \leq&\!\!\!  C(T) \Big(\Delta t +\max\{h^{r_1+1},h^{r_2+1},h^{r_3+1},h^{r_4+1},
	h^{r+1} \}\Big),\\
		\|[e^m_n,e^m_w,e^m_c,e^m_{\mathbf{u}},e^m_{\boldsymbol{s}}]\|_{l^2(H^1)} \!\!\!& \leq&\!\!\!  C(T) \Big(\Delta t +\max\{h^{r_1},h^{r_2},h^{r_3},h^{r_4},
	h^{r} \}\Big),\\
	\|e^m_{\mathbf{u}}\|_{l^{\infty}(H^1)}  \!\!\!& \leq&\!\!\!  C(T) \Big(\Delta t +\max\{h^{r_1+1},h^{r_2+1},h^{r_3+1},h^{r_4+1},
	h^{r}\}\Big),
	\end{eqnarray*}
	where the constant $C(T)>0$ is independent of $m, \Delta t$ and $h$.
\end{coro}

\subsection{Preliminary error estimates}
In this subsection, we will obtain some bounds for the discrete errors of the species densities $n$ and $w$, the chemical concentration $c,$ the flux ${\bf s}$ and the velocity $\mathbf{u}.$
\\

{\underline{{\it 1. Error estimates for the species densities $n$ and $w$}}}\\

Using (\ref{Interp1New})$_1$,  (\ref{u1an}) and (\ref{u1asig})-(\ref{u1au}) in (\ref{errn}), we have
\begin{eqnarray}\label{errn-int}
&(\delta_t \xi_n^m,&\!\!\!\!\bar{n}) + D_n(\nabla \xi_n^m, \nabla \bar{n})  +A(\mathbf{u}^{m-1}_h,\xi^m_n,\bar{n}) =(\rho_n^m,\bar{n})-(\delta_t \theta_n^m,\bar{n})+D_n(\theta_n^m,\bar{n}) -A(\mathbf{u}^{m-1}_h,\theta^m_n,\bar{n})\nonumber\\
&&\!\!\!\!\!\!\!\!\!\!\! 
-A(\mathbf{u}^m-\mathbf{u}^{m-1} +\xi^{m-1}_{\mathbf{u}}+\theta^{m-1}_{\mathbf{u}},n^m,\bar{n}) +\chi_1(n^{m}(\mathbf{s}^m-\mathbf{s}^{m-1}+\xi_{\mathbf{s}}^{m-1}+\theta_{\mathbf{s}}^{m-1}),\nabla\bar{n})\nonumber
\\
&&\!\!\!\!\!\!\!\!\!\!\! +\chi_1((\xi^{m}_{n}+\theta^{m}_{n})\mathbf{s}^{m-1}_{h},\nabla\bar{n})+\mu_1((\xi^m_{n}+\theta^m_{n})(1-[n^{m-1}_{h}]_{+}-a_{1}[w^{m-1}_{h}]_{+}),\bar{n})\nonumber\\
&&\!\!\!\!\!\!\!\!\!\!\!
-\mu_1(n^{m}(n^m-n^{m-1}+ n^{m-1} - [n^{m-1}_h]_+ +a_1(w^m-w^{m-1}+ w^{m-1} - [w^{m-1}_h]_+)),\bar{n}),
\end{eqnarray}
for all $\bar{n}\in N_h$. Taking $\bar{n}=\xi_n^m$ in (\ref{errn-int}) and using (\ref{a8}) we get
\begin{eqnarray}\label{errn-int2}
&\displaystyle\frac{1}{2}\delta_t\Vert \xi_n^m&\!\!\!\!\!\!\Vert_{L^2}^2 +\frac{\Delta t}{2}\Vert\delta_t \xi_n^m\Vert_{L^2}^2 +D_n \Vert \nabla\xi_n^m\Vert_{L^2}^2  =(\rho_n^m,\xi_{n}^m)+(D_n \theta_n^m - \delta_t \theta_n^m,\xi_{n}^m)-A(\mathbf{u}^{m-1}_h,\theta^m_n,\xi_{n}^m)\nonumber\\
&&\!\!\!\!\!\!\!\!\!\!\! 
-A(\mathbf{u}^m-\mathbf{u}^{m-1}+\xi^{m-1}_{\mathbf{u}}+\theta^{m-1}_{\mathbf{u}},n^m,\xi_{n}^m) +\mu_1((\xi^m_{n}+\theta^m_{n})(1-[n^{m-1}_{h}]_{+}-a_{1}[w^{m-1}_{h}]_{+}),\xi_{n}^m)\nonumber
\\
&&\!\!\!\!\!\!\!\!\!\!\! 
-\mu_1(n^{m}(n^m-n^{m-1}+ n^{m-1} - [n^{m-1}_h]_+ +a_1(w^m-w^{m-1}+ w^{m-1} - [w^{m-1}_h]_+)),\xi_{n}^m)\nonumber\\
&&\!\!\!\!\!\!\!\!\!\!\!
+\chi_1(n^{m}(\mathbf{s}^m-\mathbf{s}^{m-1}+\xi_{\mathbf{s}}^{m-1}+\theta_{\mathbf{s}}^{m-1}),\nabla\xi_{n}^m)+\chi_1((\xi^{m}_{n}+\theta^{m}_{n})\mathbf{s}^{m-1}_{h},\nabla\xi_{n}^m)=\sum_{k=1}^{8} I_k.
\end{eqnarray}
The terms on the right hand side of (\ref{errn-int2}) are controlled in the following way: First, using the H\"older and Young inequalities and (\ref{aprox01})$_1$, we get
\begin{eqnarray}\label{Ea1a}
&I_1 + I_2&\!\!\!\leq (\Vert \xi_n^m\Vert_{L^2}+ \Vert \nabla \xi_n^m\Vert_{L^2})\Vert \rho_n^m\Vert_{(H^1)'} + (\Vert(\mathcal{I} - \mathbb{P}_n) \delta_t n^m \Vert_{L^2}+D_n \Vert \theta_n^m\Vert_{L^2})\Vert \xi_n^m\Vert_{L^2} \nonumber\\
&&\!\!\!\leq \displaystyle\frac{D_n}{10}\Vert \nabla \xi_n^m\Vert_{L^2}^2 + \frac{1}{2}\Vert\xi_n^m\Vert_{L^2}^2+ C\Big(1+\frac{1}{D_n}\Big) \Vert \rho_n^m\Vert_{(H^1)'}^2 +C h^{2(r_1+1)}\Big[\Vert \delta_t n^m\Vert_{H^{r_1+1}}^2+D_n^2 \Vert n^m \Vert_{H^{r_1 +1}}^2\Big] \nonumber\\
&&\!\!\!\leq \displaystyle\frac{D_n}{10}\Vert \nabla \xi_n^m\Vert_{L^2}^2 +\frac{1}{2}\Vert \xi_n^m\Vert_{L^2}^2+C\Delta t \!\int_{t_{m-1}}^{t_m}\!\!\!\Vert \partial_{tt} n(t)\Vert_{(H^1)'}^2 dt \nonumber\\
&& +C h^{2(r_1+1)}\Big[\frac{1}{\Delta t}\int_{t_{m-1}}^{t_m}\!\!\!\Vert \partial_t n(t) \Vert_{H^{r_1+1}}^2  dt + \Vert n^m \Vert_{H^{r_1 +1}}^2\Big].
\end{eqnarray}
Moreover, from (\ref{a6aa}), the H\"older and Young inequalities, (\ref{ime1}), (\ref{StabStk2}), (\ref{aprox01})$_1$  and (\ref{aprox01-a}) we obtain
\begin{eqnarray}\label{Ea1c}
&I_3&\!\!\! = A(\xi_{\mathbf{u}}^{m-1},\theta^m_n,\xi^m_n) -A(\mathbb{P}_{\mathbf{u}} \mathbf{u}^{m-1},\theta^m_n,\xi^m_n)  \nonumber\\
&&\!\!\! \leq  \Vert \xi_{\mathbf{u}}^{m-1}\Vert_{L^2}   \Vert \theta_n^m\Vert_{L^\infty\cap W^{1,3}} \Vert \xi_n^m\Vert_{H^1} + \Vert \mathbb{P}_{\mathbf{u}} \mathbf{u}^{m-1}\Vert_{L^\infty\cap W^{1,3}}   \Vert \theta_n^m\Vert_{L^2} \Vert \xi_n^m\Vert_{H^1} \nonumber\\
&&\!\!\! \leq \displaystyle\frac{D_n}{10} \Vert \nabla \xi_n^m\Vert_{L^2}^2 + \frac{1}{2}\Vert \xi_n^m\Vert_{L^2}^2+ C \Vert n^m\Vert_{H^2}^2 \Vert \xi_{\mathbf{u}}^{m-1}\Vert_{L^2}^2    \nonumber\\
&&+C h^{2(r_1 +1)} \Vert [\mathbf{u}^{m-1},\pi^{m-1}]\Vert_{H^2\times H^1}^2   \Vert n^m\Vert_{H^{r_1 + 1}}^2.
\end{eqnarray}
Now, using the H\"older and Young inequalities, (\ref{ime1}), (\ref{StabStk1}) and (\ref{aprox01})$_{1,2,4}$, we have
\begin{eqnarray}\label{Ea1d}
&I_4 + I_6+I_7&\!\!\!\! \leq (\Vert \mathbf{u}^m\!\!-\!\mathbf{u}^{m-1}\Vert_{L^2}+\Vert \xi_{\mathbf{u}}^{m-1}\Vert_{L^2}+\Vert \theta_{\mathbf{u}}^{m-1}\Vert_{L^2}  )  \Vert n^m\Vert_{L^\infty\cap W^{1,3}} \Vert \xi_n^m\Vert_{H^1}  \nonumber\\
&& + \mu_1 (\Vert n^m\!\!-\! n^{m-1}\Vert_{L^2}+\Vert \xi_{n}^{m-1}\Vert_{L^2}+\Vert \theta_{n}^{m-1}\Vert_{L^2})  \Vert n^m\Vert_{L^3} \Vert \xi_n^m\Vert_{H^1} \nonumber\\
&& + \mu_1 a_1(\Vert w^m\!\!-\! w^{m-1}\Vert_{L^2}+\Vert \xi_{w}^{m-1}\Vert_{L^2}+\Vert \theta_{w}^{m-1}\Vert_{L^2})  \Vert n^m\Vert_{L^3} \Vert \xi_n^m\Vert_{H^1}  \nonumber\\
&&+ \chi_1(\Vert \mathbf{s}^m\!\!-\!\mathbf{s}^{m-1}\Vert_{L^2}+\Vert \xi_{\mathbf{s}}^{m-1}\Vert_{L^2}+\Vert \theta_{\mathbf{s}}^{m-1}\Vert_{L^2}  )  \Vert n^m\Vert_{L^\infty} \Vert \nabla \xi_n^m\Vert_{L^2} \nonumber\\
&&\!\!\!\! \leq \displaystyle\frac{D_n}{10} \Vert \nabla \xi_n^m\Vert_{L^2}^2 + \frac{1}{2}\Vert \xi_n^m\Vert_{L^2}^2+C \Vert [\xi_{\mathbf{u}}^{m-1}, \xi_{n}^{m-1},\xi_{w}^{m-1}, \xi_{\mathbf{s}}^{m-1}]\Vert_{L^2}^2 \Vert n^m\Vert_{H^2}^2\nonumber\\
&&+C \Vert [\mathbf{u}^m-\mathbf{u}^{m-1},n^m - n^{m-1}, w^m-w^{m-1},\mathbf{s}^m-\mathbf{s}^{m-1}]\Vert_{L^2}^2  \Vert n^m\Vert_{H^2}^2\nonumber\\
&&+C (h^{2(r+1)} \|[\mathbf{u}^{m-1},\pi^{m-1}]\|_{H^{r+1}\times H^{r}}^2+h^{2(r_1+1)} \|n^{m-1}\|_{H^{r_1+1}}^2) \Vert n^m\Vert_{H^2}^2\nonumber\\
&&+C (h^{2(r_2+1)} \|w^{m-1}\|_{H^{r_2+1}}^2+h^{2(r_4+1)} \|\mathbf{s}^{m-1}\|_{H^{r_4+1}}^2) \Vert n^m\Vert_{H^2}^2
\end{eqnarray}
and
\begin{eqnarray}\label{Ea1g}
&\displaystyle I_{5}&\!\!\! \leq \mu_1(\Vert \xi^m_{n}\Vert_{L^2}+\Vert\theta^m_{n}\Vert_{L^2})\Vert \xi_n^{m}\Vert_{L^2} + \mu_1(\Vert \xi^m_{n}\Vert_{L^2}+\Vert\theta^m_{n}\Vert_{L^2}) (\Vert n^{m-1}_{h}\Vert_{L^6} + a_1\Vert w^{m-1}_{h}\Vert_{L^6}) \Vert \xi_n^{m}\Vert_{L^3}  \nonumber\\
&&\!\!\! \leq \displaystyle\frac{D_n}{10} \Vert \nabla\xi_n^m\Vert_{L^2}^2+\frac{1}{2} \Vert \xi_n^m\Vert_{L^2}^2 + C(1+\Vert n^{m-1}_{h}\Vert_{L^6}^2 + \Vert w^{m-1}_{h}\Vert_{L^6}^2) (\Vert \xi_n^{m}\Vert_{L^2}^2 + h^{2(r_1+1)}\|n^{m}\|_{ H^{r_1+1}}^2).\ \ \ \ \ \ \
\end{eqnarray}
Finally, from (\ref{ime1}), (\ref{aprox01})$_1$, (\ref{aprox01-a}), the 3D interpolation inequality (\ref{in3Dl4}) and the inductive hypothesis (\ref{IndHyp}), we can bound
\begin{eqnarray}\label{Ea1e-new}
& I_8&\!\!\! = \chi_1 (\xi^{m}_n {\bf s}^{m-1}_h,\nabla \xi_n^m) - \chi_1(\theta^{m}_n \xi_{\bf s}^{m-1},\nabla \xi_n^m)+ \chi_1(\theta^{m}_n \mathbb{P}_{\bf s}{\bf s}^{m-1},\nabla \xi_n^m) \nonumber\\
&&\!\!\! \leq \chi_1(\Vert \xi^{m}_{n}\Vert_{L^2}^{1/10}\Vert \xi^{m}_{n}\Vert_{H^1}^{9/10} \Vert{\bf s}^{m-1}_h\Vert_{L^{10/3}}+ \Vert \xi_{\bf s}^{m-1}\Vert_{L^2}   \Vert \theta_n^{m}\Vert_{L^\infty}  + \Vert \mathbb{P}_{\bf s} {\bf s}^{m-1}\Vert_{L^\infty}  \Vert \theta_n^{m}\Vert_{L^2}) \Vert\nabla \xi_n^m\Vert_{L^2}\nonumber\\
&&\!\!\! \leq \displaystyle\frac{D_n}{10} \Vert \nabla\xi_n^m\Vert_{L^2}^2 +  C \Vert \xi^{m}_{n}\Vert_{L^2}^{2}+C\Vert n^{m}\Vert_{H^2}^2 \Vert \xi_{\bf s}^{m-1}\Vert_{L^2}^2    + \displaystyle Ch^{2(r_1 +1)} \Vert {\bf s}^{m-1}\Vert_{H^2}^2   \Vert n^{m}\Vert_{H^{r_1 + 1}}^2.
\end{eqnarray}
Thus, from (\ref{errn-int2})-(\ref{Ea1e-new}), we arrive at
\begin{eqnarray}\label{errnfin}
&\displaystyle\frac{1}{2}&\!\!\!\!\!\delta_t  \Vert \xi_n^m\Vert_{L^2}^2 + \frac{\Delta t}{2} \Vert \delta_t \xi_n^m\Vert_{L^2}^2 + \frac{D_n}{2}\Vert\nabla\xi_n^m\Vert_{L^2}^2 \leq   C h^{2(r_1+1)}\Big[\frac{1}{\Delta t}\int_{t_{m-1}}^{t_m}\!\!\!\Vert \partial_t n(t) \Vert_{H^{r_1+1}}^2  dt + \Vert n^m \Vert_{H^{r_1 +1}}^2\Big] \nonumber\\
&&\!\!\!\!\!\!\!\!\!+C\Delta t \!\int_{t_{m-1}}^{t_m}\!\!\!\Vert \partial_{tt} n(t)\Vert_{(H^1)'}^2 dt+ C \Vert n^m\Vert_{H^2}^2 \Vert \xi_{\mathbf{u}}^{m-1}\Vert_{L^2}^2+C h^{2(r_1 +1)} \Vert [\mathbf{u}^{m-1},\pi^{m-1}]\Vert_{H^2\times H^1}^2   \Vert n^m\Vert_{H^{r_1 + 1}}^2\nonumber\\
&&\!\!\!\!\!\!\!\!\! +C \Vert [\xi_{\mathbf{u}}^{m-1}, \xi_{n}^{m-1},\xi_{w}^{m-1}, \xi_{\mathbf{s}}^{m-1}]\Vert_{L^2}^2 \Vert n^m\Vert_{H^2}^2 +C (h^{2(r_2+1)} \|w^{m-1}\|_{H^{r_2+1}}^2+h^{2(r_4+1)} \|\mathbf{s}^{m-1}\|_{H^{r_4+1}}^2) \Vert n^m\Vert_{H^2}^2\nonumber\\
&&\!\!\!\!\!\!\!\!\! +C \Vert [\mathbf{u}^m-\mathbf{u}^{m-1},n^m - n^{m-1}, w^m-w^{m-1},\mathbf{s}^m-\mathbf{s}^{m-1}]\Vert_{L^2}^2  \Vert n^m\Vert_{H^2}^2  + \displaystyle Ch^{2(r_1 +1)} \Vert {\bf s}^{m-1}\Vert_{H^2}^2   \Vert n^{m}\Vert_{H^{r_1 + 1}}^2 \nonumber\\
&&\!\!\!\!\!\!\!\!\!+C (h^{2(r+1)} \|[\mathbf{u}^{m-1},\pi^{m-1}]\|_{H^{r+1}\times H^{r}}^2+h^{2(r_1+1)} \|n^{m-1}\|_{H^{r_1+1}}^2) \Vert n^m\Vert_{H^2}^2+C\Vert n^{m}\Vert_{H^2}^2 \Vert \xi_{\bf s}^{m-1}\Vert_{L^2}^2 \nonumber\\
&&\!\!\!\!\!\!\!\!\! + C(1+\Vert n^{m-1}_{h}]\Vert_{L^6}^2 + \Vert w^{m-1}_{h}\Vert_{L^6}^2) (\Vert \xi_n^{m}\Vert_{L^2}^2 + h^{2(r_1+1)}\|n^{m}\|_{ H^{r_1+1}}^2).
\end{eqnarray}

On the other hand, taking into account (\ref{Interp1New})$_2$, from (\ref{errw}),  (\ref{u1aw}) and (\ref{u1asig})-(\ref{u1au}) we have
\begin{eqnarray}\label{errw-int}
&(\delta_t \xi_w^m,&\!\!\!\!\bar{w}) + D_w(\nabla \xi_w^m, \nabla \bar{w})  +A(\mathbf{u}^{m-1}_h,\xi^m_w,\bar{w}) =(\rho_w^m,\bar{w})-(\delta_t \theta_w^m,\bar{w})+D_w(\theta_w^m,\bar{w})-A(\mathbf{u}^{m-1}_h,\theta^m_w,\bar{w})\nonumber\\
&&\!\!\!\!\!\!\!\!\!\!\! 
-A(\mathbf{u}^m-\mathbf{u}^{m-1}+\xi^{m-1}_{\mathbf{u}}+\theta^{m-1}_{\mathbf{u}},w^m,\bar{w})+\chi_2(w^{m}(\mathbf{s}^m-\mathbf{s}^{m-1}+\xi_{\mathbf{s}}^{m-1}+\theta_{\mathbf{s}}^{m-1}),\nabla\bar{w})\nonumber
\\
&&\!\!\!\!\!\!\!\!\!\!\! +\chi_2((\xi^{m}_{w}+\theta^{m}_{w})\mathbf{s}^{m-1}_{h},\nabla\bar{w})+\mu_2((\xi^m_{w}+\theta^m_{w})(1-a_2[n^{m-1}_{h}]_{+}-[w^{m-1}_{h}]_{+}),\bar{w})\nonumber\\
&&\!\!\!\!\!\!\!\!\!\!\!
-\mu_2(w^{m}(w^m-w^{m-1}+ w^{m-1} - [w^{m-1}_h]_+ +a_2(n^m-n^{m-1}+ n^{m-1} - [n^{m-1}_h]_+)),\bar{w}),
\end{eqnarray}
for all $\bar{w}\in W_h$. Taking $\bar{w}=\xi_w^m$ in (\ref{errw-int}), and proceeding analogously to (\ref{errn-int2})-(\ref{Ea1e-new}), we arrive at
\begin{eqnarray}\label{errwfin}
&\displaystyle\frac{1}{2}&\!\!\!\!\!\delta_t  \Vert \xi_w^m\Vert_{L^2}^2 + \frac{\Delta t}{2} \Vert \delta_t \xi_w^m\Vert_{L^2}^2 + \frac{D_w}{2}\Vert\nabla\xi_w^m\Vert_{L^2}^2 \leq   C h^{2(r_2+1)}\Big[\frac{1}{\Delta t}\int_{t_{m-1}}^{t_m}\!\!\!\Vert \partial_t w(t) \Vert_{H^{r_2+1}}^2  dt + \Vert w^m \Vert_{H^{r_2 +1}}^2\Big] \nonumber\\
&&\!\!\!\!\!\!\!\!\!+C\Delta t \!\int_{t_{m-1}}^{t_m}\!\!\!\Vert \partial_{tt} w(t)\Vert_{(H^1)'}^2 dt+ C \Vert w^m\Vert_{H^2}^2 \Vert \xi_{\mathbf{u}}^{m-1}\Vert_{L^2}^2+C h^{2(r_2 +1)} \Vert [\mathbf{u}^{m-1},\pi^{m-1}]\Vert_{H^2\times H^1}^2   \Vert w^m\Vert_{H^{r_2 + 1}}^2\nonumber\\
&&\!\!\!\!\!\!\!\!\! +C \Vert [\xi_{\mathbf{u}}^{m-1}, \xi_{n}^{m-1},\xi_{w}^{m-1}, \xi_{\mathbf{s}}^{m-1}]\Vert_{L^2}^2 \Vert w^m\Vert_{H^2}^2 +C (h^{2(r_2+1)} \|w^{m-1}\|_{H^{r_2+1}}^2+h^{2(r_4+1)} \|\mathbf{s}^{m-1}\|_{H^{r_4+1}}^2) \Vert w^m\Vert_{H^2}^2\nonumber\\
&&\!\!\!\!\!\!\!\!\! +C \Vert [\mathbf{u}^m-\mathbf{u}^{m-1},n^m - n^{m-1}, w^m-w^{m-1},\mathbf{s}^m-\mathbf{s}^{m-1}]\Vert_{L^2}^2  \Vert w^m\Vert_{H^2}^2  + \displaystyle Ch^{2(r_2 +1)} \Vert {\bf s}^{m-1}\Vert_{H^2}^2   \Vert w^{m}\Vert_{H^{r_2 + 1}}^2 \nonumber\\
&&\!\!\!\!\!\!\!\!\!+C (h^{2(r+1)} \|[\mathbf{u}^{m-1},\pi^{m-1}]\|_{H^{r+1}\times H^{r}}^2+h^{2(r_1+1)} \|n^{m-1}\|_{H^{r_1+1}}^2) \Vert w^m\Vert_{H^2}^2+C\Vert w^{m}\Vert_{H^2}^2 \Vert \xi_{\bf s}^{m-1}\Vert_{L^2}^2 \nonumber\\
&&\!\!\!\!\!\!\!\!\! + C(1+\Vert n^{m-1}_{h}]\Vert_{L^6}^2 + \Vert w^{m-1}_{h}\Vert_{L^6}^2) (\Vert \xi_w^{m}\Vert_{L^2}^2 + h^{2(r_2+1)}\|w^{m}\|_{ H^{r_2+1}}^2).
\end{eqnarray}
\

{\underline{{\it 2. Error estimate for the chemical concentration} $c$}}\\

Taking into account (\ref{Interp1New})$_3$, from (\ref{errc}),  (\ref{u1ac}) and (\ref{u1au}), we have
\begin{eqnarray}\label{errc-int}
&\left(\delta_t \xi_c^m,\bar{c}\right) &\!\!\!\!+ D_c(\nabla \xi_c^m, \nabla \bar{c})  +A(\mathbf{u}^{m-1}_h,\xi^m_c,\bar{c}) =(\rho_c^m,\bar{c})-\left(\delta_t \theta_c^m,\bar{c}\right)+D_c(\theta_c^m,\bar{c})-A(\mathbf{u}^{m-1}_h,\theta^m_c,\bar{c})\nonumber\\
&&\!\!\!
-A(\mathbf{u}^m-\mathbf{u}^{m-1}+\xi^{m-1}_{\mathbf{u}}+\theta^{m-1}_{\mathbf{u}},c^m,\bar{c}) - ( (\xi^m_c + \theta^m_c)(\alpha [n^m_h]_+ + \beta [w^m_h]_+),\bar{c})\nonumber\\
&&\!\!\!  -(c^{m}(\alpha n^{m}-\alpha [n^{m}_h]_+ + \beta w^{m}-\beta [w^{m}_h]_+),\bar{c}), \ \ \forall \bar{c}\in C_h,
\end{eqnarray}
from which, considering $\bar{c}=\xi_c^m$ and using (\ref{a8}), we get
\begin{eqnarray}\label{errc-int2}
&\displaystyle\frac{1}{2}&\!\!\!\!\!\delta_t  \Vert \xi_c^m\Vert_{L^2}^2 + \frac{\Delta t}{2} \Vert \delta_t \xi_c^m\Vert_{L^2}^2 + D_c\Vert \nabla \xi_c^m\Vert_{L^2}^2   =(\rho_c^m - \delta_t \theta_c^m +D_c \theta_c^m,\xi_{c}^m)-A(\mathbf{u}^{m-1}_h,\theta^m_c,\xi_c^m)\nonumber\\
&&\!\!\!
-A(\mathbf{u}^m-\mathbf{u}^{m-1}+\xi^{m-1}_{\mathbf{u}}+\theta^{m-1}_{\mathbf{u}},c^m,\xi^m_c) - ( (\xi^m_c + \theta^m_c)(\alpha [n^m_h]_+ + \beta [w^m_h]_+),\xi^m_c)\nonumber\\
&&\!\!\!  -(c^{m}(\alpha n^{m}-\alpha [n^{m}_h]_+ + \beta w^{m}-\beta [w^{m}_h]_+),\xi^m_c)=\sum_{k=1}^{5}L_k.
\end{eqnarray}
Notice that, using the H\"older and Young inequalities and (\ref{aprox01})$_3$, we obtain
\begin{eqnarray}\label{Ea1aC}
&L_1 &\!\!\!\leq (\Vert \xi_c^m\Vert_{L^2}+ \Vert \nabla \xi_c^m\Vert_{L^2})\Vert \rho_c^m\Vert_{(H^1)'} + (\Vert(\mathcal{I} - \mathbb{P}_c) \delta_t c^m \Vert_{L^2}+D_c\Vert \theta_c^m\Vert_{L^2})\Vert \xi_c^m\Vert_{L^2} \nonumber\\
&&\!\!\!\leq \displaystyle\frac{D_c}{6}\Vert \nabla \xi_c^m\Vert_{L^2}^2 + \frac{1}{2}\Vert\xi_c^m\Vert_{L^2}^2+ C\Big(1+\frac{1}{D_c}\Big) \Vert \rho_c^m\Vert_{(H^1)'}^2 +C h^{2(r_3+1)}\Big[\Vert \delta_t c^m\Vert_{H^{r_3+1}}^2+D_c^2 \Vert c^m \Vert_{H^{r_3 +1}}^2\Big] \nonumber\\
&&\!\!\!\leq \displaystyle\frac{D_c}{6}\Vert \nabla \xi_c^m\Vert_{L^2}^2 +\frac{1}{2}\Vert \xi_c^m\Vert_{L^2}^2+C\Delta t \!\int_{t_{m-1}}^{t_m}\!\!\!\Vert \partial_{tt} c(t)\Vert_{(H^1)'}^2 dt \nonumber\\
&& +C h^{2(r_3+1)}\Big[\frac{1}{\Delta t}\int_{t_{m-1}}^{t_m}\!\!\!\Vert \partial_t c(t) \Vert_{H^{r_3+1}}^2  dt + \Vert c^m \Vert_{H^{r_3 +1}}^2\Big].
\end{eqnarray}
Moreover, from the definition of the skew-symmetric trilinear form (\ref{a6aa}), the H\"older and Young inequalities, (\ref{ime1}), (\ref{StabStk1}), (\ref{StabStk2}), (\ref{aprox01})$_{1,2,3}$ and (\ref{aprox01-a}) we have
\begin{eqnarray}\label{Ea1cC}
&L_2&\!\!\! = A(\xi_{\mathbf{u}}^{m-1},\theta^m_c,\xi^m_c) -A(\mathbb{P}_{\mathbf{u}} \mathbf{u}^{m-1},\theta^m_c,\xi^m_c)  \nonumber\\
&&\!\!\! \leq  \Vert \xi_{\mathbf{u}}^{m-1}\Vert_{L^2}   \Vert \theta_c^m\Vert_{L^\infty\cap W^{1,3}} \Vert \xi_c^m\Vert_{H^1} + \Vert \mathbb{P}_{\mathbf{u}} \mathbf{u}^{m-1}\Vert_{L^\infty\cap W^{1,3}}   \Vert \theta_c^m\Vert_{L^2} \Vert \xi_c^m\Vert_{H^1} \nonumber\\
&&\!\!\! \leq \displaystyle\frac{D_c}{6}\Vert \nabla \xi_c^m\Vert_{L^2}^2 + \frac{1}{2}\Vert \xi_c^m\Vert_{L^2}^2+ C \Vert c^m\Vert_{H^2}^2 \Vert \xi_{\mathbf{u}}^{m-1}\Vert_{L^2}^2   +C h^{2(r_3 +1)} \Vert [\mathbf{u}^{m-1},\pi^{m-1}]\Vert_{H^2\times H^1}^2   \Vert c^m\Vert_{H^{r_3 + 1}}^2\ \ \ \ \ \ \ \ 
\end{eqnarray}
and
\begin{eqnarray}\label{Ea1dC}
&\displaystyle\sum_{k=3}^{5} L_k&\!\!\!\! \leq (\Vert \mathbf{u}^m\!\!-\!\mathbf{u}^{m-1}\Vert_{L^2}+\Vert \xi_{\mathbf{u}}^{m-1}\Vert_{L^2}+\Vert \theta_{\mathbf{u}}^{m-1}\Vert_{L^2}  )  \Vert c^m\Vert_{L^\infty\cap W^{1,3}} \Vert \xi_c^m\Vert_{H^1}  \nonumber\\
&& + (\Vert \xi_{c}^{m}\Vert_{L^2}+\Vert \theta_{c}^{m}\Vert_{L^2})(\alpha \Vert n^m_h \Vert_{L^6}+\beta \Vert w^m_h \Vert_{L^6})\Vert \xi_c^m\Vert_{L^3}  \nonumber\\
&&+ (\alpha \Vert \xi_{n}^{m}\Vert_{L^2}+\alpha\Vert \theta_{n}^{m}\Vert_{L^2} +\beta\Vert \xi_{w}^{m}\Vert_{L^2}+\beta\Vert \theta_{w}^{m}\Vert_{L^2})\Vert c^m\Vert_{L^3}\Vert \xi_c^m\Vert_{L^6} \nonumber\\
&&\!\!\!\! \leq \displaystyle\frac{D_c}{6} \Vert \nabla \xi_c^m\Vert_{L^2}^2 + \frac{1}{2}\Vert  \xi_c^m\Vert_{L^2}^2+ C (\Vert \mathbf{u}^m-\mathbf{u}^{m-1}\Vert_{L^2}^2+h^{2(r+1)} \|[\mathbf{u}^{m-1},\pi^{m-1}]\|_{H^{r+1}\times H^{r}}^2)  \Vert c^m\Vert_{H^2}^2\nonumber\\
&&+C \Vert \xi_{\mathbf{u}}^{m-1}\Vert_{L^2}^2 \Vert c^m\Vert_{H^2}^2+ C(\Vert \xi_{c}^{m}\Vert_{L^2}^2+h^{2(r_3 +1)}\Vert {c}^{m}\Vert_{H^{r_3 +1}}^2)( \Vert n^m_h\Vert_{L^6}^2+ \Vert w^m_h\Vert_{L^6}^2)\nonumber\\
&& + C(\Vert \xi_{n}^{m}\Vert_{L^2}^2+h^{2(r_1 +1)}\Vert {n}^{m}\Vert_{H^{r_1 +1}}^2+\Vert \xi_{w}^{m}\Vert_{L^2}^2+h^{2(r_2 +1)}\Vert {w}^{m}\Vert_{H^{r_2 +1}}^2) \Vert c^m\Vert_{L^3}^2.
\end{eqnarray}
Therefore, from (\ref{errc-int2})-(\ref{Ea1dC}), we arrive at
\begin{eqnarray}\label{errcfin}
&\displaystyle\frac{1}{2}&\!\!\!\!\!\delta_t  \Vert \xi_c^m\Vert_{L^2}^2 + \displaystyle\frac{\Delta t}{2} \Vert \delta_t \xi_c^m\Vert_{L^2}^2 + \frac{D_c}{2}\Vert \nabla \xi_c^m\Vert_{L^2}^2  \leq C\Delta t \!\int_{t_{m-1}}^{t_m}\!\!\!\Vert \partial_{tt} c(t)\Vert_{(H^1)'}^2 dt+ C \Vert c^m\Vert_{H^2}^2 \Vert \xi_{\mathbf{u}}^{m-1}\Vert_{L^2}^2 \nonumber\\
&&\!\!\!\!\!\!\!\!\!+ C h^{2(r_3+1)}\Big[\displaystyle\frac{1}{\Delta t }\!\int_{t_{m-1}}^{t_m}\!\!\!\Vert \partial_t c(t) \Vert_{H^{r_3+1}}^2  dt+\Vert c^m \Vert_{H^{r_3 +1}}^2\Big] +C h^{2(r_3 +1)} \Vert [\mathbf{u}^{m-1},\pi^{m-1}]\Vert_{H^2\times H^1}^2   \Vert c^m\Vert_{H^{r_3 + 1}}^2\nonumber\\
&&\!\!\!\!\!\!\!\!\!+ C( \Vert \mathbf{u}^m-\mathbf{u}^{m-1}\Vert_{L^2}^2 + h^{2(r+1)} \|[\mathbf{u}^{m-1},\pi^{m-1}]\|_{H^{r+1}\times H^{r}}^2+\Vert \xi_{\mathbf{u}}^{m-1}\Vert_{L^2}^2) \Vert c^m\Vert_{H^2}^2+ \Vert \xi^m_c\Vert_{L^2}^2\nonumber\\
&&\!\!\!\!\!\!\!\!\! + C(\Vert \xi_{n}^{m}\Vert_{L^2}^2+h^{2(r_1 +1)}\Vert {n}^{m}\Vert_{H^{r_1 +1}}^2+\Vert \xi_{w}^{m}\Vert_{L^2}^2+h^{2(r_2 +1)}\Vert {w}^{m}\Vert_{H^{r_2 +1}}^2) \Vert c^m\Vert_{L^3}^2\nonumber\\
&&\!\!\!\!\!\!\!\!\!+ C(\Vert \xi_{c}^{m}\Vert_{L^2}^2+h^{2(r_3 +1)}\Vert {c}^{m}\Vert_{H^{r_3 +1}}^2)( \Vert n^m_h\Vert_{L^6}^2+ \Vert w^m_h\Vert_{L^6}^2).
\end{eqnarray}
\

{\underline{{\it 4. Error estimate for the velocity field} $\mathbf{u}$}}\\

Using (\ref{StokesOp}),  (\ref{u1an})-(\ref{u1aw}) and (\ref{u1au})-(\ref{u1api}) in (\ref{erru})-(\ref{errpi}), we obtain
\begin{eqnarray}\label{erru-int}
&(\delta_t  \xi_{\mathbf{u}}^m,&\!\!\!\!\!\bar{\mathbf{u}})+ D_{\mathbf{u}}(\nabla \xi_{\mathbf{u}}^m, \nabla \bar{\mathbf{u}})= (\rho_{\mathbf{u}}^m-\delta_t \theta_{\mathbf{u}}^m,\bar{\mathbf{u}}) - kB(\mathbf{u}^m-\mathbf{u}^{m-1},{\mathbf{u}}^m,\bar{\mathbf{u}})   -kB(\xi^{m-1}_{\mathbf{u}}+\theta^{m-1}_{\mathbf{u}},{\mathbf{u}}^m,\bar{\mathbf{u}})\nonumber\\
&&\hspace{-0.8 cm} \!-kB(\mathbf{u}^{m-1}_h,\xi^m_{\mathbf{u}}+\theta^m_{\mathbf{u}},\bar{\mathbf{u}}) +(\xi^m_\pi,\nabla \cdot \bar{\mathbf{u}}) +((\gamma( \xi^{m}_{n}+\theta^{m}_{n})+\lambda( \xi^{m}_{w}+\theta^{m}_{w}))\nabla\phi,\bar{\mathbf{u}}),
\end{eqnarray}
\begin{equation}\label{errpi-int}
(\bar{\pi},\nabla \cdot \xi^m_{\mathbf{u}})=0,
\end{equation}
for all $[\bar{\mathbf{u}},\bar{\pi}]\in {U}_h\times \Pi_h$. Taking $\bar{\mathbf{u}}=\xi_{\mathbf{u}}^m$ in (\ref{erru-int}), $\bar{\pi}=\xi^m_\pi$ in (\ref{errpi-int}), using (\ref{a7}) and adding the resulting expressions, we get
\begin{eqnarray}\label{erru-int2}
&\displaystyle\frac{1}{2}&\!\!\!\!\!\delta_t  \Vert\xi_{\mathbf{u}}^m\Vert_{L^2}^2+\displaystyle\frac{\Delta t}{2}\Vert\delta_t \xi_{\mathbf{u}}^m\Vert_{L^2}^2 +D_{\mathbf{u}}\Vert\nabla \xi_{\mathbf{u}}^m\Vert_{L^2}^2= (\rho_{\mathbf{u}}^m - \delta_t \theta_{\mathbf{u}}^m,\xi_{\mathbf{u}}^m)  -kB(\mathbf{u}^{m-1}_h\!,\theta^m_{\mathbf{u}},\xi_{\mathbf{u}}^m) \nonumber\\
&&\!\!\!\!\! - kB(\mathbf{u}^m-\mathbf{u}^{m-1}+\xi^{m-1}_{\mathbf{u}}\!+\theta^{m-1}_{\mathbf{u}}\!,{\mathbf{u}}^m,\xi_{\mathbf{u}}^m)  +((\gamma( \xi^{m}_{n}+\theta^{m}_{n})+\lambda ( \xi^{m}_{w}+\theta^{m}_{w}))\nabla\phi,\xi_{\mathbf{u}}^m) =\sum_{i=1}^{4} S_i.\ \ \ \ \ \ \  
\end{eqnarray}
We bound the terms $S_i$ as follows: First, using the H\"older and Young inequalities, the Poincar\'e inequality (\ref{PIa}) and (\ref{StabStk1}), we have
\begin{eqnarray}\label{Ea1au}
&S_1&\!\!\!\leq \Vert \nabla \xi_{\mathbf{u}}^m\Vert_{L^2} \Vert \rho_{\mathbf{u}}^m\Vert_{(H^1)'}^2 + \Vert  \xi_{\mathbf{u}}^m\Vert_{L^2} \Vert (\mathcal{I}-\mathbb{P}_{\mathbf{u}}) \delta_t \mathbf{u}^m \Vert_{L^2}\nonumber\\
&&\!\!\!\leq \displaystyle\frac{D_{\mathbf{u}}}{6} \Vert \nabla \xi_{\mathbf{u}}^m\Vert_{L^2}^2 + \frac{C}{D_{\mathbf{u}}}\Vert \rho_{\mathbf{u}}^m\Vert_{(H^1)'}^2 + \frac{C h^{2(r+1)}}{D_{\mathbf{u}}}\Vert [\delta_t {\mathbf{u}}^m,\delta_t \pi^m]\Vert_{H^{r+1}\times H^r}\nonumber\\
&&\!\!\!\leq \frac{D_{\mathbf{u}}}{6}\Vert \nabla \xi_{\mathbf{u}}^m\Vert_{L^2}^2 + C\Delta t \int_{t_{m-1}}^{t_m}\Vert \partial_{tt} {\mathbf{u}}(t)\Vert_{(H^1)'}^2 dt+\frac{Ch^{2(r+1)}}{\Delta t}\int_{t_{m-1}}^{t_m}\Vert [\partial_t \mathbf{u},\partial_t\pi] \Vert_{H^{r+1}\times H^r}^2  dt.
\end{eqnarray}
Moreover, taking into account the definition of the skew-symmetric trilinear form (\ref{a6}), the H\"older, Young and Poincar\'e inequalities, (\ref{ime1}), (\ref{StabStk1})-(\ref{StabStk2}) and (\ref{aprox01})$_{1,2}$, we have
\begin{eqnarray}\label{Ea1du}
&S_2&\!\!\! = kB(\xi_{\mathbf{u}}^{m-1},\theta^m_{\mathbf{u}},\xi^m_{\mathbf{u}}) -kB(\mathbb{P}_{\mathbf{u}} \mathbf{u}^{m-1},\theta^m_{\mathbf{u}},\xi^m_{\mathbf{u}})  \nonumber\\
&&\!\!\! \leq k( \Vert \xi_{\mathbf{u}}^{m-1}\Vert_{L^2}   \Vert \theta_{\mathbf{u}}^m\Vert_{L^\infty\cap W^{1,3}} \Vert \xi_{\mathbf{u}}^m\Vert_{H^1} + \Vert \mathbb{P}_{\mathbf{u}} \mathbf{u}^{m-1}\Vert_{L^\infty\cap W^{1,3}}   \Vert \theta_{\mathbf{u}}^m\Vert_{L^2} \Vert \xi_{\mathbf{u}}^m\Vert_{H^1}) \nonumber\\
&&\!\!\! \leq  \frac{D_{\mathbf{u}}}{6} \Vert \nabla \xi_{\mathbf{u}}^m\Vert_{L^2}^2 + C  \Vert [\mathbf{u}^m,\pi^m]\Vert_{H^2\times H^1}^2 \Vert \xi_{\mathbf{u}}^{m-1}\Vert_{L^2}^2\nonumber\\
&&  + \displaystyle Ch^{2(r +1)} \Vert [\mathbf{u}^{m-1},\pi^{m-1}]\Vert_{H^2\times H^1}^2   \Vert [\mathbf{u}^m,\pi^m]\Vert_{H^{r + 1}\times H^r}^2
\end{eqnarray}
{and}
\begin{eqnarray}\label{Ea1cu}
&S_3 + S_4&\!\!\! \leq k(\Vert \mathbf{u}^m-\mathbf{u}^{m-1}\Vert_{L^2}+\Vert \xi_{\mathbf{u}}^{m-1}\Vert_{L^2}+\Vert \theta_{\mathbf{u}}^{m-1}\Vert_{L^2}  )  \Vert {\mathbf{u}}^m\Vert_{L^\infty\cap W^{1,3}} \Vert \xi_{\mathbf{u}}^m\Vert_{H^1}\nonumber\\
&& + (\gamma \Vert \xi^m_n\Vert_{L^2} + \gamma\Vert \theta^m_n\Vert_{L^2} +\lambda \Vert \xi^m_w \Vert_{L^2} + \lambda \Vert  \theta^m_w\Vert_{L^2} ) \Vert \nabla \phi\Vert_{L^3}\Vert \xi_{\mathbf{u}}^m\Vert_{H^1}  \nonumber\\
&&\!\!\! \leq \frac{D_{\mathbf{u}}}{6}\Vert \nabla \xi_{\mathbf{u}}^m\Vert_{L^2}^2 +C(\Vert \mathbf{u}^m\!\!-\!\mathbf{u}^{m-1}\Vert_{L^2}^2+h^{2(r+1)} \|[\mathbf{u}^{m-1}\!,\pi^{m-1}]\|_{H^{r+1}\times H^{r}}^2) \Vert \mathbf{u}^m\Vert_{H^2}^2\nonumber\\
&&+C\Vert \xi_{\mathbf{u}}^{m-1}\Vert_{L^2}^2 \Vert \mathbf{u}^m\Vert_{H^2}^2 +C (\Vert \xi^m_n\Vert_{L^2}^2  + \Vert \xi^m_w \Vert_{L^2}^2) \Vert \nabla \phi \Vert_{L^3}^2\nonumber\\
&& + C( h^{2(r_1+1)}\Vert n^m \Vert_{H^{r_1+1}}^2 +   h^{2(r_2+1)}\Vert w^m \Vert_{H^{r_2+1}}^2)\Vert \nabla \phi \Vert_{L^3}^2.
\end{eqnarray}
Therefore, from (\ref{erru-int2})-(\ref{Ea1cu}), we arrive at
\begin{eqnarray}\label{errufin}
&\displaystyle\frac{1}{2}&\!\!\!\!\!\delta_t  \Vert\xi_{\mathbf{u}}^m\Vert_{L^2}^2+\displaystyle\frac{\Delta t}{2}\Vert\delta_t \xi_{\mathbf{u}}^m\Vert_{L^2}^2 + \frac{D_{\mathbf{u}}}{2}\Vert\nabla \xi_{\mathbf{u}}^m\Vert_{L^2}^2  \leq  C\!\!\int_{t_{m-1}}^{t_m}\!\!\!\Big(\! \Delta t\Vert\partial_{tt} {\mathbf{u}}(t)\Vert_{(H^1)'}^2 + \frac{h^{2(r+1)}}{\Delta t}\Vert [\partial_t\mathbf{u},\partial_t\pi] \Vert_{H^{r+1}\times H^r}^2\!\Big)  dt\nonumber\\
&&\!\!\!\!\!\!\!\!\! + C  \Vert [\mathbf{u}^m,\pi^m]\Vert_{H^2\times H^1}^2 \Vert \xi_{\mathbf{u}}^{m-1}\Vert_{L^2}^2  + \displaystyle Ch^{2(r +1)} \Vert [\mathbf{u}^{m-1},\pi^{m-1}]\Vert_{H^2\times H^1}^2   \Vert [\mathbf{u}^m,\pi^m]\Vert_{H^{r + 1}\times H^r}^2 \nonumber\\
&&\!\!\!\!\!\!\!\!\!+C(\Vert \mathbf{u}^m\!\!-\!\mathbf{u}^{m-1}\Vert_{L^2}^2+h^{2(r+1)} \|[\mathbf{u}^{m-1}\!,\pi^{m-1}]\|_{H^{r+1}\times H^{r}}^2\!+\Vert \xi_{\mathbf{u}}^{m-1}\Vert_{L^2}^2) \Vert \mathbf{u}^m\Vert_{H^2}^2\nonumber\\
&&\!\!\!\!\!\!\!\!\! + C( h^{2(r_1+1)}\Vert n^m \Vert_{H^{r_1+1}}^2+ \Vert \xi^{m}_{n}\Vert_{L^2}^2+ h^{2(r_2+1)}\Vert w^m \Vert_{H^{r_2+1}}^2+ \Vert \xi^{m}_{w}\Vert_{L^2}^2)\Vert \nabla \phi\Vert_{L^3}^2.
\end{eqnarray}
\

{\underline{{\it 5. Error estimate for the auxiliary variable} $\bf s$}}\\

From (\ref{errs}), using (\ref{Interp1New})$_4$ and (\ref{u1an})-(\ref{u1au}), we have
\begin{eqnarray}\label{errs-int}
&(\delta_t \xi_{\bf s}^m&\!\!\!\!\!,\bar{\bf s})+ D_c(\nabla \cdot \xi_{\bf s}^m, \nabla \cdot\bar{\bf s}) + D_c(\mbox{rot}(\xi_{\bf s}^m), \mbox{rot}(\bar{\bf s}))  = (\rho_{\bf s}^m,\bar{\bf s})-(\delta_t \theta_{\bf s}^m,\bar{\bf s}) + D_c(\theta_{\bf s}^m,\bar{\bf s}) \nonumber\\
&&\!\!\!\!\!\!\!\!\!\!\!\!\!\!\!\!\!\! +(({\bf s}^m-{\bf s}^{m-1})\cdot \mathbf{u}^{m}+(\xi^{m-1}_{\mathbf{s}}+\theta^{m-1}_{\mathbf{s}})\cdot{\bf u}^{m}+(\xi^{m}_{\bf u}+\theta^{m}_{\bf u})\cdot\mathbf{s}^{m-1}_h+(\alpha n^{m}+\beta w^{m})({c}^{m}-{c}^{m-1}),\nabla \cdot\bar{\bf s})\nonumber\\
&&\!\!\!\!\!\!\!\!\!\!\!\!\!\!\!\!\!\! + ((\alpha n^{m}+\beta w^{m})(\xi_c^{m-1}+\theta_c^{m-1})+(\alpha (\xi_{n}^{m}+\theta_{n}^{m})+\beta (\xi_{w}^{m}+\theta_{w}^{m}))c^{m-1}_{h},\nabla \cdot\bar{\bf s}), \ \ \forall \bar{\bf s}\in {xxx\Sigma}_h,
\end{eqnarray}
from which, considering $\bar{\bf s}=\xi_{\bf s}^m$, we arrive at
\begin{eqnarray}\label{errs-int2}
&\displaystyle\frac{1}{2}&\!\!\!\!\!\delta_t \Vert\xi_{\bf s}^m\Vert_{L^2}^2+\displaystyle\frac{\Delta t}{2}\Vert\delta_t \xi_{\bf s}^m\Vert_{L^2}^2+D_c\Vert \nabla \cdot \xi_{\bf s}^m\Vert_{L^2}^2+D_c\Vert\mbox{rot}(\xi_{\bf s}^m)\Vert_{L^2}^2 = (\rho_{\bf s}^m - \delta_t \theta_{\bf s}^m + D_c \theta_{\bf s}^m,\xi_{\bf s}^m) \nonumber\\
&&\!\!\!\!\!\! +(\mathbf{u}^{m}\cdot({\bf s}^m-{\bf s}^{m-1})+(\xi^{m-1}_{\mathbf{s}}+\theta^{m-1}_{\mathbf{s}})\cdot{\bf u}^{m}+(\alpha n^{m}+\beta w^{m})({c}^{m}-{c}^{m-1}),\nabla \cdot\xi_{\bf s}^m)\nonumber\\
&&\!\!\!\!\!\! + ((\xi^{m}_{\bf u}+\theta^{m}_{\bf u})\cdot\mathbf{s}^{m-1}_h+(\alpha (\xi_{n}^{m}+\theta_{n}^{m})+\beta (\xi_{w}^{m}+\theta_{w}^{m}))c^{m-1}_{h},\nabla \cdot\xi_{\bf s}^m) \nonumber\\
&&\!\!\!\!\!\! +((\alpha n^{m}+\beta w^{m})(\xi_c^{m-1}+\theta_c^{m-1}),\nabla\cdot\xi_{\bf s}^{m})=\sum_{k=1}^{4} R_k.
\end{eqnarray}
Then, using the H\"older and Young inequalities, the equivalent norm in ${H}^1_{\bf s}(\Omega)$ given in (\ref{EQs}), as well as (\ref{ime1}), (\ref{StabStk1})-(\ref{StabStk2}) and (\ref{aprox01})-(\ref{aprox01-a}), we control the terms on the right hand side of (\ref{errs-int2}) as follows
\begin{eqnarray}\label{Ea1aS}
&R_1&\!\!\!\leq (\Vert \xi_{\bf s}^m\Vert_{L^2}+ \Vert \nabla\cdot \xi_{\bf s}^m\Vert_{L^2}+\Vert \mbox{rot}(\xi_{\bf s}^m)\Vert_{L^2})\Vert \rho_{\bf s}^m\Vert_{(H^1)'} + ( \Vert(\mathcal{I} - \mathbb{P}_{\bf s}) \delta_t {\bf s}^m \Vert_{L^2}+  D_c\Vert \theta_{\bf s}^m\Vert_{L^2})\Vert \xi_{\bf s}^m\Vert_{L^2}\nonumber\\
&&\!\!\!\leq \displaystyle\frac{1}{2}\Vert \xi_{\bf s}^m\Vert_{L^2}^2+\frac{D_c}{6}\Vert \nabla\cdot \xi_{\bf s}^m\Vert_{L^2}^2+\frac{D_c}{2}\displaystyle\Vert \mbox{rot } \xi_{\bf s}^m\Vert_{L^2}^2+ + C\Big(1+\frac{1}{D_c}\Big) \Vert \rho_{\bf s}^m\Vert_{(H^1)'}^2\nonumber\\
&&+C h^{2(r_4+1)}\Big[\Vert \delta_t {\bf s}^m\Vert_{H^{r_4+1}}^2+D_c^2\Vert {\bf s}^m \Vert_{H^{r_4 +1}}^2\Big] \nonumber\\
&&\!\!\!\leq \displaystyle\frac{1}{2}\Vert \xi_{\bf s}^m\Vert_{L^2}^2+\frac{D_c}{6}\Vert \nabla\cdot \xi_{\bf s}^m\Vert_{L^2}^2+\frac{D_c}{2}\displaystyle\Vert \mbox{rot } \xi_{\bf s}^m\Vert_{L^2}^2+C\Delta t \!\int_{t_{m-1}}^{t_m}\!\!\!\Vert \partial_{tt}{\bf s}(t)\Vert_{(H^1)'}^2 dt\nonumber\\
&& +C h^{2(r_4+1)}\Big[\displaystyle\frac{1}{\Delta t }\!\int_{t_{m-1}}^{t_m}\!\!\!\Vert \partial_t{\bf s}(t) \Vert_{H^{r_4+1}}^2  dt+\Vert {\bf s}^m \Vert_{H^{r_4 +1}}^2\Big],
\end{eqnarray}
\begin{eqnarray}\label{Ea1eS}
&R_2+R_4&\!\!\!\! 
\leq \displaystyle\frac{D_c}{6}\Vert \nabla\cdot \xi_{\bf s}^m\Vert_{L^2}^2+C (h^{2(r_4+1)} \Vert {\bf s}^{m-1}\Vert_{H^{r_4+1}}^2 +h^{2(r_3+1)} \Vert c^{m-1}\Vert_{H^{r_3+1}}^2 ) \Vert[\mathbf{u}^{m}, n^m, w^{m}]\Vert_{L^\infty}^2 \nonumber\\
&&+ C(\Vert [ c^m-c^{m-1},{\bf s}^m-{\bf s}^{m-1}]\Vert_{L^2}^2 + \Vert [\xi^{m-1}_{\bf s}, \xi^{m-1}_c]\Vert_{L^2}^2)  \Vert[\mathbf{u}^{m}, n^m, w^{m}]\Vert_{L^\infty}^2,
\end{eqnarray}
\begin{eqnarray}\label{Ea1e-newS}
& R_3&\!\!\! =  ({\bf s}_h^{m-1}\cdot\xi^{m}_{\mathbf{u}}+\theta^{m}_{\mathbf{u}}\cdot(\mathbb{P_{\bf s}}{\bf s}^{m-1}-\xi_{\bf s}^{m-1}),\nabla\cdot\xi_{\bf s}^m) +\alpha( c^{m-1}_h\xi_{n}^{m}+\theta_{n}^{m}(\mathbb{P}_cc^{m-1}-\xi_c^{m-1}),\nabla\cdot\xi_{\bf s}^m) \nonumber\\
&&+\beta( c^{m-1}_h\xi_{w}^{m}+\theta_{w}^{m}(\mathbb{P}_cc^{m-1}-\xi_c^{m-1}),\nabla\cdot\xi_{\bf s}^m)\nonumber\\
&&\!\!\!\leq (\Vert{\bf s}_h^{m-1}\Vert_{L^{10/3}}\Vert\xi_{\bf u}^m\Vert_{L^5}+\Vert\theta_{\bf u}^{m}\Vert_{L^\infty}\Vert\xi_{\bf s}^{m-1}\Vert_{L^2}+\Vert\theta_{\bf u}^m\Vert_{L^2}\Vert\mathbb{P}_{\bf s}{\bf s}^{m-1}\Vert_{L^\infty}+\alpha\Vert c_h^{m-1}\Vert_{L^{10/3}}\Vert\xi_n^{m}\Vert_{L^5}
\nonumber\\
&&
+\beta\Vert c_h^{m-1}\Vert_{L^{10/3}}\Vert\xi_w^{m}\Vert_{L^5}+\alpha\Vert\theta_{n}^{m}\Vert_{L^\infty}\Vert\xi_c^{m-1}\Vert_{L^2}+\alpha\Vert\theta_{n}^{m}\Vert_{L^2}\Vert\mathbb{P}_{c}{c}^{m-1}\Vert_{L^\infty}
\nonumber\\
&&
+\beta\Vert\theta_{w}^{m}\Vert_{L^\infty}\Vert\xi_c^{m-1}\Vert_{L^2}+\beta\Vert\theta_{w}^{m}\Vert_{L^2}\Vert\mathbb{P}_{c}{c}^{m-1}\Vert_{L^\infty})\Vert\nabla\cdot\xi_{\bf s}^m\Vert_{L^2}
\nonumber\\
&&\!\!\! \leq \displaystyle\frac{D_c}{6} \Vert \nabla \cdot\xi_{\bf s}^m \Vert_{L^2}^2+\frac{D_{\bf u}}{4}\Vert \nabla \xi_{\bf u}^{m}\Vert_{L^2}^2+C\Vert\xi_{\bf u}^{m}\Vert_{L^2}^2+C\Vert{\bf u}^m\Vert_{H^2}^2\Vert\xi_{\bf s}^{m-1}\Vert_{L^2}^2+\frac{D_n}{4} \Vert \nabla \xi_n^{m}\Vert_{L^2}^2 + C\Vert  \xi_n^{m}\Vert_{L^2}^2 \nonumber\\
&&  +Ch^{2(r+1)}\Vert{\bf s}^{m-1}\Vert_{H^2}\|[\mathbf{u}^{m}\!,\pi^{m}]\|_{H^{r+1}\times H^{r}}^2
+Ch^{2(r_1+1)}\Vert c^{m-1}\Vert_{H^2}^2\Vert n^{m}\Vert_{H^{r_1+1}}^2+C\Vert n^{m}\Vert_{H^2}^2\Vert \xi_c^{m-1}\Vert_{L^2}^2\nonumber\\
&&   
+C\Vert w^{m}\Vert_{H^2}^2\Vert \xi_c^{m-1}\Vert_{L^2}^2+Ch^{2(r_2+1)}\Vert c^{m-1}\Vert_{H^2}^2\Vert w^{m}\Vert_{H^{r_2+1}}^2+\frac{D_w}{4}\Vert\nabla \xi_{w}^{m}\Vert_{L^2}^2+C\Vert\xi_{w}^{m}\Vert_{L^2}^2,
\end{eqnarray}
where, in the last inequality of (\ref{Ea1e-newS}), the 3D interpolation inequality (\ref{in3Dl4}) and inductive hypothesis (\ref{IndHyp}) was used. Therefore, from (\ref{errs-int2})-(\ref{Ea1e-newS}), we arrive at
\begin{eqnarray}\label{errsfin}
&\displaystyle\frac{1}{2}&\!\!\!\!\!\delta_t \Vert\xi_{\bf s}^m\Vert_{L^2}^2+\displaystyle\frac{\Delta t}{2}\Vert\delta_t \xi_{\bf s}^m\Vert_{L^2}^2+ \frac{D_c}{2}\Vert \nabla \cdot \xi_{\bf s}^m\Vert_{L^2}^2+ \frac{D_c}{2}\Vert\mbox{rot}(\xi_{\bf s}^m)\Vert_{L^2}^2   \leq C\Vert [\xi^m_n,\xi^m_w]\Vert_{L^2}^2\nonumber\\
&&\!\!\!\!\!\!\!\!\!+ C\Delta t \!\int_{t_{m-1}}^{t_m}\!\!\!\Vert \partial_{tt} {\bf s}(t)\Vert_{(H^1)'}^2 dt+ C\Vert [\xi^m_{\mathbf{u}},\xi^m_{\mathbf{s}}]\Vert_{L^2}^2+ \frac{D_{\bf u}}{4}\Vert \nabla \xi_{\bf u}^{m} \Vert_{L^2}^2 + \frac{D_n}{4}\Vert\nabla \xi_{n}^{m} \Vert_{L^2}^2+ \frac{D_w}{4}\Vert \nabla \xi_{w}^{m} \Vert_{L^2}^2\nonumber\\
&&\!\!\!\!\!\!\!\!\! +C h^{2(r_4+1)}\Big[\displaystyle\frac{1}{\Delta t }\!\int_{t_{m-1}}^{t_m}\!\!\!\Vert \partial_t {\bf s}(t) \Vert_{H^{r_4+1}}^2  dt+\Vert {\bf s}^m \Vert_{H^{r_4 +1}}^2\Big]+ C h^{2(r_4+1)} \Vert {\bf s}^{m-1}\Vert_{H^{r_4+1}}^2  \Vert[\mathbf{u}^{m}, n^m, w^{m}]\Vert_{L^\infty}^2\nonumber\\
&&\!\!\!\!\!\!\!\!\!+C (h^{2(r_3+1)} \Vert c^{m-1}\Vert_{H^{r_3+1}}^2 +\Vert [ c^m-c^{m-1},{\bf s}^m-{\bf s}^{m-1}]\Vert_{L^2}^2)  \Vert[\mathbf{u}^{m}, n^m, w^{m}]\Vert_{L^\infty}^2 \nonumber\\
&&\!\!\!\!\!\!\!\!\!+ C \Vert [\xi^{m-1}_{\bf s}\!, \xi^{m-1}_c]\Vert_{L^2}^2  \Vert[\mathbf{u}^{m}\!, n^m\!, w^{m}]\Vert_{L^\infty}^2+Ch^{2(r+1)}\Vert{\bf s}^{m-1}\Vert_{H^2}\|[\mathbf{u}^{m}\!,\pi^{m}]\|_{H^{r+1}\times H^{r}}^2+ C\Vert{\bf u}^m\Vert_{H^2}^2\Vert\xi_{\bf s}^{m-1}\Vert_{L^2}^2   \nonumber\\
&&\!\!\!\!\!\!\!\!\!  
+C(h^{2(r_1+1)}\Vert n^{m}\Vert_{H^{r_1+1}}^2+h^{2(r_2+1)}\Vert w^{m}\Vert_{H^{r_2+1}}^2 )\Vert c^{m-1}\Vert_{H^2}^2+C(\Vert n^{m}\Vert_{H^2}^2+\Vert w^{m}\Vert_{H^2}^2)\Vert \xi_c^{m-1}\Vert_{L^2}^2.
\end{eqnarray}

\subsubsection{Proof of Theorem \ref{theo1N}}\label{sub1}
\begin{proof}First, we will prove (\ref{EEtheo1}). Observe that it holds
\begin{eqnarray}\label{mm1}
&\displaystyle\Delta t \sum_{m=1}^r \Vert &\!\!\!\!\!\! [n^m - n^{m-1},w^m - w^{m-1},c^m-c^{m-1},\mathbf{u}^m - \mathbf{u}^{m-1},{{\bf s}^m} - {\bf s}^{m-1}]\Vert_{L^2}^2\nonumber\\
&&\!\!\!\!\!\!\!\!\!\!\!\!\leq C(\Delta t)^4 \Vert [\partial_{tt} n, \partial_{tt} w, \partial_{tt} c,\partial_{tt}\mathbf{u},\partial_{tt} {\bf s}]\Vert^2_{L^2(L^2)} + C(\Delta t)^2 \Vert [\partial_t n,\partial_t w, \partial_t c,\partial_t \mathbf{u}, \partial_t {\bf s}]\Vert^2_{L^2(L^2)}.\ \ \
\end{eqnarray}
Indeed, proceeding as in (\ref{Ea1a}), with the space norm in $L^2$ instead of $(H^1)'$, we have
$$
\Vert \rho_n^m\Vert_{L^2}=\Vert \delta_t n^m - (\partial_t n)^m\Vert_{L^2}=\Big\Vert \frac{1}{\Delta t} (n^m- n^{m-1}) - (\partial_t n)^m\Big\Vert_{L^2}\leq C(\Delta t)^{1/2} \Big(\int_{t_{m-1}}^{t_m}\Vert \partial_{tt} n(t)\Vert_{L^2}^2 dt\Big)^{1/2}, 
$$
which implies 
$$
\Delta t \sum_{m=1}^r \Vert n^m - n^{m-1}\Vert_{L^2}^2\leq C(\Delta t)^4 \Vert \partial_{tt} n\Vert^2_{L^2(L^2)} + C(\Delta t)^2 \Vert \partial_t n\Vert^2_{L^2(L^2)}.
$$
Analogously, we obtain the estimate for $w,c, \mathbf{u}$ and ${\bf s}$.
Now, adding  the inequalities (\ref{errnfin}), (\ref{errwfin}), (\ref{errcfin}), (\ref{errufin}) and (\ref{errsfin}), multiplying the resulting expression  by $\Delta t$, adding from $m=1$ to $m=r$, using (\ref{mm1})  and Lemma \ref{uen}  (recalling that $[\xi^0_n, \xi^0_w,\xi^0_c,\xi^0_{\mathbf{u}},\xi^0_{\bf s}]=[0,0,0,{\bf 0},{\bf 0}]$), we get
\begin{eqnarray}\label{GWG}
&\|[\xi^r_n, \xi^r_w,&\!\!\!\!\!\xi^r_c,\xi^r_{\mathbf{u}},\xi^r_{\bf s}]\|_{L^2}^2+\displaystyle\frac{\Delta t}{2}\sum_{m=1}^{r}\Vert[\nabla\xi_n^m,\nabla\xi_w^m,\nabla\xi_c^m,\nabla \xi_{\bf u}^m,\nabla\cdot\xi_{\bf s}^m,\mbox{rot}(\xi_{\bf s}^m)]\Vert_{L^2}^2\nonumber\\
&&\!\!\!\!\!\leq C_1((\Delta t)^2+(\Delta t)^4)+C_2(h^{2(r_1+1)}+h^{2(r_2+1)}+h^{2(r_3+1)}+h^{2(r_4+1)}+h^{2(r+1)})\nonumber\\
&&+C_3\Delta t\sum_{m=1}^{r-1}\Vert[\xi_{n}^{m}, \xi_{w}^{m}, \xi_{c}^{m}, \xi_{\bf s}^{m},\xi_{\bf u}^{m}]\Vert^2_{L^2}+C_4\Delta t\Vert[\xi_{n}^{r}, \xi_{w}^{r}, \xi_{c}^{r}, \xi_{\bf s}^{r},\xi_{\bf u}^{r}]\Vert^2_{L^2}.
\end{eqnarray}
Therefore, if $\frac{1}{2} - C_4 \Delta t > 0$, by applying Lemma \ref{Diego2} to (\ref{GWG}), (\ref{EEtheo1}) is
concluded.\\

Now we prove (\ref{EEtheo2}). {{Subtracting (\ref{errpi-int}) at time $t=t_{m-1}$ from (\ref{errpi-int}) at time $t=t_{m}$ and multiplying the resulting expression by $1/\Delta t$, we have that
\begin{equation}\label{NeeA}
(\bar{\pi},\nabla \cdot \delta_t \xi^m_{\mathbf{u}})=0 \ \ \forall \bar{\pi}\in \Pi_h.
\end{equation}
Then, taking $\bar{\mathbf{u}}=\delta_t \xi_{\mathbf{u}}^m$ in (\ref{erru-int}), $\bar{\pi}=\xi_{\pi}^m$ in (\ref{NeeA}) and adding the resulting expressions, we get
\begin{eqnarray}\label{NeeB}
&\displaystyle \frac{D_{\bf u}}{2}&\!\!\!\!\!\delta_t \Vert \xi_{\mathbf{u}}^{m}\Vert_{H^1}^2+\displaystyle\frac{\Delta t D_{\bf u}}{2}\Vert\delta_t \xi_{\mathbf{u}}^m\Vert_{H^1}^2 + \Vert \delta_t \xi_{\mathbf{u}}^{m}\Vert_{L^2}^2 = (\rho_{\mathbf{u}}^m-\delta_t \theta_{\mathbf{u}}^m,\delta_t \xi_{\mathbf{u}}^m)\nonumber\\
&&\!\!\!\!\! - kB(\mathbf{u}^m-\mathbf{u}^{m-1},{\mathbf{u}}^m,\delta_t \xi_{\mathbf{u}}^m)   -kB(\xi^{m-1}_{\mathbf{u}}\!+\theta^{m-1}_{\mathbf{u}}\!,{\mathbf{u}}^m,\delta_t \xi_{\mathbf{u}}^m) -kB(\mathbf{u}^{m-1}_h\!,\xi^{m}_{\mathbf{u}}+\theta^m_{\mathbf{u}},\delta_t \xi_{\mathbf{u}}^m) \nonumber\\
&&\!\!\!\!\! + ((\gamma(\xi^{m}_n+\theta^{m}_n)+\lambda(\xi^{m}_w+\theta^{m}_w)) \nabla \phi,\delta_t \xi_{\mathbf{u}}^m):=(f,\delta_t \xi_{\mathbf{u}}^m)\leq \frac{1}{2}(\Vert\delta_t \xi_{\mathbf{u}}^m\Vert_{L^2}^2+\Vert f\Vert_{L^2}^2).
\end{eqnarray}
Now, we bound $\Vert f\Vert_{L^2}$. First, note that, for all $\varphi\in L^2,$ we have 
\begin{eqnarray}\label{NNNab0}
&\vert(\rho_{\mathbf{u}}^m&\!\!\!-\delta_t \theta_{\mathbf{u}}^m,\varphi)- kB(\mathbf{u}^m-\mathbf{u}^{m-1},{\mathbf{u}}^m,\varphi)+((\gamma(\xi^{m}_n+\theta^{m}_n)+\lambda(\xi^{m}_w+\theta^{m}_w)) \nabla \phi,\varphi)\vert\nonumber\\
&&\!\!\!\!\! \leq C  \Vert \nabla \phi\Vert_{L^\infty}(\gamma\Vert\xi_{n}^m\Vert_{L^2}+\lambda\Vert\xi_{w}^m\Vert_{L^2}+\gamma\Vert \theta_{n}^{m}\Vert_{L^2}+\lambda\Vert \theta_{w}^{m}\Vert_{L^2})\Vert \varphi\Vert_{L^2}) \nonumber\\
&&+C(\Vert \rho_{\mathbf{u}}^m\Vert_{L^2}+ \Vert \delta_t \theta_{\mathbf{u}}^m\Vert_{L^2}+ k \Vert \mathbf{u}^{m}-\mathbf{u}^{m-1}\Vert_{L^2} \Vert \nabla \mathbf{u}^{m}\Vert_{L^\infty})\Vert \varphi\Vert_{L^2},
\end{eqnarray}
\begin{eqnarray}\label{NNNabc1}
&\vert kB(\xi^{m-1}_{\mathbf{u}}\!+\theta^{m-1}_{\mathbf{u}}\!,{\mathbf{u}}^m,\varphi)\vert&\!\!\! \leq kC(\Vert \xi_{\mathbf{u}}^{m-1}\Vert_{H^1}+\Vert \theta_{\mathbf{u}}^{m-1}\Vert_{H^1}) \Vert \mathbf{u}^{m}\Vert_{L^\infty\cap W^{1,3}}\Vert \varphi\Vert_{L^2}
\end{eqnarray}
and}
\begin{eqnarray}\label{NNNabcd2}
&\vert B(\mathbf{u}^{m-1}_h\!,\xi^{m}_{\mathbf{u}}+\theta^m_{\mathbf{u}},\varphi)\vert&\!\!\!=\vert B(\mathbb{P}_{\mathbf{u}} \mathbf{u}^{m-1}\!,\xi^{m}_{\mathbf{u}}+\theta^m_{\mathbf{u}},\varphi)-B(\xi_{\mathbf{u}}^{m-1}\!,\xi^{m}_{\mathbf{u}}+\theta^m_{\mathbf{u}},\varphi)\vert \nonumber\\
&&\hspace{-3.3 cm} \leq k( C\Vert \mathbb{P}_{\mathbf{u}} \mathbf{u}^{m-1}\Vert_{L^\infty\cap W^{1,3}}(\Vert \theta_{\mathbf{u}}^{m}\Vert_{H^1}+ \Vert \xi_{\mathbf{u}}^{m}\Vert_{H^1})\Vert \varphi\Vert_{L^2}+ C \Vert \xi_{\mathbf{u}}^{m-1}\Vert_{H^1} \Vert \theta_{\mathbf{u}}^{m}\Vert_{L^\infty\cap W^{1,3}} \Vert \varphi\Vert_{L^2}\nonumber\\
&&\hspace{-3 cm}+C\Vert \nabla \xi_{\mathbf{u}}^{m}\Vert_{L^3}\Vert \xi_{\mathbf{u}}^{m-1}\Vert_{L^{6}}\Vert \varphi\Vert_{L^2}+C\Vert \nabla \xi_{\mathbf{u}}^{m-1}\Vert_{L^3}\Vert \xi_{\mathbf{u}}^{m}\Vert_{L^{6}}\Vert \varphi\Vert_{L^2})
\nonumber\\
&&\hspace{-3.3 cm} \leq k( C\Vert \mathbb{P}_{\mathbf{u}} \mathbf{u}^{m-1}\Vert_{L^\infty\cap W^{1,3}}(\Vert \theta_{\mathbf{u}}^{m}\Vert_{H^1}+ \Vert \xi_{\mathbf{u}}^{m}\Vert_{H^1})\Vert \varphi\Vert_{L^2}+ C \Vert \xi_{\mathbf{u}}^{m-1}\Vert_{H^1} \Vert \theta_{\mathbf{u}}^{m}\Vert_{L^\infty\cap W^{1,3}} \Vert \varphi\Vert_{L^2}\nonumber\\
&&\hspace{-3 cm}+C\Vert \xi_{\mathbf{u}}^{m}\Vert_{W^{1,6}}^{1/2}\Vert \xi_{\mathbf{u}}^{m}\Vert_{H^1}^{1/2}\Vert \xi_{\mathbf{u}}^{m-1}\Vert_{H^{1}}\Vert \varphi\Vert_{L^2}+C\Vert \xi_{\mathbf{u}}^{m-1}\Vert_{W^{1,6}}^{1/2}\Vert \xi_{\mathbf{u}}^{m-1}\Vert_{H^1}^{1/2}\Vert \xi_{\mathbf{u}}^{m}\Vert_{H^1}\Vert \varphi\Vert_{L^2}),
\end{eqnarray}
where, in (\ref{NNNabcd2}), the 3D interpolation inequality (\ref{in3D}) was used. Thus, from (\ref{NNNab0})-(\ref{NNNabcd2}) we deduce
\begin{eqnarray}
\Vert f\Vert_{L^2}&\!\!\!\!\!\!\!=\!\!\!\!\!\!&\sup\{\vert (f,\varphi)\vert\ :\ \varphi\in L^2(\Omega),\ \Vert \varphi\Vert_{L^2}\leq 1\}\nonumber\\
&&\!\!\!\!\!\!\!\!\!\leq  C ( \Vert \rho_{\mathbf{u}}^m\Vert_{L^2}+ \Vert \delta_t \theta_{\mathbf{u}}^m\Vert_{L^2}+  k\Vert \mathbf{u}^{m}-\mathbf{u}^{m-1}\Vert_{L^2} \Vert \nabla \mathbf{u}^{m}\Vert_{L^\infty}\nonumber\\
&&\!\!\!\!\!+(\Vert \xi_{\mathbf{u}}^{m-1}\Vert_{H^1}+\Vert \theta_{\mathbf{u}}^{m-1}\Vert_{H^1}) \Vert \mathbf{u}^{m}\Vert_{L^\infty\cap W^{1,3}} + \Vert \xi_{\mathbf{u}}^{m-1}\Vert_{H^1} \Vert \theta_{\mathbf{u}}^{m}\Vert_{L^\infty\cap W^{1,3}}\nonumber\\
&&\!\!\!\! \! + \Vert \mathbb{P}_{\mathbf{u}} \mathbf{u}^{m-1}\Vert_{L^\infty\cap W^{1,3}}(\Vert \theta_{\mathbf{u}}^{m}\Vert_{H^1}+ \Vert \xi_{\mathbf{u}}^{m}\Vert_{H^1})  +\Vert \xi_{\mathbf{u}}^{m}\Vert_{W^{1,6}}^{1/2}\Vert \xi_{\mathbf{u}}^{m}\Vert_{H^1}^{1/2}\Vert \xi_{\mathbf{u}}^{m-1}\Vert_{H^{1}} \nonumber\\
&&\!\!\!\!\! +\Vert \xi_{\mathbf{u}}^{m-1}\Vert_{W^{1,6}}^{1/2}\Vert \xi_{\mathbf{u}}^{m-1}\Vert_{H^1}^{1/2}\Vert \xi_{\mathbf{u}}^{m}\Vert_{H^1}+(\Vert\xi_{n}^m\Vert_{L^2}+\Vert\xi_{w}^m\Vert_{L^2}+\Vert \theta_{n}^{m}\Vert_{L^2}+\Vert \theta_{w}^{m}\Vert_{L^2}) \Vert \nabla \phi\Vert_{L^\infty}).\label{z1}
\end{eqnarray}
Therefore from (\ref{NeeB}) and (\ref{z1}), we arrive at 
\begin{eqnarray}\label{NNabcde}
&\displaystyle \frac{D_{\bf u}}{2}&\!\!\!\!\!\delta_t \Vert \xi_{\mathbf{u}}^{m}\Vert_{H^1}^2 + \frac{1}{2}\Vert \delta_t \xi_{\mathbf{u}}^{m}\Vert_{L^2}^2\leq \frac{1}{2} \Vert f\Vert_{L^2}^2 \leq C ( \Vert \rho_{\mathbf{u}}^m\Vert_{L^2}^2+ \Vert \delta_t \theta_{\mathbf{u}}^m\Vert_{L^2}^2+  \Vert \mathbf{u}^{m}-\mathbf{u}^{m-1}\Vert_{L^2}^2 \Vert \nabla \mathbf{u}^{m}\Vert_{L^\infty}^2\nonumber\\
&&\!\!\!\!\!\!\!\!+(\Vert \xi_{\mathbf{u}}^{m-1}\Vert_{H^1}^2+\Vert \theta_{\mathbf{u}}^{m-1}\Vert_{H^1}^2) \Vert \mathbf{u}^{m}\Vert_{L^\infty\cap W^{1,3}}^2 + \Vert \xi_{\mathbf{u}}^{m-1}\Vert_{H^1}^2 \Vert \theta_{\mathbf{u}}^{m}\Vert_{L^\infty\cap W^{1,3}}^2\nonumber\\
&& \!\!\!\!\!\!\!\! + \Vert \mathbb{P}_{\mathbf{u}} \mathbf{u}^{m-1}\Vert_{L^\infty\cap W^{1,3}}^2(\Vert \theta_{\mathbf{u}}^{m}\Vert_{H^1}^2+ \Vert \xi_{\mathbf{u}}^{m}\Vert_{H^1}^2))  + \varepsilon_1(\Vert \xi_{\mathbf{u}}^{m}\Vert_{W^{1,6}}^2 +  \Vert \xi_{\mathbf{u}}^{m-1}\Vert_{W^{1,6}}^2 )\nonumber\\
&&\!\!\!\!\!\!\!\! + C( \Vert \xi_{\mathbf{u}}^{m}\Vert_{H^1}^2(\Vert \xi_{\mathbf{u}}^{m-1}\Vert_{H^1}^4+\Vert \xi_{\mathbf{u}}^{m-1}\Vert_{H^1}^2\Vert \xi_{\mathbf{u}}^{m}\Vert_{H^1}^2) + (\Vert\xi_{n}^m\Vert_{L^2}^2+\Vert\xi_{w}^m\Vert_{L^2}^2+\Vert \theta_{n}^{m}\Vert_{L^2}^2+\Vert \theta_{w}^{m}\Vert_{L^2}^2) \Vert \nabla \phi\Vert_{L^\infty}^2). \ \ \ \ \ \ \
\end{eqnarray}
On the other hand, applying Lemma 11 of \cite{GV} to (\ref{erru-int})-(\ref{errpi-int}), we have
\begin{equation}\label{errPI1}
\Vert [\xi_{\mathbf{u}}^m,\xi_{\pi}^m]\Vert_{W^{1,6}\times L^6}^2 \leq C (\Vert \delta_t \xi_{\mathbf{u}}^m\Vert_{L^2}^2+ \Vert f\Vert_{L^2}^2),
\end{equation}
and then, adding (\ref{NNabcde}) with (\ref{errPI1})$\times \varepsilon_2$ (for $0<\varepsilon_2<1$), we arrive at}
\begin{eqnarray}\label{errPI2}
& \displaystyle \frac{D_{\bf u}}{2} &\!\!\!\! \delta_t \Vert \xi_{\mathbf{u}}^{m}\Vert_{H^1}^2 + \frac{1}{2}\Vert \delta_t \xi_{\mathbf{u}}^{m}\Vert_{L^2}^2+ \varepsilon_2\Vert [\xi_{\mathbf{u}}^m,\xi_{\pi}^m]\Vert_{W^{1,6}\times L^6}^2 \leq C (\varepsilon_2 \Vert \delta_t \xi_{\mathbf{u}}^m\Vert_{L^2}^2+ \Vert \rho_{\mathbf{u}}^m\Vert_{L^2}^2+ \Vert \delta_t \theta_{\mathbf{u}}^m\Vert_{L^2}^2\nonumber\\
&&\!\!\!\!\!\!\!+  \Vert \mathbf{u}^{m}-\mathbf{u}^{m-1}\Vert_{L^2}^2 \Vert \nabla \mathbf{u}^{m}\Vert_{L^\infty}^2+(\Vert \xi_{\mathbf{u}}^{m-1}\Vert_{H^1}^2+\Vert \theta_{\mathbf{u}}^{m-1}\Vert_{H^1}^2) \Vert \mathbf{u}^{m}\Vert_{L^\infty\cap W^{1,3}}^2 + \Vert \xi_{\mathbf{u}}^{m-1}\Vert_{H^1}^2 \Vert \theta_{\mathbf{u}}^{m}\Vert_{L^\infty\cap W^{1,3}}^2\nonumber\\
&&\!\!\!\!\!\!\!  + \Vert \mathbb{P}_{\mathbf{u}} \mathbf{u}^{m-1}\Vert_{L^\infty\cap W^{1,3}}^2(\Vert \theta_{\mathbf{u}}^{m}\Vert_{H^1}^2+ \Vert \xi_{\mathbf{u}}^{m}\Vert_{H^1}^2) ) + \varepsilon_1(\Vert \xi_{\mathbf{u}}^{m}\Vert_{W^{1,6}}^2 +  \Vert \xi_{\mathbf{u}}^{m-1}\Vert_{W^{1,6}}^2 )\nonumber\\
&& \!\!\!\!\!\!\! + C(\Vert \xi_{\mathbf{u}}^{m}\Vert_{H^1}^2(\Vert \xi_{\mathbf{u}}^{m-1}\Vert_{H^1}^4+\Vert \xi_{\mathbf{u}}^{m-1}\Vert_{H^1}^2\Vert \xi_{\mathbf{u}}^{m}\Vert_{H^1}^2) + (\Vert\xi_{n}^m\Vert_{L^2}^2+\Vert\xi_{w}^m\Vert_{L^2}^2+\Vert \theta_{n}^{m}\Vert_{L^2}^2+\Vert \theta_{w}^{m}\Vert_{L^2}^2) \Vert \nabla \phi\Vert_{L^\infty}^2). \ \ \ \ \ \ \ \
\end{eqnarray}
We consider $\varepsilon_2$ small enough in order to absorb the term $C \varepsilon_2\Vert \delta_t \xi_{\mathbf{u}}^m\Vert_{L^2}^2$ at the right hand side, and $\varepsilon_1$ small enough with respect to $\varepsilon_2$. Moreover, notice that 
\begin{equation}\label{pre2}
\Vert \xi_{\mathbf{u}}^{m-1}\Vert_{H^1}^4+\Vert \xi_{\mathbf{u}}^{m-1}\Vert_{H^1}^2\Vert \xi_{\mathbf{u}}^{m}\Vert_{H^1}^2\leq C.
\end{equation}
Indeed, denoting $h_{\max}:=\max\{h^{r_1+1},h^{r_2+1},h^{r_3+1},h^{r_4+1},
h^{r+1} \}$, from estimate (\ref{EEtheo1}) we have in particular that $\|\xi^m_{\mathbf{u}}\|_{H^1}^2 \leq C(T) \Big(\Delta t +\frac{1}{\Delta t}h_{\max}^2\Big)$, which implies that
$$
\Vert \xi_{\mathbf{u}}^{m-1}\Vert_{H^1}^4+\Vert \xi_{\mathbf{u}}^{m-1}\Vert_{H^1}^2\Vert \xi_{\mathbf{u}}^{m}\Vert_{H^1}^2 \leq C(T) \Big(\Delta t +\frac{1}{\Delta t}h_{\max}^2\Big) ^2.
$$
Therefore, we conclude (\ref{pre2}) under the  condition
\begin{equation}\label{HHa}
\frac{h_{\max}^4}{(\Delta t)^2}\leq C.
\end{equation}
On the other hand, from (\ref{EEtheo1}) we also have $\|\xi^m_{\mathbf{u}}\|_{L^2}^2 \leq C(T) \Big((\Delta t)^2 +h_{\max}^2\Big)$ for each $m$. Then, by using the inverse inequality $\Vert \xi^m_{\mathbf{u}}\Vert_{H^1}\leq h^{-1} \Vert \xi^m_{\mathbf{u}}\Vert_{L^2}$ we obtain
$$
\Vert \xi_{\mathbf{u}}^{m-1}\Vert_{H^1}^4+\Vert \xi_{\mathbf{u}}^{m-1}\Vert_{H^1}^2\Vert \xi_{\mathbf{u}}^{m}\Vert_{H^1}^2 \leq C(T) \frac{1}{h^4}\Big((\Delta t)^2 +h_{\max}^2\Big)^2.
$$
Thus, we conclude (\ref{pre2}) under the condition 
\begin{equation}\label{HHb}
\frac{(\Delta t)^4}{h^4}\leq C.
\end{equation}
Therefore, {if $\Delta t$ and $h$ are less than or equal to 1,} we conclude (\ref{pre2})  because for any choice of $(\Delta t,h)$ either (\ref{HHa}) or (\ref{HHb}) holds. Therefore, multiplyng (\ref{errPI2}) by $\Delta t$, adding from $m=1$ to $m=s$, bounding the terms $\Vert \rho_{\mathbf{u}}^m\Vert_{L^2}^2$ and $\Vert \delta_t \theta_{\mathbf{u}}^m\Vert_{L^2}^2$ as in (\ref{Ea1au}), using  (\ref{StabStk1}), (\ref{StabStk2}), (\ref{aprox01})$_{1,2}$, (\ref{mm1}), (\ref{EEtheo1}), (\ref{pre2}) and taking into account that $\xi^0_{\mathbf{u}}={\bf 0}$, we conclude (\ref{EEtheo2}).
\end{proof}

Finally, we will check the validity of (\ref{IndHyp}).
The idea is to derive (\ref{IndHyp}) by using (\ref{EEtheo1}) recursively. First, notice that $\Vert \mathbb{P}_{c}{c}^{m-1}\Vert_{L^{10/3}}\leq \Vert {c}\Vert_{L^\infty(L^{10/3})}:=C_0$ for all $m\geq 1$, from which we deduce that $\Vert {c}^{0}_h\Vert_{L^{10/3}}= \Vert \mathbb{P}_{c} {c}_0 \Vert_{L^{10/3}}\leq C_0\leq C_0+1:=K$. Moreover, 
$\Vert {c}^{m-1}_h\Vert_{L^{10/3}}\leq \Vert \xi^{m-1}_{c}\Vert_{L^{10/3}}+\Vert \mathbb{P}_{c}{c}^{m-1}\Vert_{L^{10/3}}$, and then,  it is enough to show that $\Vert  \xi^{m-1}_{c}\Vert_{L^{10/3}}\leq 1$ for all $m\geq 2.$ For that, we consider two cases: First,  if $\frac{(\Delta t)^{1/2}}{h^{3/5}}\leq C,$ then, by using the inverse inequality  $\Vert \xi^m_{c}\Vert_{L^{10/3}}\leq h^{-3/5} \Vert \xi^m_{c}\Vert_{L^2},$ recalling that $h_{\max}:=\max\{h^{r_1+1},h^{r_2+1},h^{r_3+1}, h^{r_4+1}, 
h^{r+1} \}= h^k$ (for some $k\geq 2$), and from the calculation of the $l^\infty(L^2)$-norm in (\ref{EEtheo1}), we get 
\begin{eqnarray}\label{j1}
\Vert \xi^1_{c}\Vert_{L^{10/3}}\leq \frac{1}{h^{3/5}}\Vert \xi^1_{c}\Vert_{L^2}\leq C(T,\Vert {c}^0_h\Vert_{L^{10/3}})\frac{1}{h^{3/5}}(\Delta t+h_{\max})\leq  C(T,K)(C(\Delta t)^{1/2}+h^{k-3/5}).
\end{eqnarray}
In another case, if $\frac{(\Delta t)^{1/2}}{h^{3/5}}$ is not bounded, then $\frac{h_{\max}}{\Delta t}\leq \frac{h^{6/5}}{\Delta t}\leq C;$ thus, from the calculation of the $l^2(H^1)$-norm in (\ref{EEtheo1}), we obtain
\begin{eqnarray}\label{j2}
\Vert \xi^1_{c}\Vert^2_{L^{10/3}}\leq C(T,\Vert {c}^0_h\Vert_{L^{10/3}})\frac{1}{\Delta t}((\Delta t)^{2}+h_{\max}^2)\leq  C(T,K)(\Delta t+C h_{\max}).
\end{eqnarray}
In any case, assuming $\Delta t$ and $h$ small enough (without any additional restriction relating the discrete parameters $(\Delta t,h)$), we conclude that $\Vert  \xi^{1}_{c}\Vert_{L^{10/3}}\leq 1$, which implies $\Vert {c}^{1}_h\Vert_{L^{10/3}}\leq K$.  Analogously, by using $\Vert {c}^{1}_h\Vert_{L^{10/3}}\leq K$, we can obtain $\Vert  \xi^{2}_{c}\Vert_{L^{10/3}}\leq 1$, and therefore, $\Vert {c}^{2}_h\Vert_{L^{10/3}}\leq K$. Arguing recursively we conclude that $\Vert {c}^{m-1}_h\Vert_{L^{10/3}}\leq K$, for all $m\geq 1$. Analogously, the inductive hypothesis on ${\bf s}^{m-1}_h$ can be verified.

\section{Numerical simulations} \label{SS5NS}
In this section, we present two numerical experiments. The first one is used to verify that the scheme (\ref{scheme1}) gives a good approximation to chemotaxis and interspecies competition phenomena in a liquid environment, including numerical evidence of the convergence of the densities $n^m_h$ and $w^m_h$ towards steady states proved in \cite{HYJ}, depending on the $a_i$-values ($i=1,2$).  The second one has been considered in order to check numerically the error estimates proved in our theoretical analysis. All the numerical results are computed by using the software Freefem++. We have considered finite element spaces for $n,w,c,{\bf s},{\bf u}$ and $\pi$ generated by $\mathbb{P}_1,\mathbb{P}_1,\mathbb{P}_1,\mathbb{P}_1,\mathbb{P}_1-bubble,\mathbb{P}_1$-continuous, respectively.\\

\underline{\it{Test 1}}:\
We consider the cubic domain $\Omega=[0,1]^3,$ and the initial conditions (see Figure \ref{ICs}):
\begin{eqnarray*}
n_0 &=& \sum_{i=1}^3  \Big(120 \, \text{exp} (-3(x+r_i)^2-12(y-0.5)^2-3(z+1)^2)\Big),\\
w_0 &=& \sum_{i=1}^3  \Big(8 \, \text{exp} (-5(x+r_i)^2-10(y-0.5)^2-12(z-1)^2)\Big),\\
c_0 &=& \sum_{i=1}^3  \Big(180 \, \text{exp} \displaystyle{(-2(x-2.5)^2-1.5(y-{\sigma}_i)^2-2(z-\tau_i)^2)\Big)},\\
\mathbf{u}_0&=&{\bf 0},
\end{eqnarray*}
where $r_1=0.2$, $r_2=0.5$, $r_3=1.2$, $\sigma_1=1.5$, $\sigma_2=1.8$, $\sigma_3=2.5$, $\tau_1=0.5$, $\tau_2=0.8$ and $\tau_3=1.5$. The numerical solutions are computed with a $10\times10\times10$ grid in space and time step $\Delta t= 0.001$. Additionally, we consider the parameters values $\chi_1=12$, $\chi_2=15$, $D_n=6$, $D_w=8$, $D_c=1$, $D_{\bf u}=1$, $\mu_1=0.5$, $\mu_2=0.3$, $\alpha=6$, $k=1$, $\beta=8$, $\gamma=1$, $\lambda=1$ and $\phi(x,y,z)=-9.8z$. We want to see the behaviour of the scheme (\ref{scheme1}) for different values of the  parameters $a_1$ and $a_2$ which quantify the strength of competitions.
\\

\bigskip
\begin{minipage}{\textwidth}
	\begin{center}
\begin{tabular}{ccc}
 \includegraphics[width=38mm]{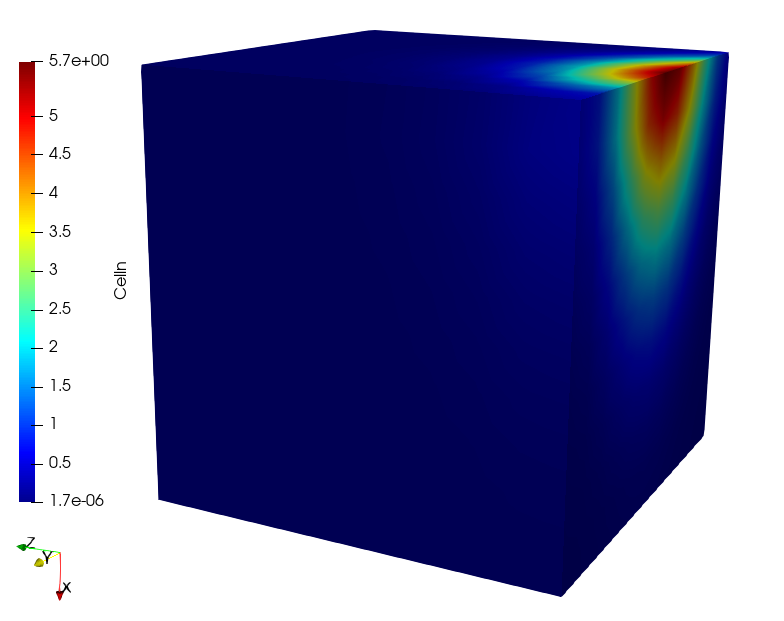} & \includegraphics[width=38mm]{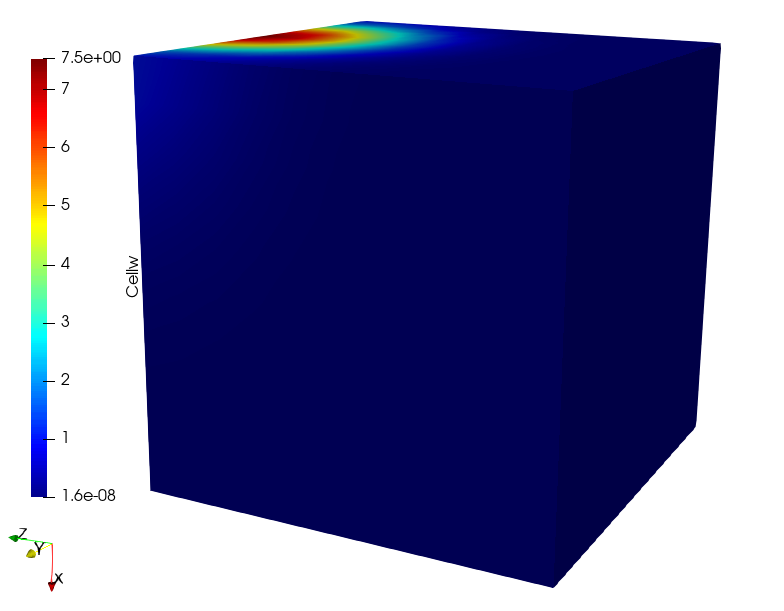} & \includegraphics[width=38mm]{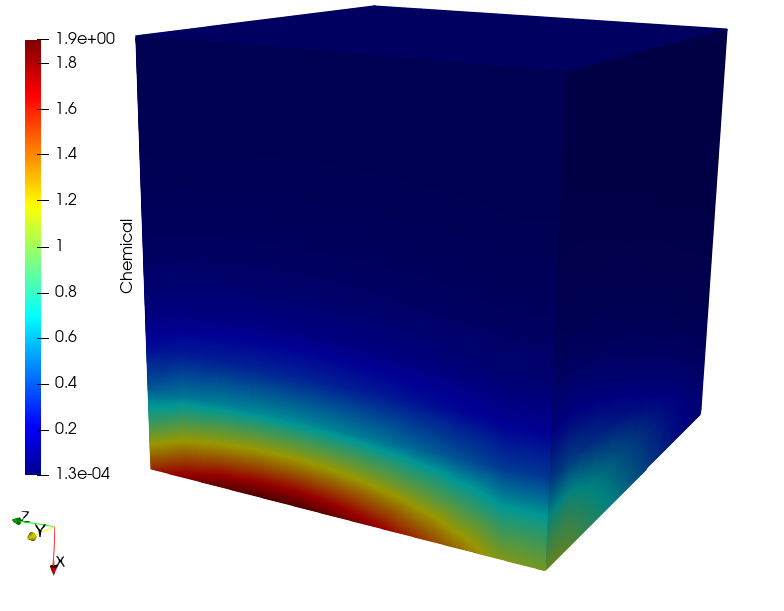}\\[1mm]  (i) $n^0_h$  & (ii) $w^0_h$ & (iii) $c^0_h $   \\[-2mm]
\end{tabular}
\end{center}
\figcaption{Intial conditions in Test 1.} \label{ICs}
\end{minipage}

\

\

 In the first case,  we consider $a_1=0.25$ and $a_2=0.3$. The evolution of the velocity field is shown in Figure \ref{fig:VF}, while the evolution  of $n^m_h$, $w^m_h$ and $c^m_h$ can be seen in Figure  \ref{fig:NWC}.  At the initial times, we observe that both species of organisms begin to direct their movement in the direction of greater concentration of the chemical substance and both species end up agglomerating at one of the lower ends of the domain, which  occurs because the cross-diffusion terms (or chemotaxis terms) are the dominant. However, as time progresses, the chemical substance is consumed by the species, which causes that the cross-diffusion loses strength, and the self-diffusion of the species begins to dominate, and therefore they begin to distribute themselves homogeneously over the domain. Furthermore, as both parameters $a_1$ and $a_2$ are less than $1$, neither species ends up dominating the other, and thus neither species tends towards extinction. In fact, in Figures \ref{fig:NWC} and \ref{fig:DT}, it can be observed that, for large times, the total density of each species of organisms tends to the constant states provided in \cite{HYJ}. Finally, the movement of the species cause some changes in the velocity field and, for large times, ${\bf u}^m_h$ tends to $0$, which is in agreement with the results proven in \cite{HYJ}.

\begin{center}
\begin{tabular}{ccc}
\includegraphics[width=38mm]{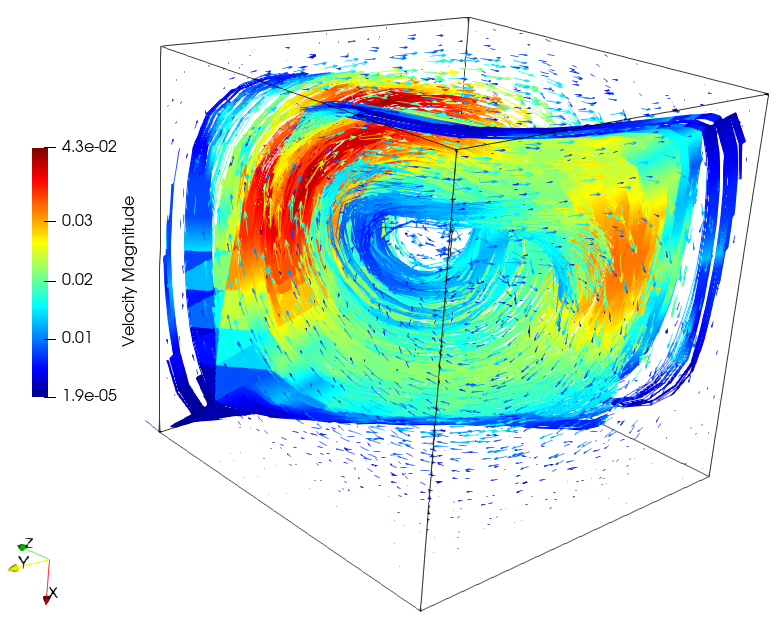} &
\includegraphics[width=38mm]{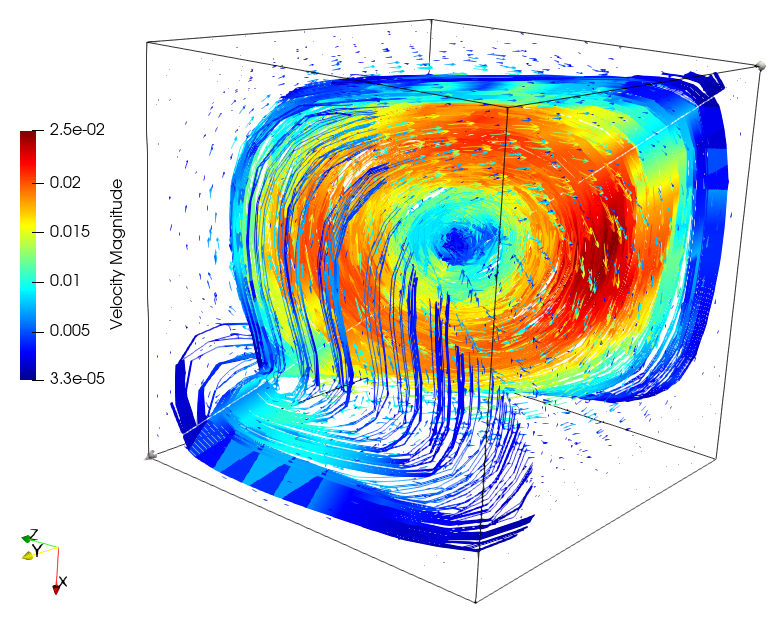} & \includegraphics[width=38mm]{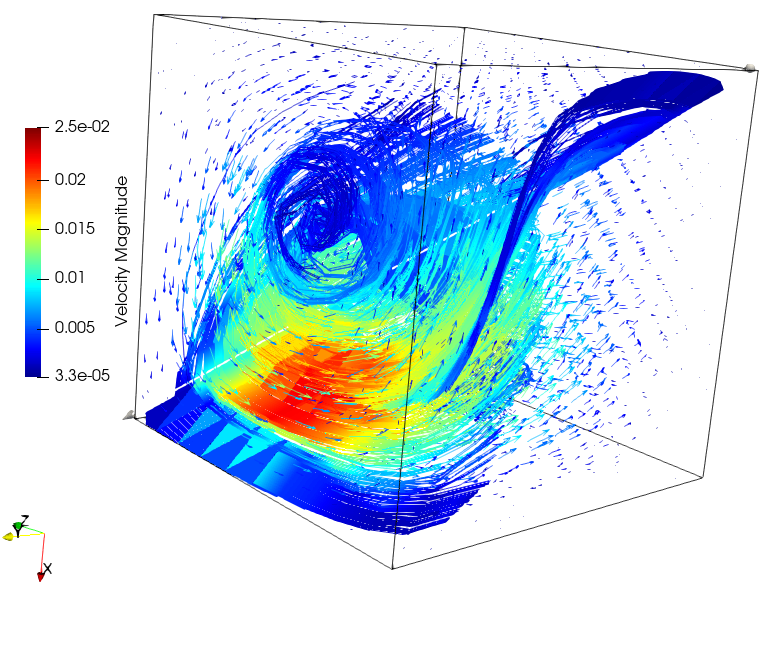} \\[-1mm]
 a)  $t=10^{-2}$ & b) $t=2.5\times10^{-2}$ & c)  $t=5\times10^{-2}$ \\[1mm]
\includegraphics[width=38mm]{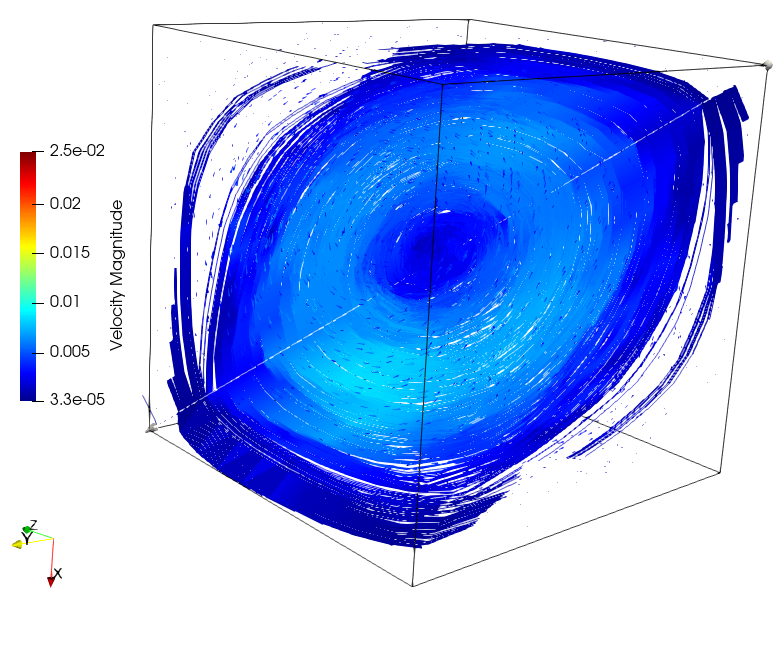} &\includegraphics[width=38mm]{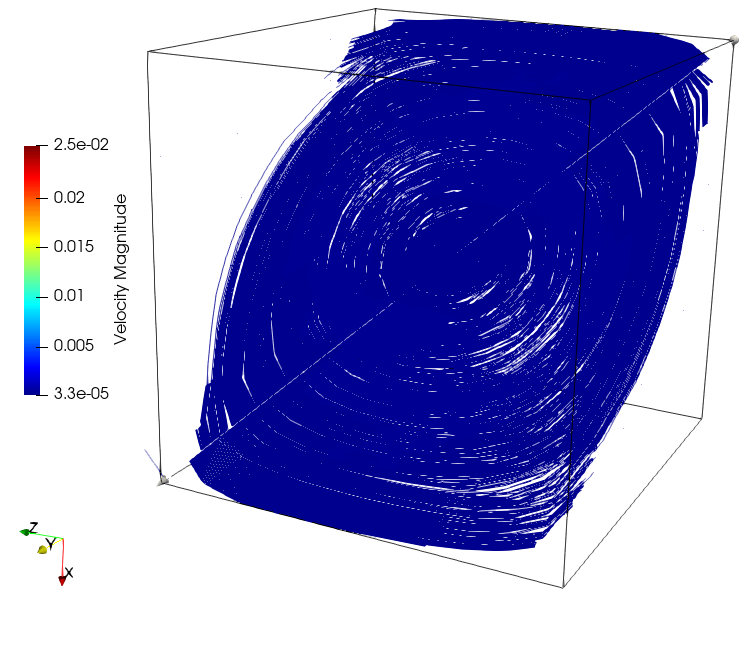} &  \includegraphics[width=38mm]{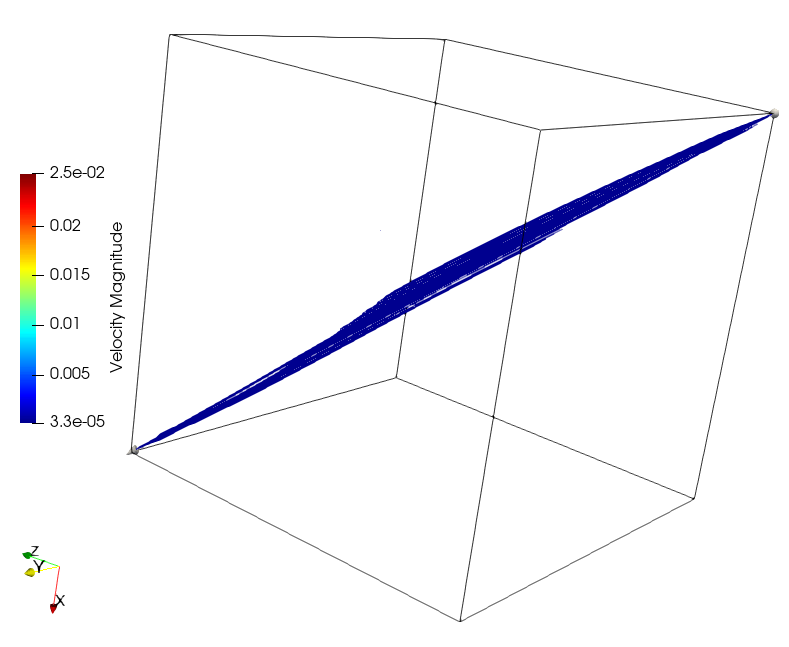} \\[-5mm]
 d) $t=15\times10^{-2}$ &  e) $t=5\times10^{-1}$ & f) $t=15\times10^{-1}$ 
\\ [1mm]
 \end{tabular}
\figcaption{Evolution in time of the velocity field ${\bf u}^m_h$, for $a_1=0.25$ and $a_2=0.3$.}\label{fig:VF}
\end{center}

\bigskip
\begin{minipage}{\textwidth}
	\begin{center}
\begin{tabular}{cccc}
\hspace{-0.8cm}\textbf{Time}&  $n^m_h$  & $w^m_h$ &  $c^m_h$  \\[1mm]
\hspace{-0.8cm}{$t=10^{-2}$} &\includegraphics[width=38mm]{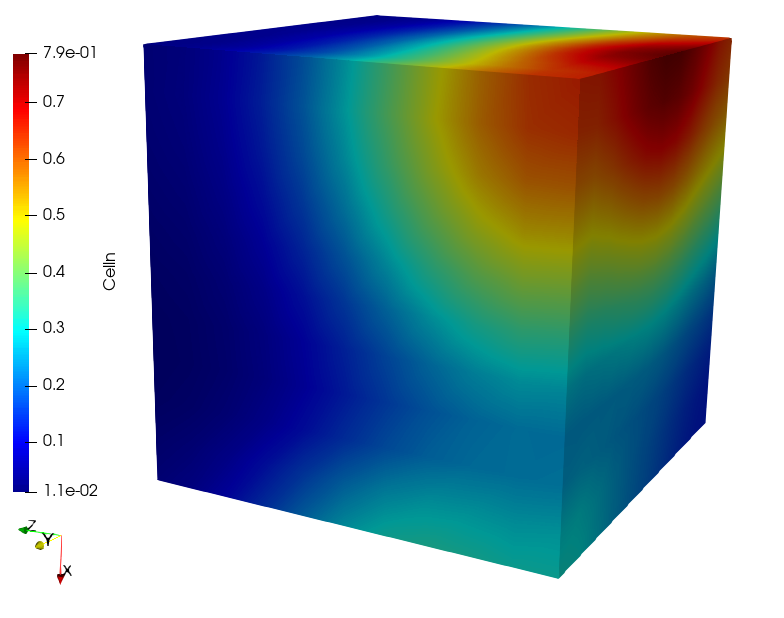} & \includegraphics[width=38mm]{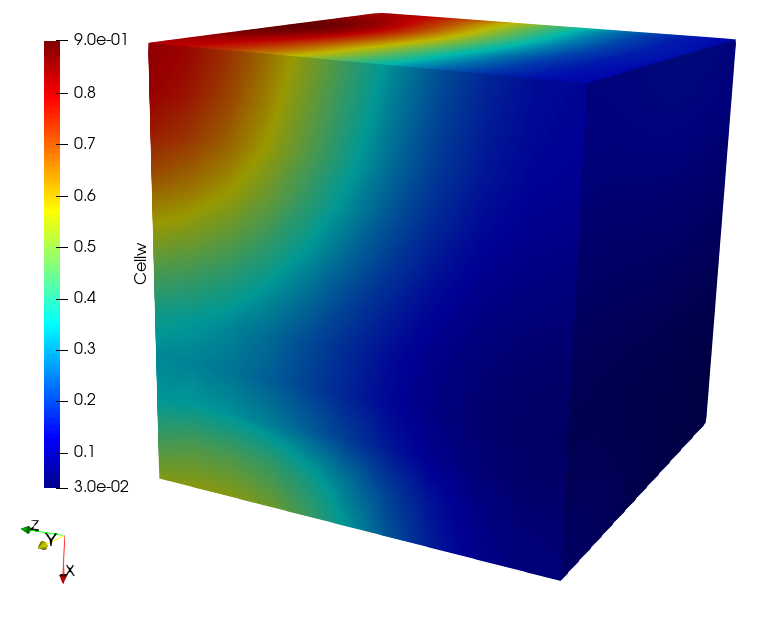}& \includegraphics[width=38mm]{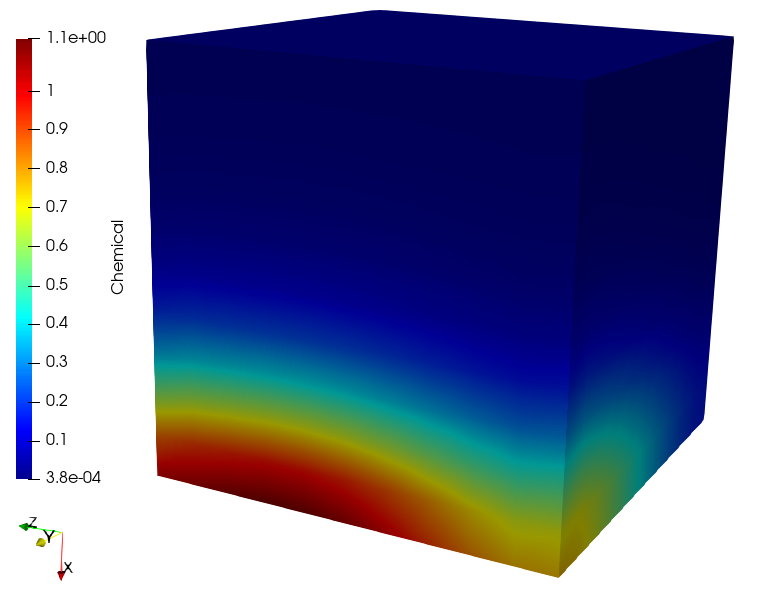}  \\[1mm]
\hspace{-0.8cm}{$t=2.5\times10^{-2}$} &\includegraphics[width=38mm]{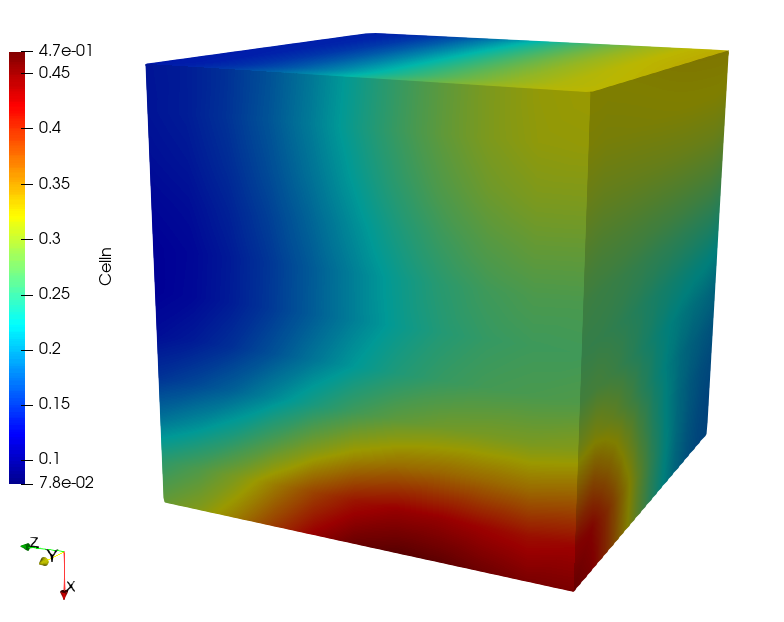} & \includegraphics[width=38mm]{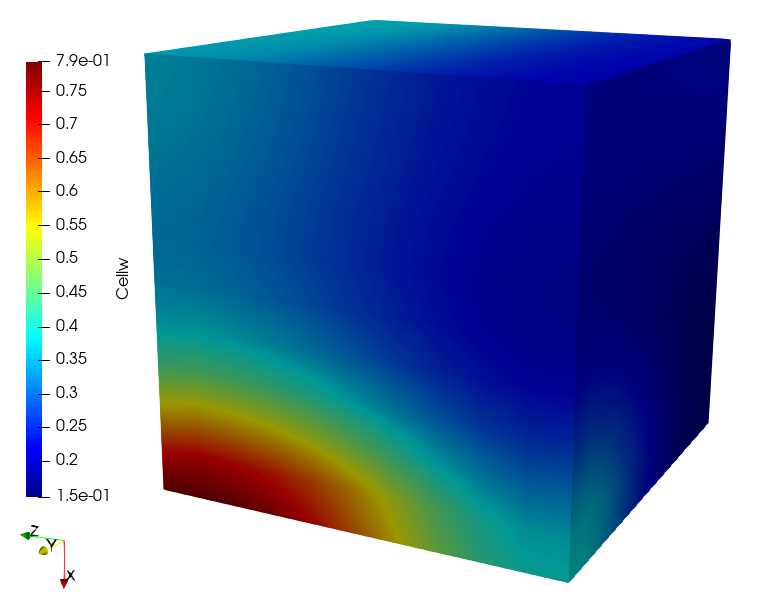}& \includegraphics[width=38mm]{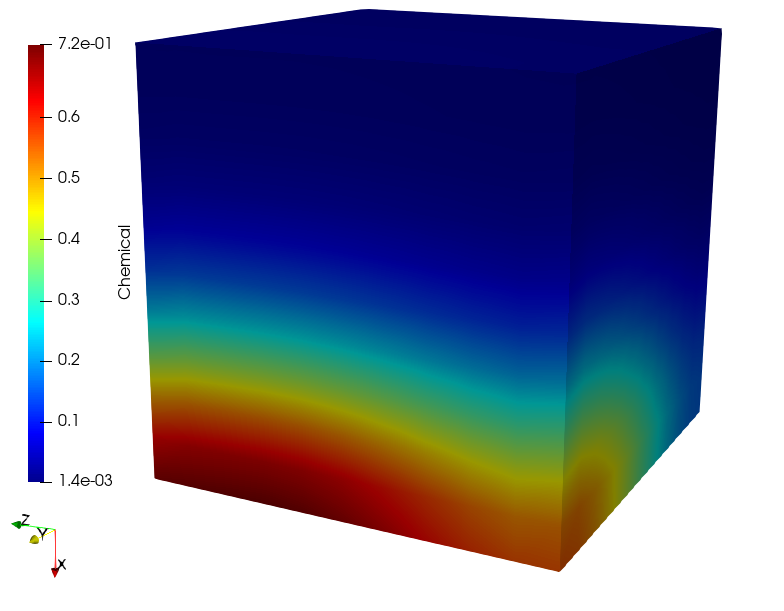}  \\[1mm]
\hspace{-0.8cm}{$t=5\times10^{-2}$}&\includegraphics[width=38mm]{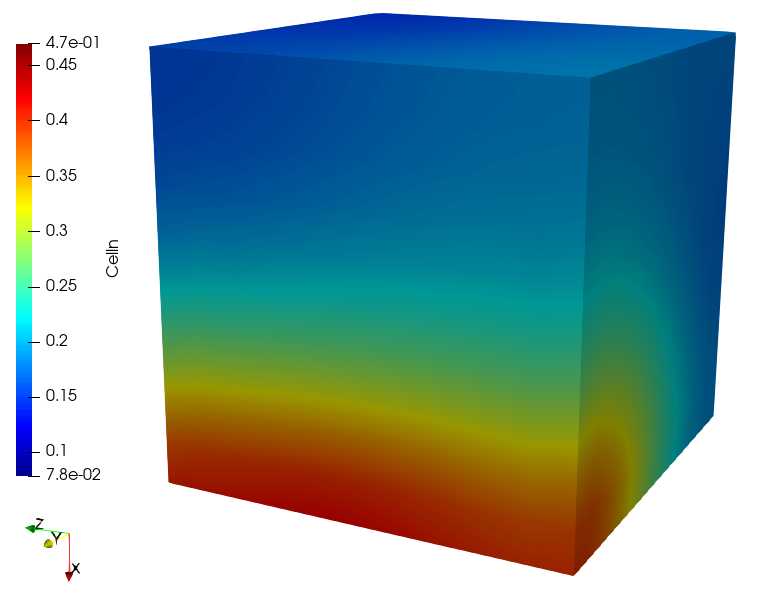} & \includegraphics[width=38mm]{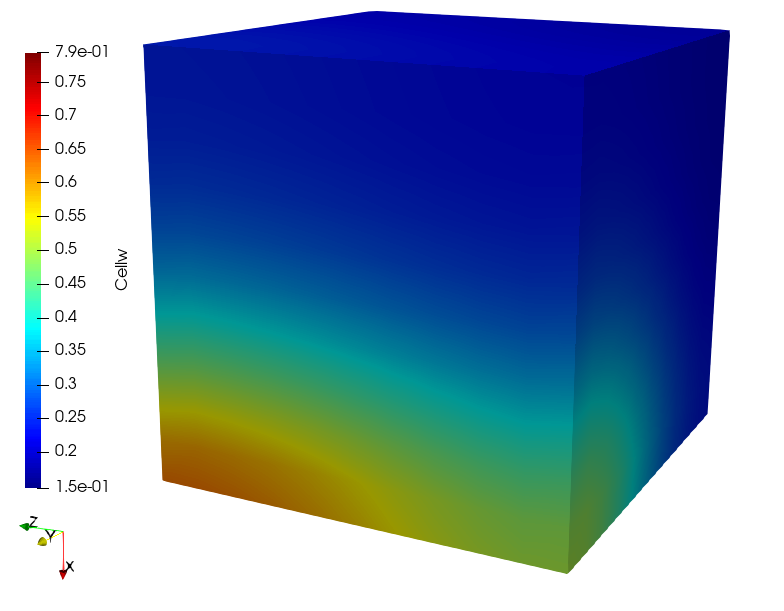}& \includegraphics[width=38mm]{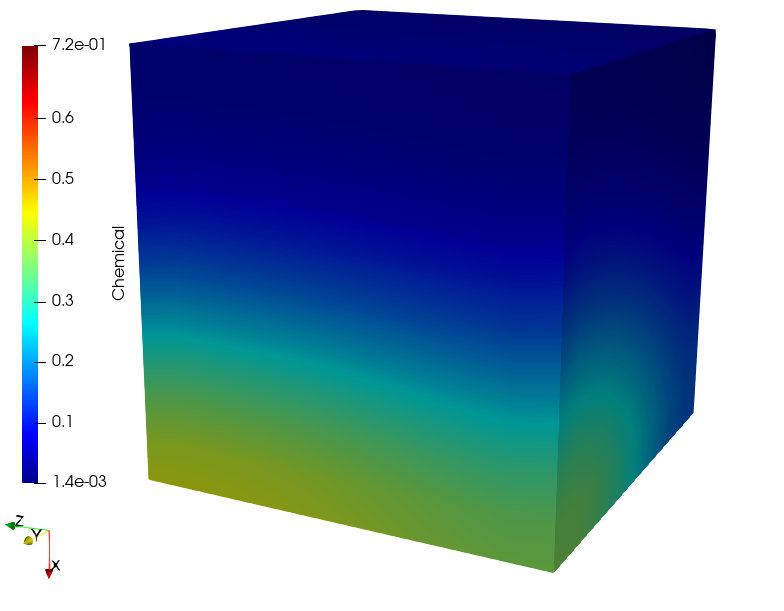}  \\[1mm]
\hspace{-0.8cm}{$t=15\times10^{-2}$}&\includegraphics[width=38mm]{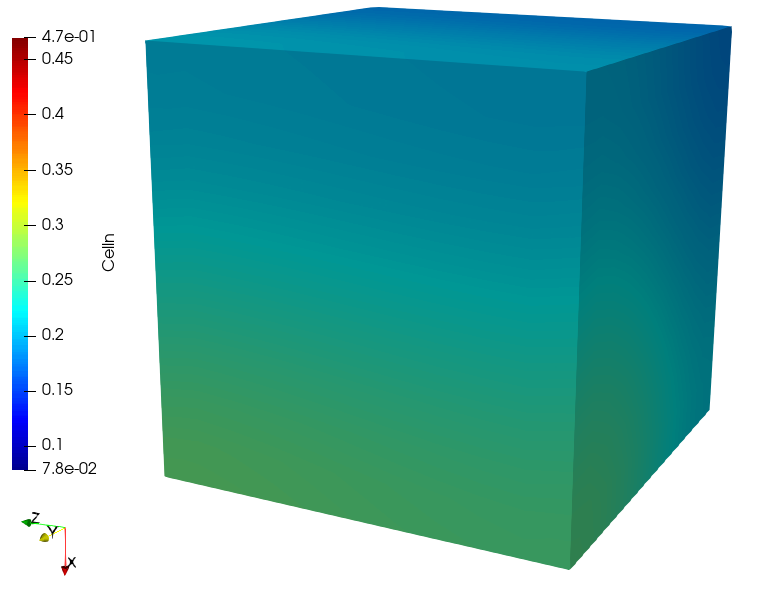} & \includegraphics[width=38mm]{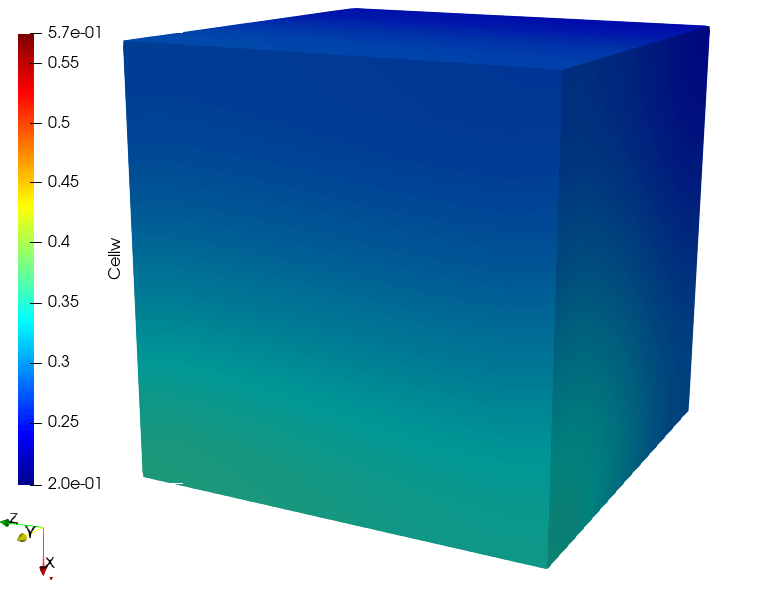} & \includegraphics[width=38mm]{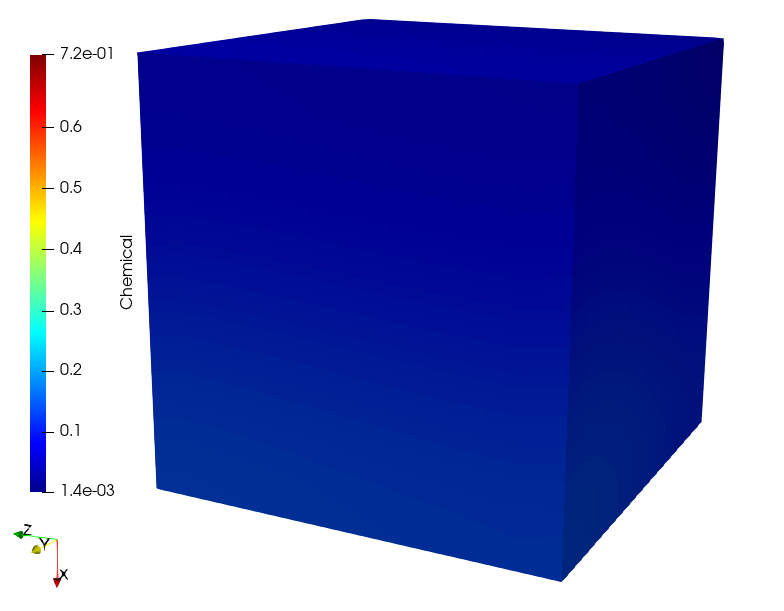} \\[1mm]
\hspace{-0.8cm}{$t=18$} &\includegraphics[width=38mm]{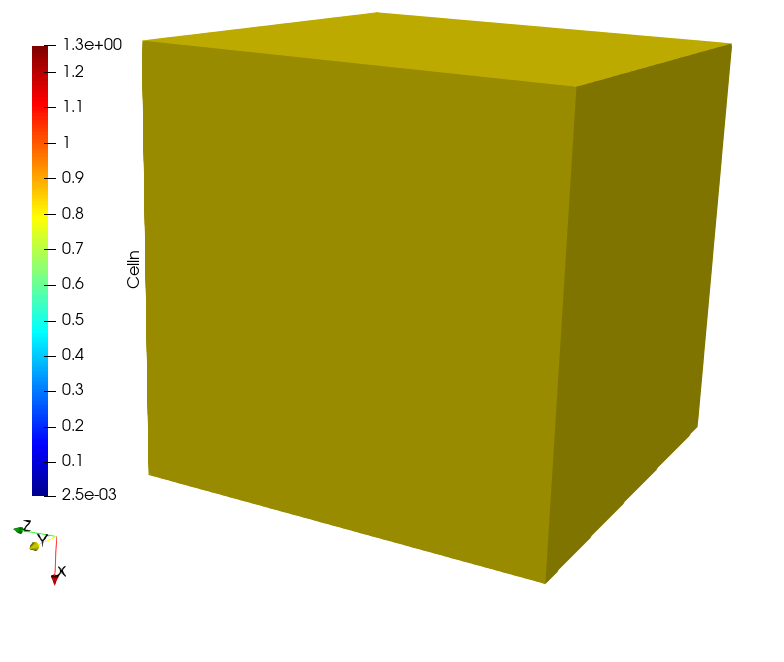} & \includegraphics[width=38mm]{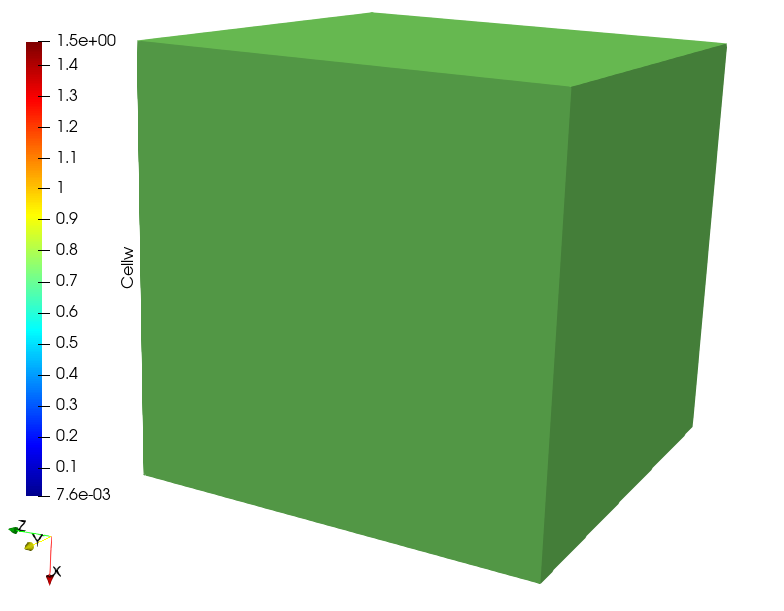}& \includegraphics[width=38mm]{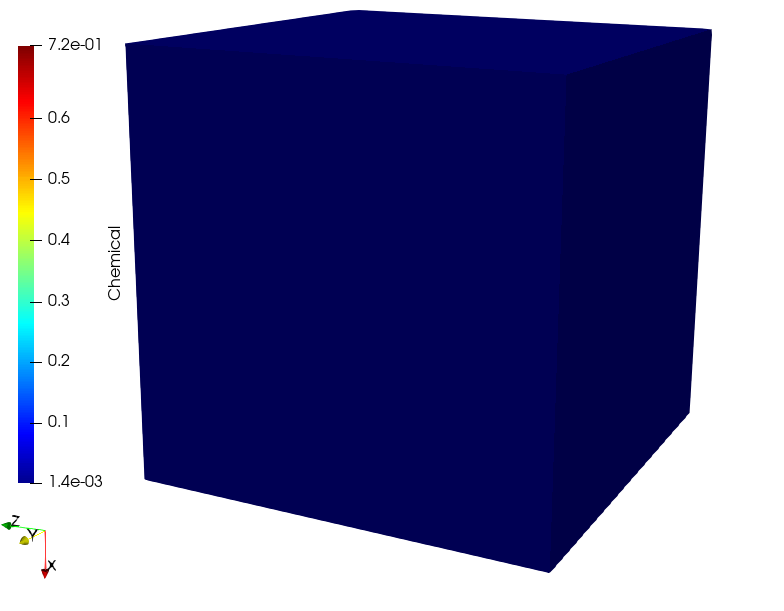}  \\[-2mm]
\end{tabular}
\end{center}
\figcaption{Evolution in time of $n^m_h$, $w^m_h$ and $c^m_h$, for $a_1=0.25$ and $a_2=0.3$.} \label{fig:NWC}
\end{minipage}

\

\

In the second case,  we consider $a_1=2$ and $a_2=0.3$, and the evolution in time for $n^m_h$, $w^m_h$, $c^m_h$ and ${\bf u}^m_h$ are presented in Figures  \ref{fig:NWC1} and  \ref{fig:VF1}. We observe a similar behaviour as in the previous case, with the difference that, in this case, the parameter $a_1>1$ and $a_2<1$,  and thus the species $w$ dominates species $n$, and the latter ends up extinguishing.  In fact, in Figures \ref{fig:NWC1} and \ref{fig:DT1}, it can be observed that, for large times, the total density of  species $w$ tends to  $1$, while  the total density of  species $n$ tends to $0$, which is agreement with the results reported in \cite{HYJ}.

\

\

\begin{minipage}{\textwidth}
	\begin{center}
\begin{tabular}{cccc}
\hspace{-0.8cm}\textbf{Time}& $n^m_h$ & $w^m_h$  & $c^m_h$  \\[1mm]
\hspace{-0.8cm}{$t=2.5\times10^{-2}$} &\includegraphics[width=38mm]{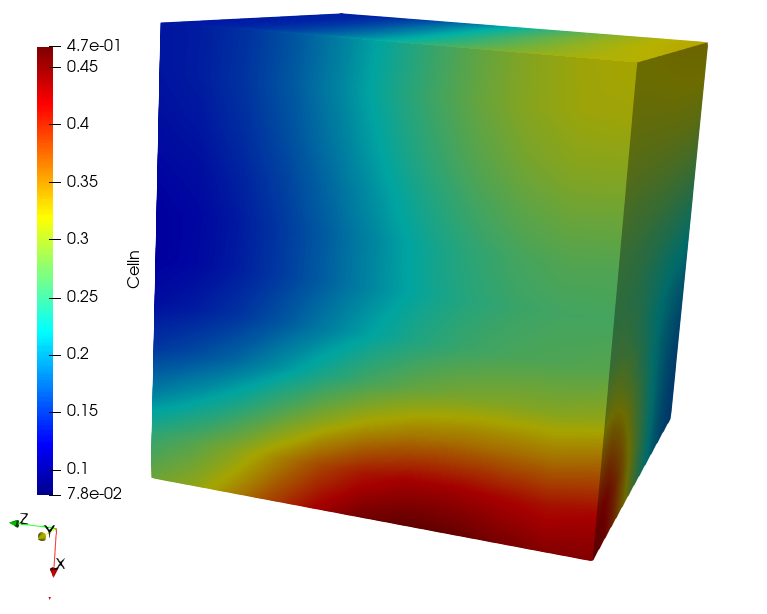} & \includegraphics[width=38mm]{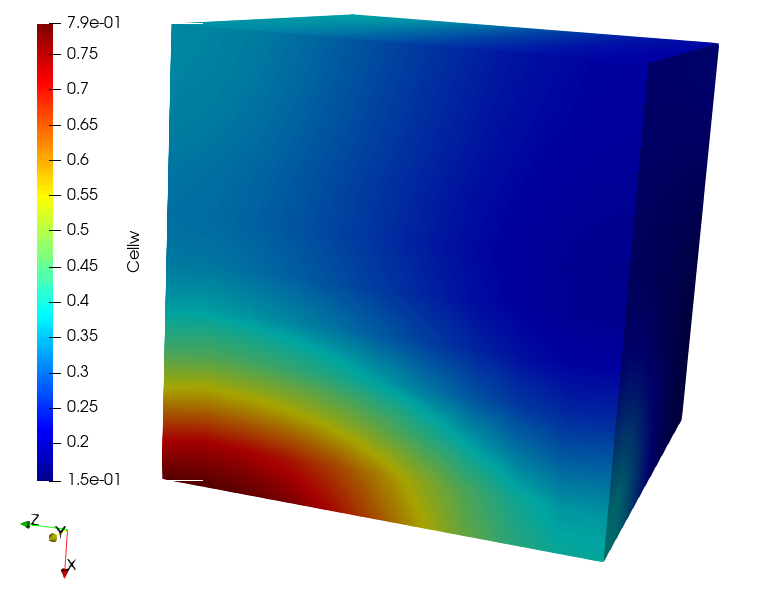}& \includegraphics[width=38mm]{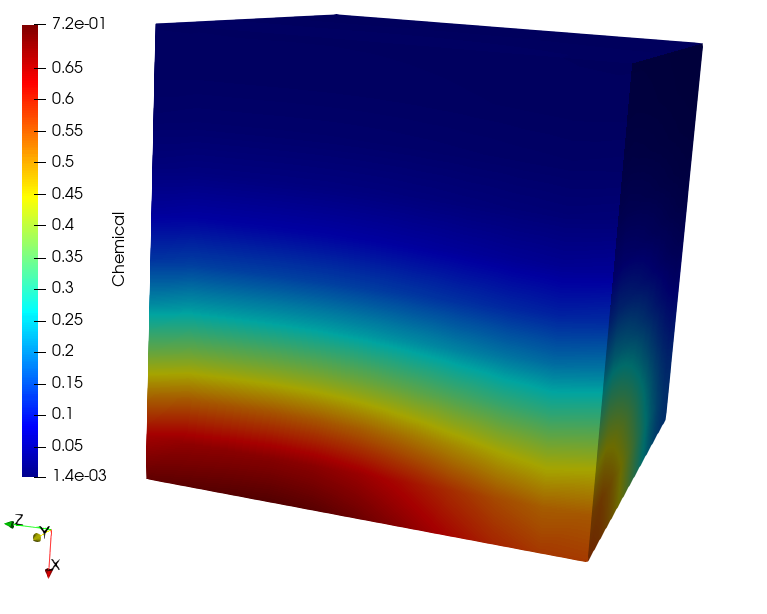}  \\[1mm]
\hspace{-0.8cm}{$t=21$} &\includegraphics[width=38mm]{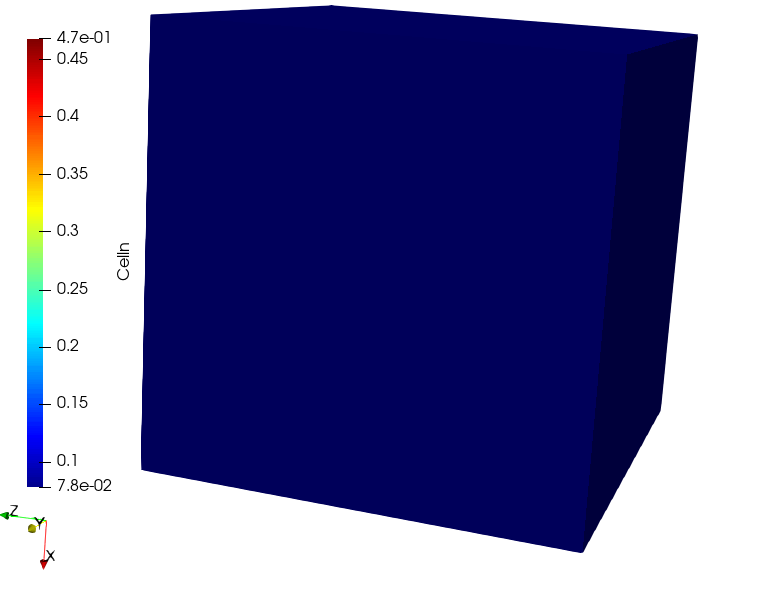} & \includegraphics[width=38mm]{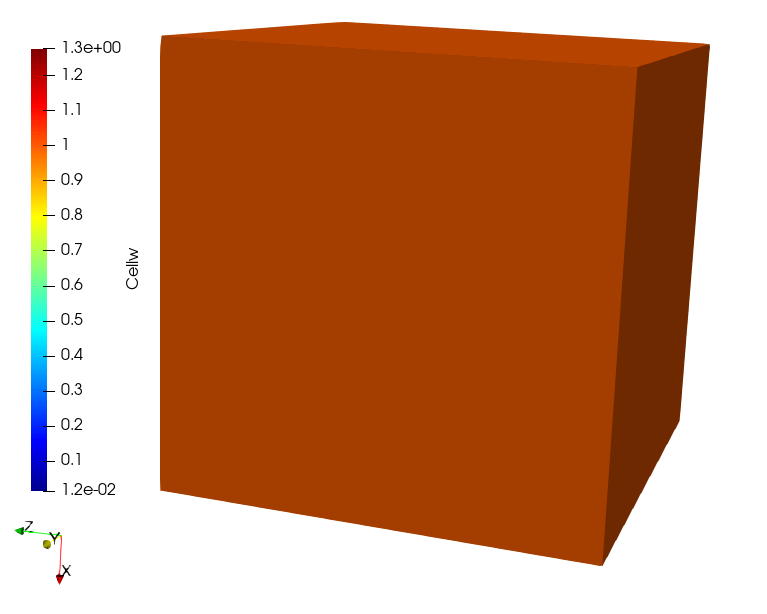}& \includegraphics[width=38mm]{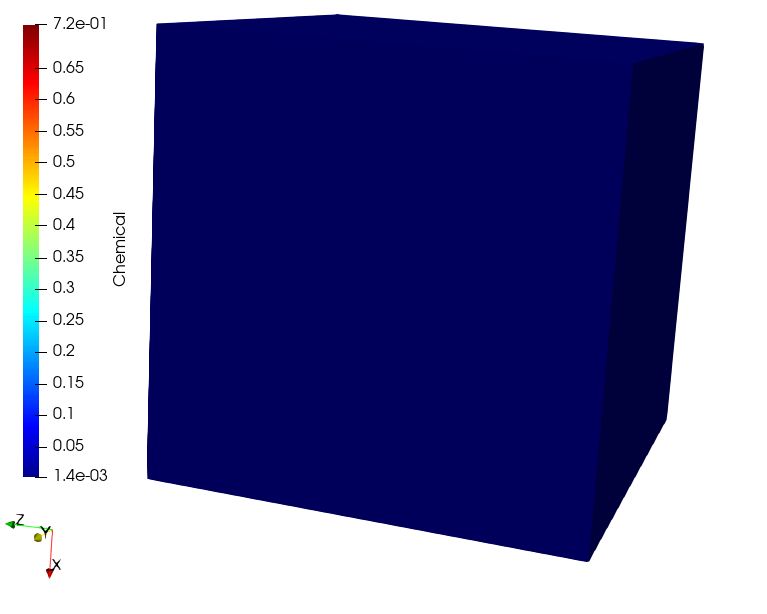}  \\[-2mm]
\end{tabular}
\end{center}
\figcaption{Evolution in time of $n^m_h$, $w^m_h$ and $c^m_h$, for $a_1=2$ and $a_2=0.3$.} \label{fig:NWC1}
\end{minipage}

\

\

\begin{center}
\begin{tabular}{cccc}
\includegraphics[width=38mm]{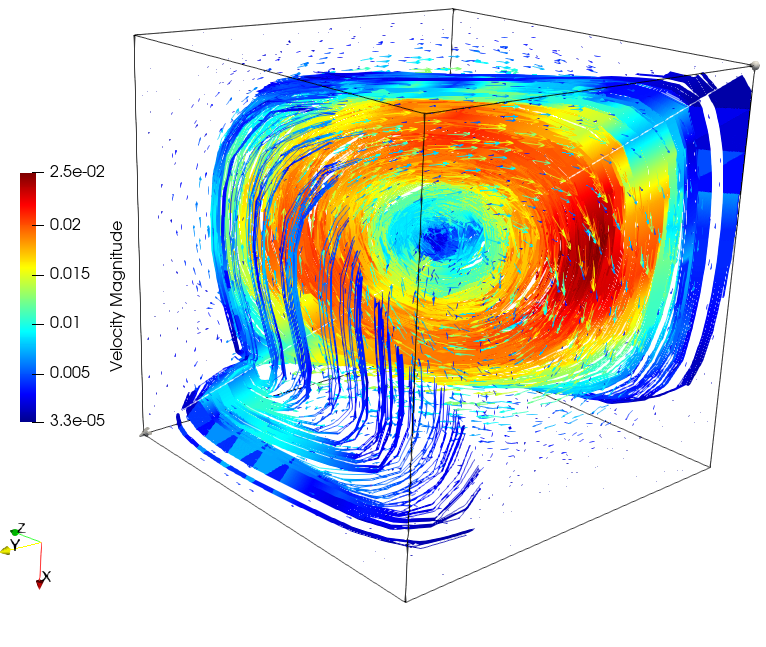} & \includegraphics[width=38mm]{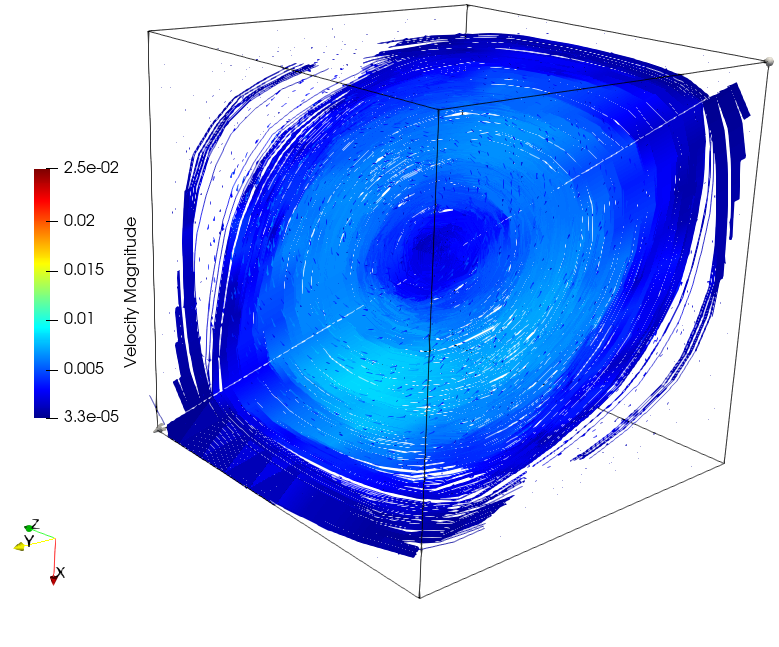} & 
\includegraphics[width=38mm]{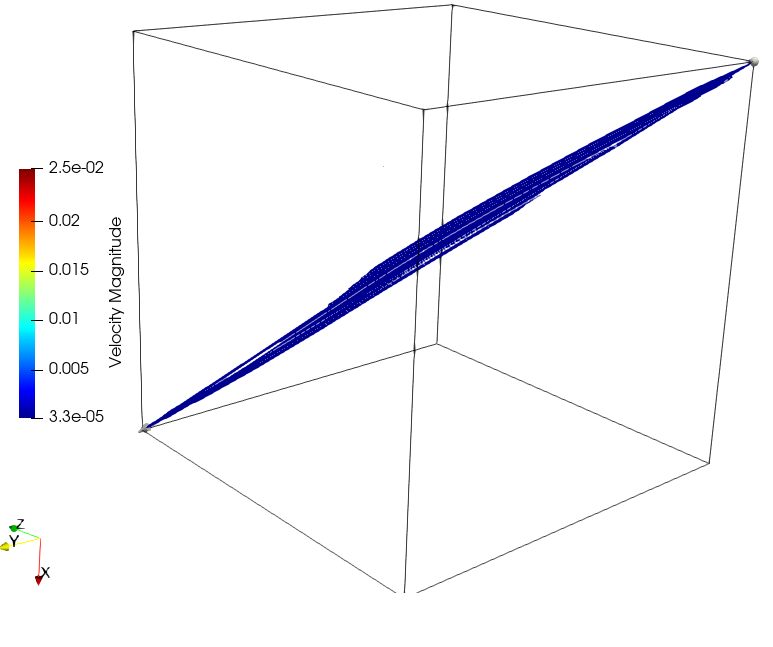}  \\[-3mm]
a)  $t=2.5\times10^{-2}$ & b) $t=15\times10^{-2}$ & 
c)  $t=15\times10^{-1}$\\ [1mm]
 \end{tabular}
\figcaption{Evolution in time of the velocity field ${\bf u}^m_h$, for $a_1=2$ and $a_2=0.3$.}\label{fig:VF1}
\end{center}

\

\begin{center}
\includegraphics[width=90mm]{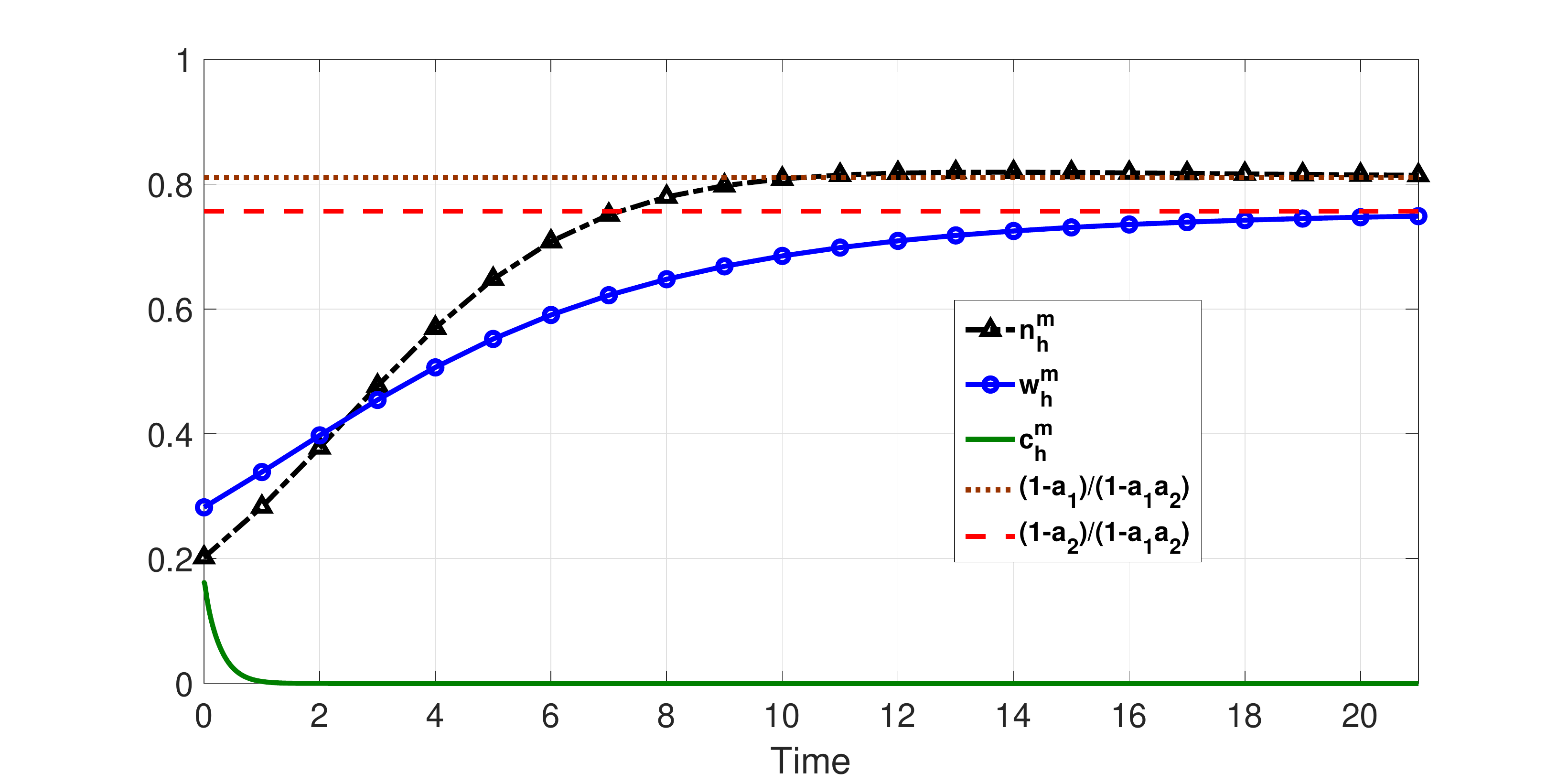}
\figcaption{Evolution in time of the total density of $n$ and $w$, and the total concentration of $c$, for $a_1=0.25$ and $a_2=0.3$.}\label{fig:DT}
\end{center}

\begin{center}
\includegraphics[width=90mm]{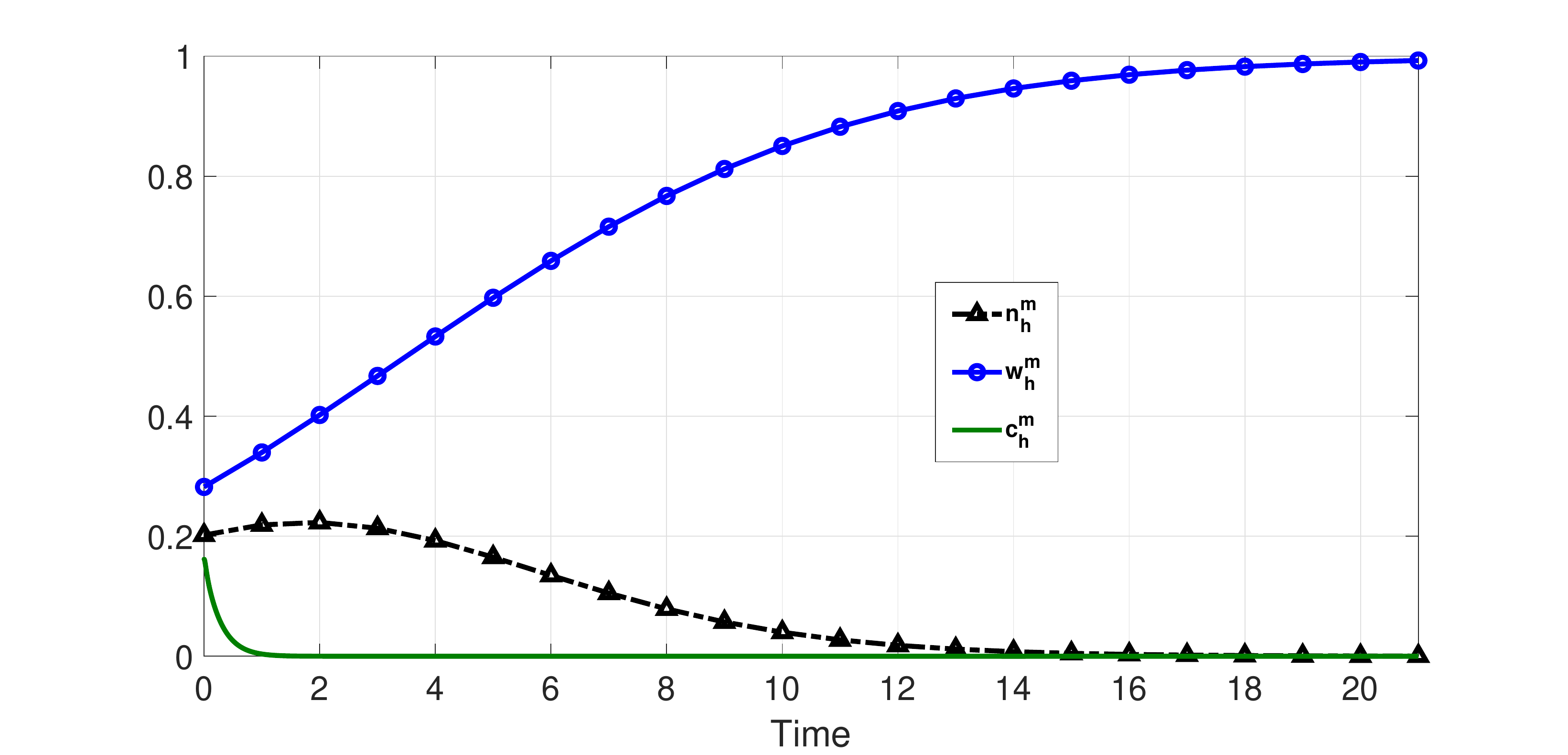}
\figcaption{Evolution in time of the total density of $n$ and $w$, and the total concentration of $c$, for $a_1=2$ and $a_2=0.3$.}\label{fig:DT1}
\end{center}

Finally, we consider $a_1=0.25$ and $a_2=2$. In this case, the species $n$ dominates to the species $w$, which tends to the  extinction (see Figure \ref{fig:DT2}).
\\
\begin{center}
\includegraphics[width=90mm]{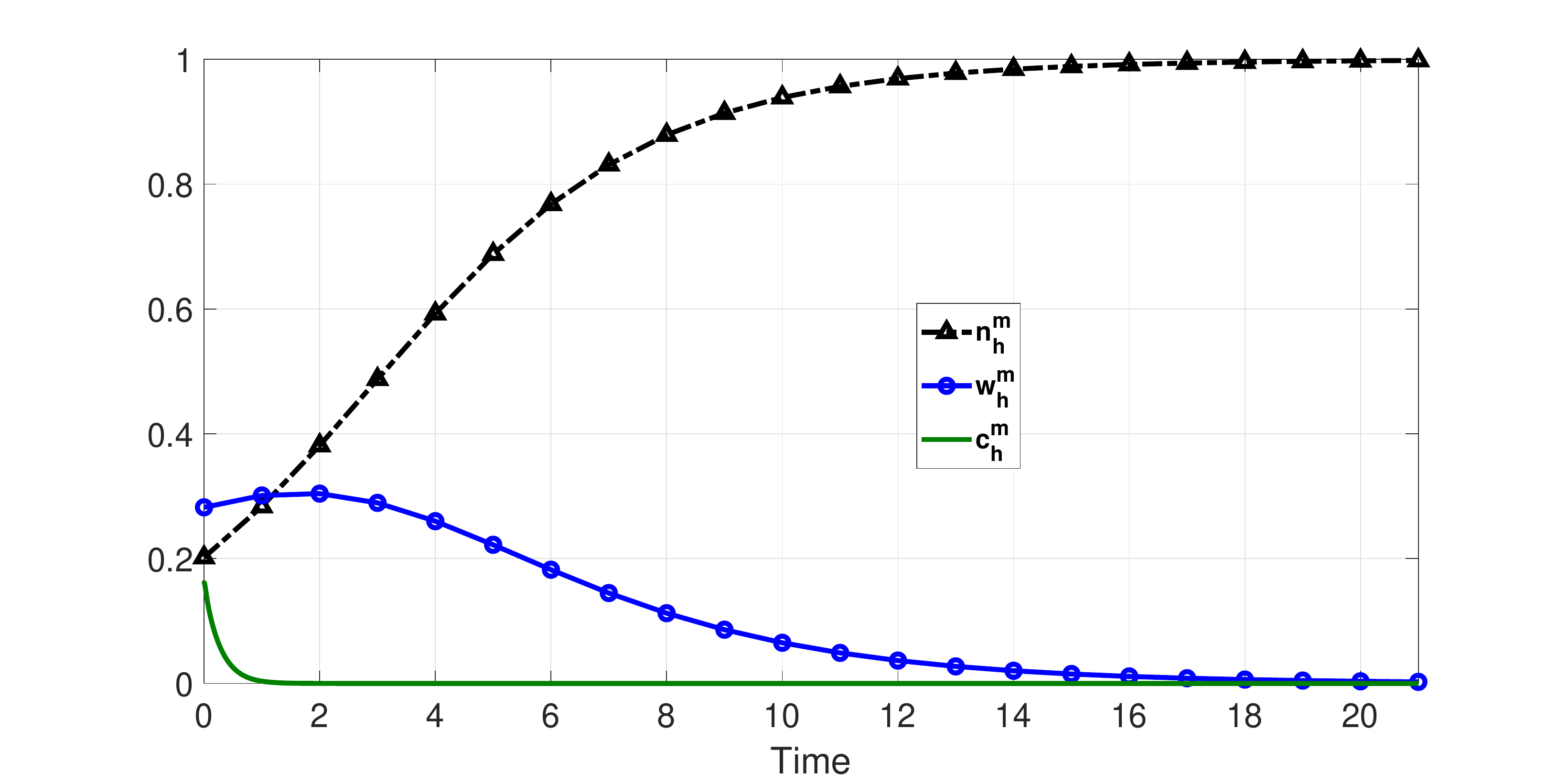}
\figcaption{Evolution in time of the total density of $n$ and $w$, and the total concentration of $c$, for $a_1=0.25$ and $a_2=2$.}\label{fig:DT2}
\end{center}
\bigskip

\underline{\it{Test 2}}:\ In this experiment we have considered the domain  $\Omega=[0,1] \times [0,1]$. Moreover,  we consider the exact solution
$$n=e^{-t}(cos(2\pi x) +cos(2\pi y) +3), \ \ w=e^{-t}(cos(2\pi y) - cos(2\pi x) +6),$$
$$ c=e^{-t}(sin(2\pi y)+cos(2\pi x) - 2\pi y +9), \ \ {\boldsymbol\sigma}=\nabla c= 2\pi e^{-t} (-sin(2\pi x), cos(2\pi y)-1),$$
$$ \mathbf{u}=e^{-t}(sin(2 \pi y)(-cos(2\pi x + \pi)-1),sin(2 \pi x)(cos(2\pi y + \pi)+1)),$$ 
$$ \pi=e^{-t}(sin(2 \pi y)+cos(2\pi x)),$$ and all parameters in (\ref{scheme1}) equal to 1. Note that $\mathbf{u}={\bf 0}$ and $\frac{\partial c}{\partial \boldsymbol{\nu}}=\frac{\partial n}{\partial \boldsymbol{\nu}}=\frac{\partial w}{\partial \boldsymbol{\nu}}=0$ on $\partial\Omega$, $\nabla \cdot \mathbf{u}=0$ in $\Omega$ and $\int_{\Omega} \pi =0$. Moreover, we use a uniform partition with $k+1$ nodes in each direction.

On the one hand, numerical results {of spatial convergence rates} are presented in Tables \ref{T1}-\ref{T6} for {$\Delta t=8.33\times 10^{-5}$} with respect to time $T=1$. We observe optimal convergence rates in space for the total errors  $e^m_{n},e^m_{w},e^m_{c}, e^m_{\mathbf{u}}$, that is,  second-order convergence in $l^{\infty}(L^2)$-norm, first-order convergence  in $l^{2}(H^1)$-norm and first-order convergence for $e^m_{\mathbf{u}}$ in $l^{\infty}(H^1)$-norm, which is in
agreement with our theoretical analysis.

\bigskip

\begin{minipage}{\textwidth}
	\begin{footnotesize} 
		\begin{center}
			\begin{tabular}{|| c | c | c | c | c||}
				\hline
				\hline
				$k \times k$ & $\Vert n(t_m) - n^m_h \Vert_{l^\infty(L^2)}$ & Order & $\Vert n(t_m) - n^m_h \Vert_{l^2(H^1)}$ & Order  \\
				\hline
				$10 \times 10$ & $5.677008 \times 10^{-2}$  & -   & $7.642456 \times 10^{-1}$  & - \\ 
				\hline
				$16 \times 16$ &  $2.227926\times 10^{-2}$ & 1.9901 & $4.715734 \times 10^{-1}$ & 1.0273 \\      
				\hline
				$22 \times 22$  &  $1.179489\times 10^{-2}$   & 1.9971 & $3.417503 \times 10^{-1}$ & 1.0111 \\       
				\hline
				$28 \times 28$  &  $7.277616 \times 10^{-3}$ & 2.0022 & $2.681382 \times 10^{-1}$ & 1.0059 \\
				\hline
				$34 \times 34$  & $4.929015 \times 10^{-3}$   & 2.0070  &  $2.206667 \times 10^{-1}$  & 1.0036 \\    
				\hline
				\hline
			\end{tabular}
		\tabcaption{Convergence rates in space for $n$.} 
		\label{T1} 
			\end{center}
	\end{footnotesize} 
\end{minipage}

\begin{minipage}{\textwidth}
	\begin{footnotesize} 
		\begin{center}
			\begin{tabular}{|| c | c | c | c | c||}
				\hline
				\hline
				$k \times k$ & $\Vert w(t_m) - w^m_h \Vert_{l^\infty(L^2)}$ & Order & $\Vert w(t_m) - w^m_h \Vert_{l^2(H^1)}$ & Order  \\
				\hline
				$10 \times 10$ & $6.639095 \times 10^{-2}$  & -   & $8.345002 \times 10^{-1}$  & - \\ 
				\hline
				$16 \times 16$ &  $2.607234\times 10^{-2}$ & 1.9887 & $4.898759 \times 10^{-1}$ & 1.1334 \\      
				\hline
				$22 \times 22$  &  $1.379666\times 10^{-2}$   & 1.9986 & $3.489831 \times 10^{-1}$ & 1.0649 \\       
				\hline
				$28 \times 28$  &  $8.509213 \times 10^{-3}$ & 2.0040 & $2.716992 \times 10^{-1}$ & 1.0380 \\
				\hline
				$34 \times 34$  & $5.761554 \times 10^{-3}$   & 2.0084  &  $2.226744 \times 10^{-1}$  & 1.0249 \\    
				\hline
				\hline
			\end{tabular}
		\tabcaption{Convergence rates in space for $w$.} 
		\label{T2} 
			\end{center}
	\end{footnotesize} 
\end{minipage}

\begin{minipage}{\textwidth}
	\begin{footnotesize} 
		\begin{center}
			\begin{tabular}{|| c | c | c | c | c||}
				\hline
				\hline
				$k \times k$ & $\Vert c(t_m) - c^m_h \Vert_{l^\infty(L^2)}$ & Order & $\Vert c(t_m) - c^m_h \Vert_{l^2(H^1)}$ & Order  \\
				\hline
				$10 \times 10$ & $3.573118 \times 10^{-2}$  & -   & $7.429560 \times 10^{-1}$  & - \\ 
				\hline
				$16 \times 16$ &  $1.403060 \times 10^{-2}$ & 1.9889 & $4.666742 \times 10^{-1}$ & 0.9894 \\      
				\hline
				$22 \times 22$  &  $7.432849 \times 10^{-3}$   & 1.9951 & $3.399509 \times 10^{-1}$ & 0.9949 \\       
				\hline
				$28 \times 28$  &  $4.591755 \times 10^{-3}$ & 1.9972 & $2.672989 \times 10^{-1}$ & 0.9970 \\
				\hline
				$34 \times 34$  & $3.115226 \times 10^{-3}$   & 1.9982  &  $2.202143 \times 10^{-1}$  & 0.9980 \\    
				\hline
				\hline
			\end{tabular}
		\tabcaption{Convergence rates in space for $c$.} 
		\label{T3} 
			\end{center}
	\end{footnotesize} 
\end{minipage}

\begin{minipage}{\textwidth}
	\begin{footnotesize} 
		\begin{center}
			\begin{tabular}{|| c | c | c | c | c||}
				\hline
				\hline
				$k \times k$ & $\Vert \mathbf{u}_1(t_m) - (\mathbf{u}_1)^m_h \Vert_{l^\infty(L^2)}$ & Order & $\Vert\mathbf{u}_1(t_m) - (\mathbf{u}_1)^m_h \Vert_{l^2(H^1)}$ & Order  \\
				\hline
				$10 \times 10$ & $5.462052 \times 10^{-2}$  & -   & $9.994213 \times 10^{-1}$  & - \\ 
				\hline
				$16 \times 16$ &  $2.147830 \times 10^{-2}$ & 1.9859 & $6.264182 \times 10^{-1}$ & 0.9939 \\      
				\hline
				$22 \times 22$  &  $1.135494 \times 10^{-2}$   & 2.0015 & $4.556049 \times 10^{-1}$ & 0.9998 \\       
				\hline
				$28 \times 28$  &  $6.996890 \times 10^{-3}$ & 2.0077 & $3.578633 \times 10^{-1}$ & 1.0013 \\
				\hline
				$34 \times 34$  & $4.733978 \times 10^{-3}$   & 2.0123  &  $2.946102 \times 10^{-1}$  & 1.0018 \\    
				\hline
				\hline
			\end{tabular}
		\tabcaption{Convergence rates in space for $\mathbf{u}_1$ in weak norms.} 
		\label{T4} 
			\end{center}
	\end{footnotesize} 
\end{minipage}

\begin{minipage}{\textwidth}
	\begin{footnotesize} 
		\begin{center}
			\begin{tabular}{|| c | c | c | c | c||}
				\hline
				\hline
				$k \times k$ & $\Vert \mathbf{u}_2(t_m) - (\mathbf{u}_2)^m_h \Vert_{l^\infty(L^2)}$ & Order & $\Vert\mathbf{u}_2(t_m) - (\mathbf{u}_2)^m_h \Vert_{l^2(H^1)}$ & Order  \\
				\hline
				$10 \times 10$ & $5.455275 \times 10^{-2}$  & -   & $9.994666 \times 10^{-1}$  & - \\ 
				\hline
				$16 \times 16$ &  $2.145280 \times 10^{-2}$ & 1.9858 & $6.264266 \times 10^{-1}$ & 0.9940 \\      
				\hline
				$22 \times 22$  &  $1.134125 \times 10^{-2}$   & 2.0016 & $4.556055 \times 10^{-1}$ & 0.9998 \\       
				\hline
				$28 \times 28$  &  $6.988373 \times 10^{-3}$ & 2.0078 & $3.578626 \times 10^{-1}$ & 1.0013 \\
				\hline
				$34 \times 34$  & $4.728202 \times 10^{-3}$   & 2.0123  &  $2.946094 \times 10^{-1}$  & 1.0018 \\    
				\hline
				\hline
			\end{tabular}
		\tabcaption{Convergence rates in space for $\mathbf{u}_2$ in weak norms.} 
		\label{T5} 
			\end{center}
	\end{footnotesize} 
\end{minipage}

\begin{minipage}{\textwidth}
	\begin{footnotesize} 
		\begin{center}
			\begin{tabular}{|| c | c | c | c | c||}
				\hline
				\hline
				$k \times k$ & $\Vert \mathbf{u}_1(t_m) - (\mathbf{u}_1)^m_h \Vert_{l^{\infty}(H^1)}$ & Order & $\Vert\mathbf{u}_2(t_m) - (\mathbf{u}_2)^m_h \Vert_{l^{\infty}(H^1)}$ & Order  \\
				\hline
				$10 \times 10$ & $2.335301 $  & -   & $2.335301$  & - \\ 
				\hline
				$16 \times 16$ &  $1.480473$ & 0.9697 & $1.480473$ & 0.9697 \\      
				\hline
				$22 \times 22$  &  $1.081356$   & 0.9865 & $1.081356$ & 0.9865\\       
				\hline
				$28 \times 28$  &  $8.512129 \times 10^{-1}$ &0.9923 & $8.512129 \times 10^{-1}$ & 0.9923 \\
				\hline
				$34 \times 34$  & $7.016738 \times 10^{-1}$   & 0.9950  &  $7.016738 \times 10^{-1}$  & 0.9950 \\    
				\hline
				\hline
			\end{tabular}
		\tabcaption{Convergence rates in space for $\mathbf{u}_1$ and $\mathbf{u}_2$ in strong norms.} 
		\label{T6} 
			\end{center}
	\end{footnotesize} 
\end{minipage}

{
On the other hand, some numerical results of convergence rates in time are listed in Tables \ref{T7}-\ref{T9} for $h=1/160$ (that is, $k=160$ nodes in space in each direction), with respect to time $T=5$. We observe approximately the first-order convergence in time in weak and strong norms for all variables, which is in
	agreement with our theoretical analysis.\\

\begin{minipage}{\textwidth}
	\begin{footnotesize} 
		\begin{center}
			\begin{tabular}{|| c | c | c | c | c||}
				\hline
				\hline
				$\Delta t$ & $\Vert w(t_m) - w^m_h \Vert_{l^{\infty}(L^2)}$ & Order & $\Vert n(t_m) - n^m_h \Vert_{l^2(H^1)}$ & Order  \\
				\hline
				$1.04 \times 10^{-1}$ & $1.9737\times 10^{-1} $  & -   & $3.3887\times 10^{-1} $  & - \\ 
				\hline
				$8.92 \times 10^{-2}$ &  $1.6237 \times 10^{-1}$ & 1.2660 & $2.7748 \times 10^{-1}$ & 1.2965 \\      
				\hline
				$7.81 \times 10^{-2}$  &  $1.3762 \times 10^{-1}$   & 1.2385 & $2.3649 \times 10^{-1}$ & 1.1972\\       
				\hline
				$6.94 \times 10^{-2}$  &  $1.1925\times 10^{-1}$ & 1.2168 & $2.0741\times 10^{-1}$ & 1.1137 \\
			\hline
			$6.25 \times 10^{-2}$  & $1.0509 \times 10^{-1}$   & 1.1991 &  $1.8573 \times 10^{-1}$  & 1.0479  \\    
				\hline
				\hline
			\end{tabular}
			\tabcaption{Convergence rates in time for $w$ and $n$.} 
			\label{T7} 
		\end{center}
	\end{footnotesize} 
\end{minipage}	

	\begin{minipage}{\textwidth}
	\begin{footnotesize} 
		\begin{center}
			\begin{tabular}{|| c | c | c | c | c||}
				\hline
				\hline
				$\Delta t$ & $\Vert \mathbf{u}_1(t_m) - (\mathbf{u}_1)^m_h \Vert_{l^{\infty}(L^2)}$ & Order & $\Vert \mathbf{u}_2(t_m) - (\mathbf{u}_2)^m_h \Vert_{l^\infty(L^2)}$ & Order  \\
				\hline
				$1.04 \times 10^{-1}$ & $3.6109\times 10^{-2} $  & -   & $3.6056\times 10^{-2}$  & - \\ 
				\hline
				$8.92 \times 10^{-2}$ &  $3.1504 \times 10^{-2}$ & 0.8850 & $3.1441 \times 10^{-2}$ & 0.8886 \\      
				\hline
				$7.81 \times 10^{-2}$  &  $2.7871 \times 10^{-2}$   & 0.9175 & $2.7801 \times 10^{-2}$ & 0.9214\\       
				\hline
				$6.94 \times 10^{-2}$  &  $2.4934 \times 10^{-2}$ & 0.9455& $2.4859 \times 10^{-2}$ & 0.9495 \\
			\hline
			$6.25 \times 10^{-2}$  & $2.2511 \times 10^{-2}$   & 0.9703 &  $2.2434 \times 10^{-2}$  & 0.9743 \\    
				\hline
				\hline
			\end{tabular}
			\tabcaption{Convergence rates in time for $\mathbf{u}_1$ and $\mathbf{u}_2$.} 
			\label{T8} 
		\end{center}
	\end{footnotesize} 
\end{minipage}

\begin{minipage}{\textwidth}
	\begin{footnotesize} 
		\begin{center}
			\begin{tabular}{|| c | c | c ||}
				\hline
				\hline
				$\Delta t$ & $\Vert c(t_m) - c^m_h \Vert_{l^{2}(H^1)}$ & Order   \\
				\hline
				$1.04 \times 10^{-1}$ & $6.6504\times 10^{-1} $  & -    \\ 
				\hline
				$8.92 \times 10^{-2}$ &  $5.7032 \times 10^{-1}$ & 0.9968  \\      
				\hline
				$7.81 \times 10^{-2}$  &  $4.9963 \times 10^{-1}$   & 0.9910 \\       
				\hline
				$6.94 \times 10^{-2}$  &  $4.4483 \times 10^{-1}$ & 0.9863  \\
			\hline
			$6.25 \times 10^{-2}$  & $4.0110 \times 10^{-1}$   & 0.9821 \\    
				\hline
				\hline
			\end{tabular}
			\tabcaption{Convergence rates in time for $c$.} 
			\label{T9} 
		\end{center}
	\end{footnotesize} 
\end{minipage}
}

{
\section*{Acknowledgements}
The authors have been supported by the Vicerrector\'ia de Investigaci\'on y Extensi\'on of Universidad Industrial de Santander.
}

\end{document}